\documentclass[11p,leqno]{amsart}
\usepackage{amsmath}
\usepackage{amsfonts}
\usepackage{amssymb}
\usepackage{graphicx}
\usepackage{color}
\newtheorem{theorem}{Theorem}[section]
\newtheorem*{theorem*}{Theorem}
\newtheorem{lemma}{Lemma}[section]
\newtheorem{corollary}[theorem]{Corollary}
\newtheorem{proposition}{Proposition}[section]

\newtheorem{remark}[theorem]{Remark}

\def\heat{\lf(\frac{\p}{\p t}-\Delta\ri)}
\def \b {\beta}

\def\Ric{\text{Ric}}
\def\lf{\left}
\def\ri{\right}

\def\a{\alpha}

\def\g{\gamma}

\def\p{\partial}
\def\delbar{{\bar{\delta}}}

\def\dbar{{\bar{\partial}}}

\def\C{\Bbb C}
\def\R{\Bbb R}

\def\dbar{\bar\partial}
\def\ba{{\bar{\alpha}}}
\def\bb{{\bar{\beta}}}
\def\abb{\alpha{\bar{\beta}}}
\def\gbd{\gamma{\bar{\delta}}}

\def\HQ{\widetilde{\mathcal{Q}}}

\def\HJ{\widetilde{\mathcal{J}}}

\def \D {\Delta}

\def\id{\operatorname{id}}
\def\Ric{\operatorname{Ric}}
\def\Rm{\operatorname{Rm}}
\def\I{\operatorname{I}}

\def\tr{\operatorname{tr}}

\def\ov{\overline}

\newcommand{\SO}{{\mathsf{SO}}}

\def\I{\operatorname{I}}

\numberwithin{equation}{section}

\begin{document}
\title[positive $(p,p)$-forms]{\bf Sharp differential estimates of Li-Yau-Hamilton type
 for positive $(p,p)$-forms on K\"ahler manifolds}

\author{ Lei Ni and Yanyan Niu}



\date{}

\maketitle

\begin{abstract} In this paper we study the heat equation (of Hodge-Laplacian) deformation of $(p, p)$-forms on a K\"ahler manifold. After identifying the condition and establishing that the positivity of a $(p, p)$-form solution is preserved under such an invariant condition we prove the sharp differential Harnack (in the sense of Li-Yau-Hamilton) estimates for the positive solutions of the Hodge-Laplacian heat equation. We also prove a nonlinear version coupled with the K\"ahler-Ricci flow and some interpolating matrix differential Harnack type estimates for both the K\"ahler-Ricci flow and the Ricci flow.
\end{abstract}

\section{Introduction}

In this paper, we study the deformation of positive $(p,p)$-forms on a K\"ahler manifold via the $\bar{\partial}$-Laplacian heat equation. One of our main goals of this paper is to prove differential Harnack estimates for positive solutions. The Harnack estimate for positive solutions of linear parabolic PDEs of divergence form goes back to the fundamental work of Moser \cite{Moser}. In another fundamental work \cite{LY}, Li and Yau, on Riemannian manifolds with nonnegative Ricci curvature,  proved a sharp differential estimate which implies a sharp form of Harnack inequality for the positive solutions. Later, Hamilton \cite{richard-harnack} proved the miraculous matrix differential estimates for the curvature operator of solutions to Ricci flow assuming that the curvature operator is nonnegative. Since the curvature operator of a Ricci flow solution satisfies a nonlinear diffusion-reaction equation, this result of Hamilton  is as surprising as it is important. Due to this development people  also call this type of sharp estimates Li-Yau-Hamilton (abbreviated as LYH) type estimates. There are many further works \cite{Andrews, brendle, Cao, CN, chow-Gauss, chow-yamabe, Hlinear, Hmcf, N-jams, Ni-JDG07, NT-ajm} in this direction since the foundational estimate of Li and Yau for linear heat equation and Hamilton's one for the Ricci flow, which cover various different geometric evolution equations, including the mean curvature flow, the Gauss curvature flow, the K\"ahler-Ricci flow, the Hermitian-Einstein flow, the Yamabe flow etc.

Since the Harnack estimate for the linear equation implies the regularity of the weak solution, it has been an interesting question that if  the celebrated De Giorgi-Nash-Moser theory for the linear equation  has its analogue for linear systems. This unfortunately  has been known to be  false in the most general setting. As a geometric interesting system, the Hodge-Laplacian operator on forms has been extensively  studied since the original works of Hodge and Kodaira (see for example Morrey's classics \cite{Morrey} and references therein). It is a natural candidate on which one would like to investigate whether or not the differential Harnack estimates of LYH type still hold. One of the main results of this paper is to prove such  LYH type estimates for this system. The positivity  (really  meaning non-negativity) of the $(p, p)$-form is in the sense of Lelong \cite{L}. In fact, in \cite{N-jams} the first author proved a LYH type estimate for  positive semi-definite Hermitian symmetric tensors satisfying the so-called Lichnerowicz-Laplacian heat equation. This in particular applies to solutions of $(1,1)$-forms to the Hodge-Laplacian heat equation. The first main result of this paper is to generalize this result for $(1, 1)$-forms to solutions of $(p, p)$-forms to the Hodge-Laplacian heat equation.  The result is proved under a new curvature condition $\mathcal{C}_p$. We say that the curvature operator $\Rm$ of a K\"ahler manifold $(M, g)$ satisfies $\mathcal{C}_p$ (or lies inside the cone $\mathcal{C}_p$) if
$
\langle \Rm(\alpha), \bar{\alpha}\rangle \ge 0
$
for any  $\alpha\in \wedge ^{1, 1}(\C^m)$ such that  $\alpha =\sum_{k=1}^p X_k \wedge  \ov{Y}_k$ with $X_k, Y_k \in T'M$. Here $TM\otimes \C =T'M\oplus T''M$, $\langle \cdot, \cdot \rangle$ is the bilinear extension of the Riemannian product,  and we identify $T'M$ with $\C^m$. Under the condition $\mathcal{C}_p$ we prove the following result.

\begin{theorem} \label{thm11} For $\phi(x, t)$, a positive $(p, p)$-form satisfying $\left(\frac{\partial}{\partial t} +\Delta_{\dbar}\right)\phi(x, t) =0$, then
$$\frac{1}{\sqrt{-1}}\bar{\partial}^*\partial^* \phi+\frac{1}{\sqrt{-1}}\iota_{V}\cdot  \bar{\partial}^* \phi-\frac{1}{\sqrt{-1}}\iota_{\ov{V}}\cdot \partial^* \phi  + \sqrt{-1} \iota_V\cdot \iota_{\ov{V}}\cdot \phi+\frac{\Lambda \phi}{t}\ge 0
$$
as a $(p-1, p-1)$-form, for any $(1, 0)$ type vector field $V$. Here $\Lambda$ is the adjoint of the operator $L\doteqdot \omega \wedge (\cdot)$ with $\omega$ being the K\"ahler form.
\end{theorem}

The above estimate is compatible with the  Hodge $*$ operator (see Section 4 for the detailed discussions). Also note the easy fact that $*$-operator maps a positive $(p, p)$-form to a positive $(m-p, m-p)$-form. It then implies that if $\psi=*\phi$,
$$
\sqrt{-1}\partial \bar{\partial} \psi +\sqrt{-1}V^*\wedge \bar{\partial}\psi-\sqrt{-1}\ov{V}^* \wedge \partial \psi +\sqrt{-1}V^*\wedge \ov{V}^* \wedge \psi+\frac{L (\psi)}{t}\ge 0
$$
as a $(m-p+1, m-p+1)$-form, with $V^*$ being a $(1,0)$-type $1$-form. This  generalizes the matrix estimate for positive solutions to the heat equation proved in \cite{CN}, which asserts $\sqrt{-1}\partial \bar{\partial}\log \psi+\frac{\omega}{t}\ge 0$. (Note that this is the matrix version of Li-Yau's estimate: $\Delta \log \psi+\frac{m}{t}\ge 0$. See also \cite{yau-harmonic} for the earlier work for  harmonic functions.)
  For the proof, it is a combination of techniques of \cite{Cao}, \cite{N-jams}, \cite{NT-jdg} and Hamilton's argument in  \cite{richard-harnack}. Applying to the static solution the above result asserts  a differential estimate:
\begin{equation}\label{eq:har1}
\frac{1}{\sqrt{-1}}\bar{\partial}^*\partial^* \phi+\frac{1}{\sqrt{-1}}\iota_{V}\cdot  \bar{\partial}^* \phi-\frac{1}{\sqrt{-1}}\iota_{\ov{V}}\cdot \partial^* \phi  +\sqrt{-1} \iota_V\cdot \iota_{\ov{V}}\cdot \phi\ge 0
\end{equation}
for  any  $\bar{\partial}$-harmonic positive $(p, p)$-form $\phi$ and any vector field $V$ of $(1, 0)$-type. Note here on a noncompact manifold, being harmonic does not imply that $\partial^*\phi =0$ (or $\bar{\partial}^* \phi=0$).

As a result of independent interest we also observe that $\mathcal{C}_p$ is an invariant condition under the K\"ahler-Ricci flow, thanks to a general invariant cone result of Wilking, whose proof we include here in the Appendix (see also a recent preprint \cite{CT}). Note that this result of Wilking includes almost all  the known invariant cones such as the nonnegativity of bisectional curvature, the nonnegativity of isotropic curvature, etc.

After establishing  the invariance of $\mathcal{C}_p$, it is natural to study the heat equation for the Hodge-Laplacian coupled with the K\"ahler-Ricci flow. For this we proved the following nonlinear version of the above estimate.

\begin{theorem} \label{thm12}Assume that $\phi(x, t)\ge 0$ is a solution to heat equation of the Hodge-Laplacian coupled with the K\"ahler-Ricci flow: $\frac{\partial}{\partial t}g_{i\bar{j}}=-R_{i\bar{j}}$. Then
$$
\frac{1}{\sqrt{-1}}\bar{\partial}^*\partial^* \phi+\frac{1}{\sqrt{-1}}\iota_{V}\cdot  \bar{\partial}^* \phi-\frac{1}{\sqrt{-1}}\iota_{\ov{V}}\cdot \partial^* \phi  +\sqrt{-1}\iota_V\cdot \iota_{\ov{V}}\cdot \phi+\Lambda_{\Ric}(\phi)+\frac{\Lambda \phi}{t}\ge 0
$$as a $(p-1, p-1)$-form for any $(1, 0)$ type vector field $V$. Here $\Lambda_{\Ric}\phi $ is the adjoint of $\Ric\wedge (\cdot)$ with $\Ric=\sqrt{-1}R_{i\bar{j}}dz^i\wedge dz^{\bar{j}}$ being the Ricci form.
\end{theorem}

To prove the above result it is necessary to prove the following family of matrix differential estimates which interpolate between Hamilton's matrix estimate and Cao's estimate for the K\"ahler-Ricci flow.

\begin{theorem} \label{thm13} Let $(M, g(t))$ be a complete solution to the K\"ahler-Ricci flow satisfying the condition $\mathcal{C}_p$ on $M \times [0, T]$. When $M$ is noncompact we assume that the curvature of $(M, g(t))$ is bounded on $M \times [0, T]$. Then
 for any $\wedge^{1, 1}$-vector $U$ which can be written as
$U=\sum_{i=1}^{p-1} X_i\wedge \bar{Y}_i+W\wedge \bar{V}$, for $(1, 0)$-type vectors $X_i, Y_i, W, V$, the Hermitian bilinear form
$$
\mathcal{Q}(U\oplus W)\doteqdot
\mathcal{M}_{\abb}W^{\a}W^{\bb}+P_{\abb\g}\bar{U}^{\bb\g}W^{\a}
+P_{\abb\bar{\g}}U^{\a\bar{\g}}W^{\bb}
+R_{\abb\gbd}U^{\a\bb}\bar{U}^{\delbar\g}
$$
satisfies that $\mathcal{Q}\ge 0$ for any $t>0$. Moreover, if the equality ever occurs for some $t>0$, the universal cover of $(M, g(t))$ must be a gradient expanding K\"ahler-Ricci soliton.
\end{theorem}

Recall that the tensors $\mathcal{M}$ and $P$ are defined as $\mathcal{M}_{\abb}=\D R_{\abb}+R_{\abb\gbd}R_{\delta\bar{\g}}+\frac{R_{\abb}}{t}, \quad
P_{\abb\g}=\nabla_{\g}R_{\abb}, \quad P_{\abb\bar{\g}}=\nabla_{\bar{\g}}R_{\abb}.$
There exists a similar condition $\widetilde{\mathcal{C}}_p$ for Riemannian manifold which can be formulated similarly. Precisely we call the curvature operator satisfies  $\widetilde{\mathcal{C}}_p$ if
$
\langle \Rm (v), \bar{v} \rangle >0
$
for any nonzero $v\in \Lambda^2(\C^n)$ which can be written as $v=\sum_{i=1}^k{Z_i\wedge W_i}$
for some complex vectors $Z_i$ and $W_i \in TM \otimes \C$. For K\"ahler manifolds it can be shown that  $\widetilde{\mathcal{C}}_p=\mathcal{C}_{2p}$ and $\mathcal{C}_{2}$ amounts to the nonnegativity of the complex sectional curvature, a notion goes back at least to the work of Sampson \cite{Sa} on harmonic maps. This leads us to discover another family of matrix differential estimates for the Ricci flow which interpolate the result of Hamilton and a recent result of Brendle.

\begin{theorem} \label{thm14} Assume that $(M, g(t))$ on $M \times [0, T]$ satisfies $\widetilde{\mathcal{C}}_p$. When $M$ is noncompact we also assume that the curvature of $(M, g(t))$ is uniformly bounded on $M \times [0, T]$. Then for any $t>0$, the quadratic form
$$
\widetilde{\mathcal{Q}}(W\oplus U)\doteqdot \langle \mathcal{M}(W), W\rangle +2\langle P(W), U\rangle +\langle \Rm(U), U\rangle
$$
satisfies that $\HQ \ge 0$ for any $(x, t)\in M\times [0,  T]$, $W\in T_x M \otimes \C$ and $U\in \wedge^2(T_xM\otimes \C) $ such that $U=\sum_{\mu=1}^p W_\mu \wedge Z_\mu$ with  $W_p=W$. Furthermore, the equality holds for some $t>0$ implies that the universal cover of $(M, g(t))$ is a gradient expanding Ricci soliton.
\end{theorem}

Here $\mathcal{M}$ and $P$ are defined similarly. In fact for $p=1$, our result is slightly stronger than Brendle's estimate. After we finished our paper, we were brought the attention  to a recent preprint \cite{CT}, where a similar, but seemly more general result, was formulated in terms of the space-time consideration of  Chow and Chu \cite{CC}. In the Spring of 2009 Wilking informed us that he has obtained a differential Harnack estimate for the Ricci flow with positive isotropic curvature, whose precise statement however is not known to us. It is very possible that the above result is a special case of his. Nevertheless  our statement and proof here are direct/explicit without involving  the space-time formulation. The proof is also rather short (see Section 9), can be easily checked and  is motivated by the K\"ahler case.

Here is how we organize the paper. In Section 2 we prove that under the condition $\mathcal{C}_p$ the positivity of the $(p,p)$-forms is preserved under the Hodge-Laplacian heat equation. In Section 3 we derive the invariance of $\mathcal{C}_p$ by refining an argument of  Wilking which is detailed in the Appendix. In  Section 4 we collect and prove some  preliminary formulae needed for the proof of the Theorem \ref{thm11}. The rigidity on the equality case as well as a monotonicity formula implied by Theorem \ref{thm11} was also included in Section 4.
Section 5 is devoted to the proof of Theorem \ref{thm11}. Sections 6 and  8 are devoted to the proof of Theorem \ref{thm12} as Section 7 is on the proof  Theorem \ref{thm13}, which is needed in Section 8. Section 9 is on the proof of Theorem \ref{thm14}. Since up to Section 9 we present only the argument for the compact manifolds, Section 10 is  the noncompact version of Sections 3, 7, 10, where the metric  is  assumed to have bounded curvature, while Section 11 supplies the argument for the noncompact version of Section 2, 6, where no upper bound on the curvature is assumed. Due to the length of the paper we shall study the applications of the estimates in a forth coming article.

\section{Heat equation deformation of $(p, p)$-forms}

Let $(M^m, g)$ be a complex Hermitian manifold of complex dimension $m$. Recall that a $(p,p)$-form $\phi$ is called positive if for any $x\in M$ and for any vectors $v_1, v_2, \cdot\cdot\cdot , v_p\in T^{1,0}_xM$, $\langle \phi, \frac{1}{\sqrt{-1}}v_1\wedge \overline {v}_1\wedge \cdot\cdot\cdot \wedge \frac{1}{\sqrt{-1}}v_p\wedge \overline{v}_p\rangle \ge 0$. By linear algebra (see also \cite{Siu-74}) it is equivalent to the condition that the nonnegativity  holds  for $v_1, \cdots, v_p$ satisfying that $\langle v_i, v_j\rangle=\delta_{ij}$. We also denote $\langle \phi, \frac{1}{\sqrt{-1}}v_1\wedge \overline {v}_1\wedge \cdot\cdot\cdot \wedge \frac{1}{\sqrt{-1}}v_p\wedge \overline{v}_p\rangle  $ by $\phi(v_1, v_2, \cdots, v_p; \bar{v}_1, \bar{v}_2, \cdots, \bar{v}_p)$, or even $\phi_{v_1v_2\cdots v_p, \bar{v}_1 \bar{v}_2\cdots \bar{v}_p}$. We say that $\phi$ is strictly positive if $\phi_{v_1v_2\cdots v_p, \bar{v}_1 \bar{v}_2\cdots \bar{v}_p}$ is positive for any linearly  independent $\{v_i\}_{i=1}^p$.
Let $\Delta_{\dbar}=\dbar \dbar^* +\dbar^* \dbar$ be the $\dbar$-Hodge Laplacian operator. There also exists a Laplacian operator $\Delta$ defined by
$$
\Delta =\frac{1}{2}\left(\nabla_i \nabla_{\bar{i}}+\nabla_{\bar{i}}\nabla_i\right).
$$
where $\nabla$ is the induced co-variant derivative on $(p, p)$-forms.
 Since the complex geometry, analysis and Riemannian geometry fit better when the manifold is K\"ahler, we assume that $(M, g)$ is a K\"ahler manifold for our discussion. Let $\omega=\sqrt{-1} g_{i\bar{j}}dz^i\wedge dz^{\bar{j}}$ be the K\"ahler form. Clearly $\omega^p$ is a strictly positive $(p, p)$-form.

For a $(p, p)$-form $\phi_0$, consider the evolution  equation:
\begin{equation}\label{eq:11}
\left(\frac{\partial}{\partial t} +\Delta_{\dbar}\right)\phi(x, t) =0
\end{equation}
with initial value $\phi(x, 0)=\phi_0(x)$.
Our first concern is when the positivity of the $(p, p)$-forms is preserved under the above evolution equation. If we denote by $\mathcal{P}_p$  the closed cone consisting all positive $(p, p)$-forms, an equivalent  question is whether or not $\mathcal{P}_p$ is preserved under the heat equation (\ref{eq:11}).
The answer is well known for the cases $p=0$ and $p=m$ since the equation is nothing but the regular heat equation.
When $p=1$, this question  was studied in \cite{NT-ajm} as well as \cite{NT-jdg} and it was proved that when $(M, g)$ is a complete K\"ahler manifold with nonnegative bisectional curvature, then the positivity is preserved for the solutions satisfying certain reasonable growth conditions, which is needed for  the uniqueness of the solution with the given initial data.

It turns out,  to prove the invariance of $\mathcal{P}_p$ for $m-1\ge p\ge 2$,  we need to introduce a new curvature condition which we shall formulate below. We say that the curvature operator $\Rm$ of a K\"ahler manifold $(M, g)$ satisfies $\mathcal{C}_p$ (or lies inside the cone $\mathcal{C}_p$) if
$$
\langle \Rm(\alpha), \bar{\alpha}\rangle \ge 0
$$
for any  $\alpha\in \wedge ^{1, 1}(\C^m)$ (we use $\wedge^{1,1}_{\R} (\C^m)$ to denote the space of real wedge-$2$ vectors of  $(1, 1)$ type),   such that it can be  written as $\alpha =\sum_{k=1}^p X_k \wedge  \bar{Y}_k$. Here $TM\otimes \C =T'M\oplus T''M$, $\langle \cdot, \cdot \rangle$ is the bilinear extension of the Riemannian product,  and we identify $T'M$ with $\C^m$. Note that $\langle \Rm(X\wedge \bar{Y}), \overline{X\wedge \bar{Y}}\rangle =R_{X\bar{X} Y\bar{Y}}$, the bisectional curvature of the complex plane spanned by $\{ X, Y\}$. Here the cones $\mathcal{C}_p$ interpolate between the cone of nonnegative bisectional curvature and that of the nonnegative curvature operator. In the next section we shall show that in fact $\mathcal{C}_p$ is an invariant condition under the K\"ahler-Ricci flow, which generalizes an earlier result of Bando-Mok \cite{Bando, Mok} on the invariance of the nonnegative bisectional curvature cone. Let us first recall the following well-known computational lemma of Kodaira \cite{M-K}.

\begin{lemma} \label{lemma11}Let $\phi$ be a $(p, p)$-form,  which can be locally expressed as
$$
\phi=\frac{1}{(p!)^2}\sum \phi_{I_p,\bar{J}_p}\left(\sqrt{-1}dz^{i_1}\wedge dz^{\bar{j_1}}\right)\wedge \cdot\cdot\cdot\wedge \left(\sqrt{-1}dz^{i_p}\wedge dz^{\bar{j_p}}\right)
$$
where $I_p=(i_1, \cdot\cdot \cdot i_p)$ and $\bar{J}_p=(\bar{j_1}, \cdot\cdot\cdot, \bar{j_p})$. Then
\begin{eqnarray}
\left(\Delta_{\dbar} \phi\right)_{I_p, \bar{J}_p}&=&-\frac{1}{2}\left(\sum_{ij}g^{i\bar{j}}\nabla_{\bar{j}}\nabla_{i}
\phi_{I_p,\bar{J}_p}+\sum_{ij}g^{\bar{j}i}\nabla_i\nabla_{\bar{j}}\phi_{I_p,\bar{J}_p}\right)
\nonumber\\
&\quad& -\sum_{\mu = 1}^{p}\sum_{\nu =1 }^{p}R^{\,
k\bar{l}}_{i_\mu\,\,\bar{j}_\nu}\phi_{i_1\cdots(k)_\mu\cdots
 i_p,\bar{j}_1\cdots(\bar{l})_\nu\cdots\bar{j_p}} \label{eq:Kodaira}\\
 &\quad&+\frac{1}{2}\left(\sum_{\nu =1}^{p}R^{\bar{l}}_{\,\,
\bar{j}_\nu}\phi_{I_p,\bar{j_1}\cdots(\bar{l})\cdots\bar{j}_p}+\sum_{\mu =1 }^{p}R_{i_\mu}^{\,\,
k}\phi_{i_1\cdots(k)_\mu\cdots i_p, \bar{J}_p}  \right)\nonumber.
\end{eqnarray}
Here $R_{i\bar{j}k\bar{l}}$, $R_{i\bar{j}}$, $R^{\,\,
k\bar{l}}_{i_\mu\,\,\bar{j_\nu}}$, $R^{\bar{l}}_{\,\,
\bar{j_\nu}}$ are the curvature tensor, Ricci tensor and the index raising  of them via the K\"ahler metric, $(k)_{\mu}$ means that the index in the $\mu$-th position is replaced by $k$. Here the repeated index is summed from $1$ to $m$.
\end{lemma}

An immediate consequence of (\ref{eq:Kodaira}) is that if $\phi$ is a solution of (\ref{eq:11}), then it  satisfies that
\begin{equation}\label{KB-heat}
\left(\frac{\partial}{\partial t} -\Delta\right)\phi(x, t) =\mathcal{KB}(\phi)
\end{equation}
where
\begin{eqnarray*}
\left(\mathcal{KB}(\phi)\right)_{I_p,\bar{J_p}}&=&\sum_{\mu = 1}^{p}\sum_{\nu =1 }^{p}R^{\,
k\bar{l}}_{i_\mu\,\bar{j_\nu}}\phi_{i_1\cdots(k)_\mu\cdots
 i_p,\bar{j_1}\cdots(\bar{l})_\nu\cdots\bar{j}_p} \\
 &\quad&
-\frac{1}{2}\left(\sum_{\nu =1}^{p}R^{\bar{l}}_{\,\,
\bar{j}_\nu}\phi_{I_p,\bar{j_1}\cdots(\bar{l})\cdots\bar{j_p}}+\sum_{\mu =1 }^{p}R_{i_\mu}^{\,\,
k}\phi_{i_1\cdots(k)_\mu\cdots i_p, \bar{J}_p}\right).
\end{eqnarray*}

\begin{proposition}\label{max-pp} Let $(M, g)$ be a K\"ahler manifold whose curvature operator $\Rm \in \mathcal{C}_p$. Assume that $\phi(x, t)$ is a solution of (\ref{eq:11}) such that $\phi(x, 0)$ is positive. Then $\phi(x, t)$ is positive for $t>0$.
\end{proposition}
\begin{proof} When $M$ is a compact manifold, applying Hamilton's tensor maximum principle, it suffices to show that if at $(x_0, t_0)$ there exist $v_1, \cdots v_p$ such that $\phi_{v_1 \cdots v_p, \bar{v}_1\cdots \bar{v}_p}=0$ and for any $(x, t)$ with $t\le t_0$, $\phi \ge 0$,
$$
\mathcal{KB}(\phi)_{v_1 \cdots v_p, \bar{v}_1\cdots \bar{v}_p}\ge 0.
$$
This holds obviously if $\{v_i\}_{i=1}^p$ is linearly dependent since $\mathcal{KB}(\phi)$ is a $(p, p)$-form. Hence we assume that $\{v_i\}_{i=1}^p$ is linearly independent. By Gramm-Schmidt process, which does not change the sign (or being zero) of $\phi_{v_1 \cdots v_p, \bar{v}_1\cdots \bar{v}_p}$, we can assume that $v_1, \cdots, v_p$ can be extended to a unitary frame. Hence we may assume that $(v_1, \cdots, v_p)$=$(\frac{\partial\,\,}{\partial z_1}, \cdots, \frac{\partial\,\,}{\partial z_p})$ with $(z^1,\cdots, z^m)$ being a normal coordinate centered at $x_0$. Hence what we need to verify is that $(\mathcal{KB}(\phi))_{1\,\cdots\, p,\, \bar{1}\, \cdots\, \bar{p}}\ge 0$.
Since we have that
$\phi_{1\, 2\, \cdots\, p,\,  \bar{1}\, \bar{2}\, \cdots\, \bar{p}}=0$ and
$$
I(t)\doteqdot \phi\left(\frac{1}{\sqrt{-1}}(v_1+tw_1)\wedge \overline{v_1+tw_1}\wedge \cdots \wedge \frac{1}{\sqrt{-1}}(v_p+tw_p)\wedge \overline{v_p+tw_p}\right)\ge 0
$$
for any $t\ge 0$ and any vectors $w_1, \cdots,  w_p$. The equation $I'(0)=0$ implies that
\begin{equation}\label{1st-var}
\sum_{1\le k, l\le p} \phi_{1\,\cdots\, (w_k)_k\, \cdots\, p,\, \bar{1}\, \cdots\, \bar{p}}+\phi_{1\, \cdots\, p,\, \bar{1}\, \cdots\, (\overline{w_l})_l\, \cdots\, \bar{p}}=0.
\end{equation}
Here $\phi_{1\,\cdots\, (w_k)_k\, \cdots\, p,\, \bar{1}\, \cdots\, \bar{p}}$ the $k$-the holomorphic position is filled by vector $w_k$.
For the simplicity of the notation we write $\phi_{1\,\cdots\, (w_k)_k\, \cdots\, p,\, \bar{1}\, \cdots\, \bar{p}}$ as $\phi_{1\,\cdots\, w_k\, \cdots\, p,\, \bar{1}\, \cdots\, \bar{p}}$.
Since this holds for any $p$-vectors $w_1, \cdots, w_p$, if we replace $t$ by $\sqrt{-1}t$, one can deduce from (\ref{1st-var}) that
\begin{equation}\label{1st-var-2}
\sum_{1\le k\le p}\phi_{1\,\cdots\, w_k\, \cdots\, p,\, \bar{1}\, \cdots\, \bar{p}}=\sum_{1\le l\le p}\phi_{1\, \cdots\, p,\, \bar{1}\, \cdots\, \overline{w_l}\, \cdots\, \bar{p}}=0.
\end{equation}
This implies that
$$\left(\sum_{1\le l\le m} \sum_{\nu =1}^{p}R^{\bar{l}}_{\,\,
\bar{\nu}}\phi_{1\, \cdots\, p, \, \bar{1}\, \cdots\, (\bar{l})_\nu\, \cdots\,\bar{p}}+\sum_{1\le k\le m}\sum_{\mu =1 }^{p}R_{\mu}^{\,\,
k}\phi_{1\, \cdots\, (k)_\mu\, \cdots\, p,\, \bar{1}\, \cdots\, \bar{p}}\right)=0.
$$
Now the fact $I''(0)\ge 0$ implies that
\begin{eqnarray}\label{2nd-var}
&\quad&\sum_{1\le k,\, l\le p} \phi_{1\,\cdots\, w_k\, \cdots\, p,\, \bar{1}\, \cdots\, \overline{w_l}\, \cdots\, \bar{p}}+\sum_{1\le k\ne l \le p}\phi_{1\, \cdots\, w_k\, \cdots\, w_l\, \cdots\,p,\, \bar{1}\, \cdots\, \bar{p}}\\
&\quad&\quad \quad + \sum_{1\le k\ne l \le p}\phi_{1\, \cdots\, p,\, \bar{1}\, \cdots\, \overline{w_k}\, \cdots\, \overline{w_l}\, \cdots\, \bar{p}}\ge 0. \nonumber
\end{eqnarray}
Replacing $t$ by $\sqrt{-1} t$ in $I(t)$, the fact that $I''(0)\ge 0$ will yield
\begin{eqnarray}\label{2nd-var-2}
&\quad&\sum_{1\le k,\, l\le p} \phi_{1\,\cdots\, w_k\, \cdots\, p,\, \bar{1}\, \cdots\, \overline{w_l}\, \cdots\, \bar{p}}-\sum_{1\le k\ne l \le p}\phi_{1\, \cdots\, w_k\, \cdots\, w_l\, \cdots\,p,\, \bar{1}\, \cdots\, \bar{p}}\\
&\quad&\quad \quad - \sum_{1\le k\ne l \le p}\phi_{1\, \cdots\, p,\, \bar{1}\, \cdots\, \overline{w_k}\, \cdots\, \overline{w_l}\, \cdots\, \bar{p}}\ge 0. \nonumber
\end{eqnarray}
Adding them up we have that for any $\mathfrak{w}=\left(\begin{array}{l} w_1\\ \vdots\\
 w_p\end{array}\right)\in \oplus_{1}^p T^{1, 0}_{x_0}M$, the Hermitian  form
\begin{equation}\label{2nd-var-3}
\mathcal{J}(\mathfrak{w}, \overline{\mathfrak{w}})\doteqdot \sum_{1\le k,\, l\le p} \phi_{1\,\cdots\, w_k\, \cdots\, p,\, \bar{1}\, \cdots\, \overline{w_l}\, \cdots\, \bar{p}}
\end{equation}
is semi-positive definite. A Hermitian-bilinear form $\mathcal{J}(\mathfrak{w}, \overline{\mathfrak{z}}) $
can be obtained via the polarization. In matrix form, the nonnegativity of $\mathcal{J}(\cdot, \cdot)$ is equivalent to that
\begin{eqnarray*}
A= \left(\begin{array}{l}
\phi_{(\cdot)\, 2\cdots\, p, \,      \bar{(\cdot)}\, \bar{2}\,\cdots\, \bar{p}}\quad
\phi_{1\, (\cdot)\, \cdots\, p, \,      \bar{(\cdot)}\, \bar{2}\,\cdots\, \bar{p}}\quad \quad \quad
\cdots \quad \quad \quad
\phi_{1\, 2\, \cdots\, (\cdot)_p, \,      \bar{(\cdot)}\, \bar{2}\,\cdots\, \bar{p}}\\
\phi_{(\cdot)\, 2\cdots\, p, \,     \bar{1}\, \bar{ (\cdot)}\, \cdots\, \bar{p}}
\quad    \phi_{1\, (\cdot)\, \cdots\, p, \,     \bar{1}\,  \bar{(\cdot)}\, \cdots\, \bar{p}}
\quad \quad \quad
\cdots
\quad \quad \quad
\phi_{1\, 2\, \cdots\, (\cdot)_p, \,    \bar{1}\,   \bar{(\cdot)}\, \cdots\, \bar{p}}\\
 \quad \quad \quad \cdots\quad \quad \quad \quad \quad \quad \cdots \quad \quad \quad \quad \quad \, \,\, \cdots \quad \quad \quad \quad \quad \quad \cdots \quad \quad \quad\\
\phi_{(\cdot)\, 2\cdots\, p, \,     \bar{1}\,\bar{2}\, \cdots\,  \bar{ (\cdot)}_p}\quad
\phi_{1\, (\cdot)\, \cdots\, p, \,     \bar{1}\,  \bar{2}\, \cdots\,  \bar{ (\cdot)}_p}\quad \quad \,
\cdots \quad \quad \quad
\phi_{1\, 2\, \cdots\, (\cdot)_p, \,    \bar{1}\,   \bar{2}\, \cdots\,  \bar{ (\cdot)}_p}
\end{array}\right)
\end{eqnarray*}
is a semi-positive definite Hermitian symmetric matrix. Namely $\ov{\mathfrak{w}}^{tr} A \mathfrak{w}\ge 0$.  Here we view $\phi_{1\,\cdots\,(\cdot)_\mu\, \cdots\,
 p, \, \bar{1}\, \cdots\, (\bar{\cdot})_\nu\, \cdots\, \bar{p}}$ as  a matrix such that for any vectors $w, z$, $\ov{z}^{tr} \phi_{1\,\cdots\,(\cdot)_\mu\, \cdots\,
 p, \, \bar{1}\, \cdots\, (\bar{\cdot})_\nu\, \cdots\, \bar{p}} \cdot w$  is the
   Hermitian-bilinear form  $\phi_{1\,\cdots\,(w)_\mu\, \cdots\,
 p, \, \bar{1}\, \cdots\, (\bar{z})_\nu\, \cdots\, \bar{p}}$.
The equivalence can be made via the identity
$
\mathcal{J}(\mathfrak{w}, \overline{\mathfrak{z}})=\langle A (\mathfrak{w}), \overline{\mathfrak{z}}\rangle.
$
Here $\mathfrak{w}=w_1\oplus\cdots \oplus w_p$, $\mathfrak{z}=z_1\oplus \cdots \oplus z_p$.
If we define $\phi^{\mu\bar{v}}$ by
$$
\langle \phi^{\mu\bar{\nu}}(X), \ov{Y}\rangle =\phi_{1\cdots (X)_{\mu}\cdots p, \bar{1}\cdots (\ov{Y})_\nu\cdots \bar{p}}
$$
it is easy to see that $\ov{(\phi^{\mu\bar{\nu}})^{\tr}}=\phi^{\nu \bar{\mu}}$. Using this notation
$\mathcal{J}(\mathfrak{w}, \overline{\mathfrak{z}})$ or $\langle \mathcal{J}(\mathfrak{w}), \overline{\mathfrak{z}}\rangle$ can be expressed as $\sum \langle \phi^{\mu \bar{\nu}}(w_\mu), \ov{z}_\nu\rangle$. It is easy to check that $\mathcal{J}$ is Hermitian symmetric.

What to be checked is that
$$
 \mathcal{KB}(\phi)_{1\, 2\, \cdots\, p,\, \bar{1}\, \bar{2}\, \cdots\,
\bar{p}}=\sum_{\mu = 1}^{p}\sum_{\nu =1 }^{p}R^{\,
k\bar{l}}_{\mu\,\,\bar{\nu}}\phi_{1\,\cdots\,(k)_\mu\, \cdots\,
 p, \, \bar{1}\, \cdots\, (\bar{l})_\nu\, \cdots\, \bar{p}} \ge0.
$$
Under the unitary frame it is equivalent to
\begin{equation}\label{goal1}\sum_{\mu = 1}^{p}\sum_{\nu =1 }^{p}R_{\mu\,\bar{\nu}\, l\,\bar{k}}\phi_{1\,\cdots\,(k)_\mu\, \cdots\,
 p, \, \bar{1}\, \cdots\, (\bar{l})_\nu\, \cdots\, \bar{p}} \ge0.
\end{equation}
Here we have used the 1st-Bianchi identity. If we can show that the Hermitian matrix
\begin{eqnarray*}
B=\left(\begin{array}{l}
R_{1\bar{1}(\cdot)\bar{(\cdot)}}\quad
R_{1\bar{2}(\cdot)\bar{(\cdot)}}\quad\quad
\cdots\quad\quad
R_{1\bar{p}(\cdot)\bar{(\cdot)}}\\
R_{2\bar{1}(\cdot)\bar{(\cdot)}}\quad
R_{2\bar{2}(\cdot)\bar{(\cdot)}}\quad\quad
\cdots\quad\quad
R_{2\bar{p}(\cdot)\bar{(\cdot)}}\\
\quad\cdots\quad\quad\quad\cdots\quad\quad\quad\cdots\quad\quad\quad\cdots\\
R_{p\bar{1}(\cdot)\bar{(\cdot)}}\quad
R_{p\bar{2}(\cdot)\bar{(\cdot)}}\quad\quad
\cdots\quad\quad
R_{p\bar{p}(\cdot)\bar{(\cdot)}}
\end{array}\right)
\end{eqnarray*}
is nonnegative, then the inequality (\ref{goal1}) holds since the left hand side  of (\ref{goal1}) is just the trace of the product matrix
$B\cdot A$ of the two nonnegative Hermitian symmetric matrices.

On the other hand the nonnegativity of $B$ is equivalent to for any $(1,0)$-vectors $w_1, \cdots,  w_p$,
$\langle B(\mathfrak{w}), \ov{\mathfrak{w}}\rangle \ge 0$ with  $\mathfrak{w}=w_1\oplus\cdots \oplus w_p$. This is equivalent to
$$
\sum R_{i\bar{j} w_j \ov{w}_i}\ge 0.
$$
Let $\alpha =\sum_{i=1}^p \frac{\partial}{\partial z^i}\wedge \ov{w}_i$ (we later simply denote as $\sum i\wedge \ov{w}_i$). Then $\langle \Rm (\alpha), \ov{\alpha}\rangle \ge 0$ is equivalent to the above inequality.  This proves (\ref{goal1}), hence the proposition,  at least for the case when $M$ is compact.

 Another way to look at this is to define the transformation $R^{\mu \bar{\nu}}$ as $\langle R^{\mu \bar{\nu}}(X), \ov{Y}\rangle =R_{\mu\bar{\nu}X \ov{Y}}$. Similarly one can easily check that
 $\ov{(R^{\mu \bar{\nu}})^{\tr}}=R^{\nu \bar{\mu}}$. Define transformation $\mathcal{K}$ on $\oplus_{i=1}^p T'M$ by
 $\langle \mathcal{K}(\mathfrak{w}), \ov{\mathfrak{z}}\rangle $ as $\sum \langle R^{\mu \bar{\nu}}(w_\nu), \ov{z}_\mu\rangle $. It is easy to check that $\mathcal{K}$ is Hermitian symmetric and that $\Rm \in \mathcal{C}_p$ implies  $\mathcal{K}\ge 0$.  Simple algebraic manipulation shows that:
 \begin{eqnarray*}
  \mbox{LHS of (\ref{goal1})} &=& \sum_{\mu, \nu=1}^{p} \langle R^{\mu\bar{\nu}}(e_l), \ov{e}_k\rangle \langle \phi^{\mu \bar{\nu}}(e_k), \ov {e}_l\rangle\\
  &=& \sum_{\mu, \nu=1}^p \langle R^{\mu\bar{\nu}}(\phi^{\mu \bar{\nu}}(e_k)), \ov{e}_k\rangle,
  \end{eqnarray*}
 if $\{e_k\}_{k=1\cdots m}$ is a unitary frame. One can see that this is nothing but the trace of $\mathcal{K} \cdot \mathcal{J}$ since a natural unitary base for $\oplus_{\mu=1}^p T'M$ is $\{E_{\mu k}\}$,  where $1\le \mu \le p, 1\le k \le m$, $E_{\mu k} =\vec{0} \oplus \cdots\oplus  (e_k)_{\mu}\oplus \cdots \oplus \vec{0}$. Then $\mathcal{J}(E_{\mu k})=\oplus_{\nu} \phi^{\mu \bar{\nu}}(e_k)$. Hence $\mathcal{K}(\mathcal{J}(E_{\mu k}))=\oplus_{\sigma}R^{\sigma \bar{\nu}}(\phi^{\mu \bar{\nu}}(e_k))$. This shows that $\langle \mathcal{K}(\mathcal{J}(E_{\mu k})), \ov{E_{\mu k}}\rangle = \sum \langle R^{\mu\bar{\nu}}(\phi^{\mu \bar{\nu}}(e_k)), \ov{e}_k\rangle$. Hence the left hand side of (\ref{goal1}) can be written as $\operatorname{trace} (\mathcal{K}\cdot \mathcal{J})$. Similarly for any $X_1, \cdots,  X_p$ we  can define $R^{\mu\bar{\nu}}$ and $\phi^{\mu \bar{\nu}}$ and $\mathcal{K}$ and $\mathcal{J}$. The above argument shows that
 \begin{equation}
 \mathcal{KB}_{X_1\cdots X_p, \ov{X}_1\cdots \ov{X}_p}=\operatorname{trace} (\mathcal{K}\cdot \mathcal{J}).
 \end{equation}
For noncompact complete manifolds we postpone it to Section 11 (Theorem \ref{max-pp-noncom}).

\end{proof}

\section{Invariance of $\mathcal{C}_p$ under the K\"ahler-Ricci flow}

Recently \cite{Wilking}, Wilking proved a very general result on invariance of cones of curvature operators under Ricci flow. The result is formulated for any Riemannian manifold $(M, g)$ of real dimension $n$. Since our proof is a modification of his we first state his result. Identify $TM $ with $\R^n$ and its complexification $TM \otimes \C$ with $\C^n$. Also identify $\wedge^2(\R^n)$ with the Lie algebra $\mathfrak{so}(n)$. The complexified Lie algebra $\mathfrak{so}(n, \C)$ can be identified with $\wedge^2(\C^n)$. Its associated Lie group is $\SO(n, \C)$, namely all complex matrices $A$ satisfying $A\cdot A^{\tr} =A^{\tr} \cdot A =\id$. Recall that there exists the natural action of $\SO(n, \C)$ on $\wedge ^2(\C^n)$ by extending the action $g\in \SO(n)$ on $x\otimes y$ ($g(x\otimes y) =gx\otimes gy$). Let $\Sigma \subset \wedge^2(\C^m)$ be a set which is invariant under the action of $\SO(n, \C)$.
 Let $\widetilde{\mathcal{C}}_{\Sigma}$ be the cone of curvature operators  satisfying that $\langle R(v), \bar{v} \rangle \ge 0$ for any $v\in \Sigma$. Here we view the space of algebraic curvature operators as a subspace of $S^2(\wedge^2(\R^n))$  satisfying the first Bianchi identity. Recently (May of 2008)  \cite{Wilking}, Wilking proved the following result.

\begin{theorem}[Wilking]\label{ww}
Assume that $(M, g(t))$, for $0\le t\le T$,  is a solution of Ricci flow on a compact manifold. Assume that $\Rm(g(0))\in \widetilde{\mathcal{C}}_{\Sigma}$. Then $\Rm(g(t))\in \widetilde{\mathcal{C}}_{\Sigma}$ for all $t\in [0, T]$.
\end{theorem}

It is not hard to see that this result contains the previous result of Brendle-Schoen \cite{BS} and Nguyen \cite{Ng} on the invariance of the cone of nonnegative isotropic curvature under the Ricci flow. In particular it implies the invariance of the cone of nonnegative complex sectional curvature, a useful  consequence first observed in  \cite{BS}  (see also \cite{NW} for an alternative proof). By modifying the argument of Wilking one can prove the following result.

\begin{corollary}\label{thm:p-NBC}
The K\"ahler-Ricci flow on a compact K\"ahler manifold preserves the cone $\mathcal{C}_p$ for any $p$.
\end{corollary}

For Riemannian manifolds, there exists another family of invariant cones which is analogous to $\mathcal{C}_p$.
We say that the complex sectional curvature of  $(M, g)$  is $k$-positive if
$$
\langle \Rm (v), \bar{v} \rangle >0
$$
for any nonzero $v\in \Lambda^2(\C^n)$ which can be written as $v=\sum_{i=1}^k{Z_i\wedge W_i}$
for some complex vectors $Z_i$ and $W_i \in TM \otimes \C$. Clearly the complex sectional curvature is $1$-positive is the same as positive complex sectional curvature. The $k$-positivity for $k\ge \frac{n(n-1)}{2}$ is the same as positive curvature operator.
Similarly one has the notion that the complex sectional curvature is $k$-nonnegative. In the space of the algebraic curvature operators $S_B(\mathfrak{so}(n))$, the ones with $k$-nonnegative complex sectional curvature form a cone $\widetilde{\mathcal{C}}_k$. Clearly $\widetilde{\mathcal{C}}_k \subset\widetilde{\mathcal{C}}_{k-1}$. The argument in \cite{NW} (this in fact was proved in an updated version of \cite{NW}) proves that

\begin{theorem}\label{thm:PkCSC}
The Ricci flow on a compact manifold preserves $\widetilde{\mathcal{C}}_k$.
\end{theorem}
Of course, this result is now also included in the previous mentioned general theorem of Wilking. In fact almost all the known invariant cones of nonnegativity type can  be formulated as a special case of Wilking's above  theorem.

Note  that even for a K\"ahler manifold $\mathcal{C}_p$ is a bigger cone than $\widetilde{\mathcal{C}}_p$. For example, for K\"ahler manifold $(M ,g)$, the nonnegativity of the complex sectional curvature implies that
$$
R_{s\bar{t}j\bar{i}}(a_i \bar{b}_j-c_i \bar{d}_j) \overline{(a_s\bar{b}_t -c_s\bar{d}_t)}\ge 0
$$
for any complex vector $\vec{a}$($=(a_1, \cdots, a_m)$), $\vec{b}$, $\vec{c}$ and $\vec{d}$.
Namely $(M, g)$ has {\it strongly nonnegative sectional curvature in the sense of Siu}, which is in general stronger than the nonnegativity of the sectional curvature (or bisectional curvature). On a K\"ahler manifold, if $\{E_i\}$ is a unitary basis of $T'M$, and letting $X=a_i E_i$, $Y=b_i E_i$, $Z=c_i E_i$ and $W=d_i E_i$, then the above is equivalent to $$\langle \Rm( (X+\bar{Z})\wedge (\bar{Y}+W)), \overline{ (X+\bar{Z})\wedge (\bar{Y}+W)}\rangle\ge 0.$$
In fact a simple computation as the above proves that $\widetilde{\mathcal{C}}_p= \mathcal{C}_{2p}$: For $1\le i\le p$, let $Z_i=X_{2i-1}+\bar{Y}_{2i}$ and $W_i=\bar{Y}_{2i-1}-X_{2i}$. Then $\Rm(\sum_{i=1}^p Z_i \wedge W_i)=\Rm (\sum_{j=1}^{2p} X_j\wedge \bar{Y}_{j})$. Thus by the fact that $\Rm$ is self-adjoint
$$
\langle \Rm (\sum_{i=1}^p Z_i \wedge W_i), \sum_{i=1}^p\overline {Z_i \wedge W_i}\rangle =\langle \Rm (\sum_{j=1}^{2p} X_j\wedge \bar{Y}_{j}),\sum_{j=1}^{2p}  \overline{X_j\wedge \bar{Y}_{j}}\rangle.
$$

In order to prove Corollary \ref{thm:p-NBC}, first let $G$ be the subgroup of $\SO(n, \C)$ consisting of  matrices $A \in \SO(n, \C)$ such that $A$ commutes with the almost complex structure $J=\left(\begin{matrix} 0 &\id\\ -\id &0\end{matrix}\right)$, noting here that $n=2m$. Let $\mathfrak{g}$ be the Lie algebra of $G$. A key observation is that $\mathfrak{g}$ consisting of $c\in \mathfrak{so}(n, \C)$ which commutes with $J$. It is easy to show that $\mathfrak{g}$ is the same as $\wedge^{1, 1}(\C^m)$ under the identification of $\wedge^2(\C^n)$ with $\mathfrak{so}(n, \C)$. More precisely $c=(c_{ij})$ (which is identified with $c_{ij}X_i \wedge \bar{X}_j$ for a unitary basis $\{X_i\}$) is identified with $\left(\begin{matrix} a & -b\\ b & a\end{matrix}\right)$ with $a=c-c^{\tr}$ and $b=-\sqrt{-1}(c+c^{\tr})$. Now  the argument of Wilking can be adapted to show the following  result for the K\"ahler-Ricci flow.

\begin{theorem} \label{kaehler} Let $\Sigma\subset \wedge^{1,1}(\C^m)$ be a set invariant under the adjoint action of $G$. Let $\mathcal{C}_\Sigma=\{ \Rm |\, \langle \Rm(v), \bar{v}\rangle \ge 0\}$ for any $v\in \Sigma$.
Assume that $(M, g(t))$ (with $t\in [0, T]$) is a solution to K\"ahler-Ricci flow on a compact K\"ahler manifold such that $\Rm(g(0))\in \mathcal{C}_\Sigma$. Then for any $t\in [0, T]$, $\Rm(g(t))\in \mathcal{C}_\Sigma$.
\end{theorem}

If $C\in G$ and $v=\sum_{k=1}^p X_k \wedge \bar{Y}_k$ with $X_k\in T'M$ and $\bar{Y}_k\in T''M$, $C(v)=\sum_{k=1}^p C(X_k) \wedge C(\bar{Y}_k)$. Since $C$ is commutative with $J$, $X_k'=C(X_k)\in T'M$ and $\bar{Y}_k'=C(\bar{Y}_k)\in T''M$. Hence the set consisting of all such $v$ is an invariant set $\Sigma_k$ under the adjoint action of $G$. The invariance of  cone $\mathcal{C}_k$ follows from the above theorem by applying to $\Sigma=\Sigma_k$.

In fact, one can easily generalize Wilking's result to manifolds with special holonomy group. When the manifold $(M, g)$ has a special holonomy group $G$ with holonomy algebra $\mathfrak{g}\subset \mathfrak{so}(n)$, since for any $v\in \mathfrak{so}(n)$, $\Rm(v)\in \mathfrak{g}$ by Ambrose-Singer theorem, Wilking's proof, in particular (\ref{A1})  remains the same even if  $\{b^\alpha\}$ being an orthonormal basis of $\mathfrak{g}$ instead of  an orthonormal basis of $\mathfrak{so}(n)$. Note that in this case $\Rm(b)=0$ for any $b\in \mathfrak{g}^{\perp}$. Let $\mathfrak{g}^\C\subset  \mathfrak{so}(n, \C)$ denote the complexified Lie algebra. It is easy to see then that for  any $b\in \left({\mathfrak{g}^{\C}}\right)^{\perp}$, $\langle \Rm (b), \overline{w}\rangle =\langle b, \overline{\Rm(w)}\rangle =0$ for any $w\in \mathfrak{so}(n, \C)$. Hence $\Rm(b)=0$.  This implies that  $\langle \Rm^{\#} (v), \overline{w}\rangle=\frac{1}{2}\operatorname{trace}(-\operatorname{ad}_{\overline{w}} \cdot \Rm \cdot \operatorname{ad}_{v} \cdot \Rm )$ with the trace taken for transformations of $\mathfrak{g}^\C$.

\begin{theorem}\label{general-holo} Assume that $(M, g_0)$ is a compact manifold with special holonomy group $G$ (and corresponding Lie algebra $\mathfrak{g}$). Let $\Sigma$ be  a subset of $\mathfrak{g}^\C$ satisfying the assumption that it is invariant under the adjoint action of $G^\C$, the complexification of $G$. Then if the curvature operator $\Rm$ of $g_0$ lies inside the cone $\mathcal{C}_\Sigma$, the curvature operator $\Rm$ of $g(t)$, the solution to Ricci flow with initial value $g(0)=g_0$,  also lies inside   $\mathcal{C}_\Sigma$.
\end{theorem}

When $G= \mathsf{U}(m)$ (with $n=2m$),  the unitary group, the above result implies Theorem \ref{kaehler}.
All above results in this section remain true on noncompact manifolds if we assume that the solution $g(t)$ has bounded curvature. This shall be proved in Section 10.

\section{A LYH type estimate for positive $(p, p)$-forms.}

First we recall some known computations of Kodaira \cite{M-K}.  Let $\phi$ be a $(p, q)$-form valued in a holomorphic Hermitian vector bundle $E$ with local frame $\{E_\a\}$ and locally $\phi=\sum \phi^\a E_{\a}$.
$$
\phi^\a = \frac{1}{p!q!}\sum \phi^\a_{I_p\bar{J}_q}dz^{I_p}\wedge dz^{\bar{J}_q}.
$$
Here $I_p=(i_1,\cdots,i_p), \bar{J}_q=(\bar{j}_1, \cdots,\bar{j}_q)$ and $dz^{I_p}=dz^{i_1}\wedge
\cdots\wedge dz^{i_p}$, $dz^{\bar{J}_q}=dz^{\overline{j_1}}\wedge\cdots \wedge dz^{\overline{j_q}}.$
For $(p, p)$ forms, $\phi_{1\cdots p, \bar{1}\cdots \bar{p}}$ differs from $\phi_{1\cdots p \bar{1}\cdots \bar{p}}$ by a factor $\left(\frac{1}{\sqrt{-1}}\right)^p(-1)^{\frac{p(p-1)}{2}}$.
The following two formulae are well known
\begin{equation}\label{eq:41}
(\bar{\p} \phi)^\a_{I_p\bar{j}_0\cdots\bar{j}_q}=(-1)^p\sum_{\nu =
0}^{q} (-1)^\nu \nabla_{\bar{j}_\nu}
\phi^{\a}_{I_p\bar{j}_0\cdots \hat{\bar{j}}_\nu\cdots \bar{j}_q},
\end{equation}
\begin{equation}\label{eq:42}
(\bar{\p}^*\phi)^\a_{I_p\bar{j}_1\cdots\bar{j}_{q-1}}=(-1)^{p+1}\sum _{ij}g^{\bar{j}i}\nabla_{i}
 \phi^{\a}_{I_p\ov{jj_1}\cdots\bar{j}_{q-1}}.
\end{equation}
Here $\hat{\bar{j_\nu}}$ means that the index $\bar{j}_\nu$ is
removed. From (\ref{eq:41}) and (\ref{eq:42}) we have
\begin{eqnarray}
 (\D_{\bar{\partial}}\phi)^{\a}_{I_p \bar{J}_q} & = & -\sum_{ij}g^{\bar{j}i}\nabla_i\nabla_{\bar{j}}\phi^{\a}_{I_p \bar{J}_q}+
\sum_{\nu = 1}^q \Omega^{\a\bar{l}}_{\b\, \bar{j}_\nu}\phi^{\b}_{I_p \bar{j}_1\cdots (\bar{l})_\nu \cdots\bar{j}_q}
 \nonumber \\
& \,&  +\sum_{\nu =1}^{q}R^{\bar{l}}_{\,\,
\bar{j}_\nu}\phi^{\a}_{I_p \bar{j}_1\cdots(\bar{l})\cdots\bar{j}_q}
-\sum_{\mu = 1}^{p}\sum_{\nu =1 }^{q}R^{\,
k\bar{l}}_{i_\mu\,\,\bar{j}_\nu}\phi^{\a}_{i_1\cdots(k)_\mu\cdots
 i_p \bar{j}_1\cdots(\bar{l})_\nu\cdots\bar{j}_q}\label{eq:43}
\end{eqnarray}
and
\begin{eqnarray}
 (\D_{\bar{\partial}}\phi)^{\a}_{I_p\bar{J}_q} & =&
-\sum_{ij}g^{i\bar{j}}\nabla_{\bar{j}}\nabla_{i}
\phi^\a_{I_p\bar{J}_q} + \sum_{\nu = 1}^q \Omega^{\a\bar{l}}_{\b
\bar{j}_\nu} \phi^{\b}_{I_p \bar{j}_1\cdots (\bar{l})_\nu
\cdots\bar{j}_q} -\sum_{\b}\Omega_{\b}^{\a}
\phi^{\b}_{I_p \bar{J}_q} \nonumber \\
& \, & +\sum_{\mu =1 }^{p}R_{i_\mu}^{\,\,
k}\phi^\a_{i_1\cdots(k)_\mu\cdots i_p \bar{J}_q} -\sum_{\mu =
1}^{p}\sum_{\nu =1 }^{q}R^{\, \, k\bar{l}}_{i_\mu\, \, \,
\bar{j}_\nu}\phi^{\a}_{i_1\cdots(k)_\mu\cdots
 i_p \bar{j}_1\cdots(\bar{l})_\nu\cdots\bar{j}_q}, \label{eq:44}
\end{eqnarray}
where $(k)_{\mu}$ means that the index in the $\mu$-th position is replaced by $k$. Here
$$
\Theta^\a_{\beta}=\frac{\sqrt{-1}}{2\pi}\sum \Omega^\a_{\beta \, i\bar{j}}dz^i\wedge d\bar{z}^j
$$
is the curvature of $E$ and $\Omega^\a_\beta$ is the mean curvature. Here we shall focus on the case that $E$ is the trivial bundle. Recall the contraction $\Lambda$ operator,  $\Lambda: \wedge^{p, q}\to \wedge^{p-1, q-1}$, defined  by
$$
(\Lambda \phi)_{i_1 \cdots i_{p-1} \bar{j}_1 \cdots \bar{j}_{q-1}}=\frac{1}{\sqrt{-1}}(-1)^{p-1} g^{i\bar{j}}\phi_{i i_1 \cdots i_{p-1} \bar{j} \bar{j}_1 \cdots \bar{j}_{p-1}}.
$$
Note our definition of $\Lambda$ differs from \cite{M-K} by a sign. With the above notations the K\"ahler identities assert
\begin{equation}\label{id1}
\partial \Lambda -\Lambda \partial =-\sqrt{-1}\bar{\partial} ^*, \quad
\bar{\partial}\Lambda -\Lambda \bar{\partial}=\sqrt{-1}\partial^*.
\end{equation}
In \cite{Ni-SDG}, the first author speculated that for $\phi$, a nonnegative $(p, p)$-form satisfying the Hodge-Laplacian heat equation, the $(p-1, p-1)$-form
\begin{equation}
Q(\phi, V)\doteqdot \frac{1}{2\sqrt{-1}}\left( \bar{\partial}^*\partial^*-\partial^*
\bar{\partial}^* \right)\phi
+\frac{1}{\sqrt{-1}}(\bar{\partial}^* \phi)_V  -\frac{1}{\sqrt{-1}}(\partial ^*\phi)_{\bar{V}} +\phi_{V,
\bar{V}}+\frac{\Lambda \phi}{t}\ge 0 \label{qn1}
\end{equation}
for any $(1,0)$ type vector field $V$.  Here $\phi_{V, \ov{V}}$ is a $(p-1,p-1)$-form defined as
$$(\phi_{V, \ov{V}})_{I_{p-1},\ov{J}_{p-1}}\doteqdot \phi_{V I_{p-1}, \ov{V} J_{p-1}}$$
or equivalently
$$
(\phi_{V, \ov{V}})_{I_{p-1}\ov{J}_{p-1}}\doteqdot \frac{1}{\sqrt{-1}}(-1)^{p-1}\phi_{V I_{p-1} \ov{V} J_{p-1}},
$$
and for $\psi$ and $\psi'$, $\psi_V$ and $\psi'_{\bar{V}}$ are defined as
$$
(\psi_V)_{I_{p-1} \ov{J}_q}\doteqdot\psi_{V I_{p-1} \ov{J}_q}, \quad \quad  (\psi'_{\bar{V}})_{I_p \ov{J}_{q-1}}\doteqdot (-1)^p\psi'_{I_p \bar{V}\ov{J}_{q-1}}.
$$
When the meaning is clear we abbreviate $Q(\phi, V)$ as $Q$.
The expression of $Q$ coincides with the quantity $Z$ in Theorem 1.1 of \cite{N-jams}  for $(1,1)$-forms due to (\ref{eq:41}), (\ref{eq:42}) as well as their cousins
\begin{eqnarray}
(\partial \phi)^\a_{i_0i_1\cdots i_p \ov{J}_q}&=&\sum_{\mu=0}^p (-1)^\mu \nabla_{i_\mu} \phi^\a_{i_0\cdots \hat{i}_\mu \cdots i_p \ov{J}_q},\label{eq:47}\\
(\partial^* \phi)^\a_{i_1\cdots i_{p-1} \ov{J}_q}&=& -\sum_{ij}g^{i\bar{j}}\nabla_{\bar{j}}\phi^\a_{ii_1\cdots i_{p-1} \ov{J}_q}\label{eq:48}
\end{eqnarray}
since the operators (defined in \cite{N-jams} for the case $p=1$) $\operatorname{div}(\phi)_{i_1\cdots i_{p-1}, \bar{J}_p}$, $\operatorname{div}(\phi)_{I_p, \bar{j}_1\cdots \bar{j}_{p-1}}$ and $g^{i\bar{j}}\nabla_i \operatorname{div}(\phi)_{I_{p-1}, \bar{j}\bar{j}_1\cdots \bar{j}_{p-1}}$  can be identified with $\partial^*$, $\bar{\partial}^*$ and $\bar{\partial}^*\partial^*$ etc.  To make it precise, first note that in our discussion the bundle is trivial and we can forget about the upper index in $\phi^\alpha$. It is easy to check that $\Lambda(\phi)_{I_{p-1}, \ov{J}_{p-1}}=g^{i\bar{j}}\phi_{i I_{p-1}, \bar{j}\ov{J}_{p-1}}$. In (\ref{qn1}), $(\bar{\partial}^* \phi)_V$ is a $(p-1, p-1)$-form which by the definition can be written as
$$
(\bar{\partial}^* \phi)_V=\sum_{i=1}^m V^i \iota_i  (\bar{\partial}^* \phi)
$$
where $\iota_i$ is the adjoint of the operator $dz^i \wedge(\cdot)$. Hence $(\bar{\partial}^* \phi)_V=\iota_V \cdot \bar{\partial}^*\phi$.
A direct calculation then shows that
$$
\frac{1}{\sqrt{-1}}\left((\bar{\partial}^* \phi)_V\right)_{I_{p-1},\ov{J}_{p-1}}=V^{i} g^{\bar{j}k}\nabla_k \phi_{i I_{p-1}, \bar{j}\ov{J}_{p-1}}.$$
Similarly, $(\partial^* \phi)_{\ov{V}}=\sum_{j=1}^m V^{\bar{j}}\iota_{\bar{j}}(\partial^* \phi)=\iota_{\ov{V}} \cdot \partial^*\phi $ and
another direct computation  implies that
$$ \frac{1}{\sqrt{-1}}\left((\partial^*\phi)_{\ov{V}}\right)_{I_{p-1}, \ov{J}_{p-1}}=-\ov{V^j}g^{i\bar{k}}\nabla_{\bar{k}}\phi_{iI_{p-1}, \bar{j}\ov{J}_{p-1}}.$$
If we define
\begin{eqnarray*}
(\operatorname{div}''(\phi))_{I_{p},\ov{J}_{p-1}}  \doteqdot\sum _{ij}g^{\bar{j}i}\nabla_{i}
 \phi_{I_p,\ov{j} \ov{J}_{p-1}}, &\,& (\operatorname{div}'(\phi))_{I_{p-1},\ov{J}_{p}} \doteqdot \sum_{ij}g^{i\bar{j}}\nabla_{\bar{j}}\phi_{iI_{p-1}, \ov{J}_p}, \\(\operatorname{div}''_{V}(\phi))_{I_{p-1}, \ov{J}_{p-1}}\doteqdot(\operatorname{div}''(\phi))_{VI_{p-1},  \bar{J}_{p-1}}
,\quad &\, &\quad (\operatorname{div}'_{\ov{V}}(\phi))_{I_{p-1}, \ov{J}_{p-1}}\doteqdot (\operatorname{div}'(\phi))_{I_{p-1}, \ov{V} \bar{J}_{p-1}}
\end{eqnarray*}
then simple calculation shows that
\begin{eqnarray*}
\frac{1}{\sqrt{-1}}(\bar{\partial}^* \phi)_V =\operatorname{div}''_{V}(\phi), \quad &\, & \quad \frac{1}{\sqrt{-1}}(\partial^* \phi)(\ov{V})=-\operatorname{div}'_{\ov{V}}(\phi),\\
 (\bar{\partial}^*\partial^* \phi)_{I_{p-1}, \ov{J}_{p-1}} &=& \sqrt{-1} \operatorname{div}''  (\operatorname{div}'(\phi))_{I_{p-1}, \ov{J}_{p-1}},\\
(\partial^*\bar{\partial}^* \phi)_{I_{p-1}, \ov{J}_{p-1}} &=& -\sqrt{-1} \operatorname{div}' (\operatorname{div}''(\phi))_{I_{p-1}, \ov{J}_{p-1}}.
\end{eqnarray*}
Hence $Q$ also has the following equivalent form:
\begin{eqnarray*}
Q_{I_{p-1}, \ov{J}_{p-1}}&=&\frac{1}{2}\left[\operatorname{div}''  (\operatorname{div}'(\phi))+ \operatorname{div}' (\operatorname{div}''(\phi))\right]_{I_{p-1}, \ov{J}_{p-1}}+(\operatorname{div}'(\phi))_{I_{p-1},\bar{V} \bar{j}_1\cdots \bar{j}_{p-1}}\\
&\quad&+(\operatorname{div}''(\phi))_{Vi_1\cdots i_{p-1},\ov{J}_{p-1}}+\phi_{V I_{p-1}, \bar{V} \ov{J}_{p-1}}+\frac{1}{t}(\Lambda \phi)_{I_{p-1}, \ov{J}_{p-1}}.
\end{eqnarray*}
Recall that $d=\partial +\bar{\partial}$ and $d_c=-\sqrt{-1}(\partial -\bar{\partial})$. One can also write $Q$ as
\begin{equation}\label{lyh-ddc}
Q=\frac{1}{2}d_c^* d^* \phi -\Pi_{p-1, p-1}\cdot  \iota_{V+\ov{V}}\cdot  d_c^*\phi +\phi_{V, \bar{V}}+\frac{\Lambda \phi}{t}
\end{equation}
as well as
\begin{equation}\label{lyh-ddpar}
Q=\frac{1}{\sqrt{-1}}\bar{\partial}^*\partial^* \phi+\frac{1}{\sqrt{-1}}\iota_{V}\cdot  \bar{\partial}^* \phi-\frac{1}{\sqrt{-1}}\iota_{\ov{V}}\cdot \partial^* \phi  +\phi_{V, \bar{V}}+\frac{\Lambda \phi}{t}
\end{equation}
Here  $\Pi_{p-1, p-1}$ is the projection to the $\wedge^{p-1, p-1}$-space.
When $\phi$ is $d$-closed and write $\psi=\Lambda \phi$, the K\"ahler identities and its consequence
$\Delta_{\bar{\partial}}=\Delta_\partial$ imply that
$$
Q=\psi_t +\frac{1}{2}\left(\bar{\partial}\bar{\partial}^*+\partial \partial^*\right) \psi +\iota_{V}\cdot  \partial \psi+\iota_{\ov{V}}\cdot \bar{\partial}\psi +\phi_{V, \ov{V}}+\frac{\psi}{t}.
$$
Sometimes we also abbreviate $\iota_{V} \partial \psi,  \iota_{\ov{V}}\bar{\partial}\psi$ as $\partial_V \psi, \bar{\partial}_{\ov{V}} \psi$.

\begin{theorem}\label{main-lyh} Let $(M, g)$ be a complete K\"ahler manifold. Assume that $\phi(x, t)$ is a positive solution to  (\ref{eq:11}). Assume further that the curvature of $(M, g)$  satisfies $\mathcal{C}_p$. Then $Q\ge 0$ for any $t>0$. Furthermore, if the equality holds somewhere for $t>0$, then $(M, g)$ must be flat.
\end{theorem}

\begin{proof} We postpone the proof on the part that $Q\ge 0$ to Section 5 and 11. Here we show the rigidity result implied by the  equality case, which can be reduced to the $p=1$ case treated in \cite{N-jams}. The observation is that $\Lambda (Q(\phi, V))=Q(\Lambda \phi, V)$. This can be seen via the well-known facts (cf. Corollary 4.10 of Chapter 5 of \cite{wells}) that
$$
[\Lambda, \partial^*]=[\Lambda, \bar{\partial}^*]=0
$$
as well as $\phi_{V, \bar{V}}=\sqrt{-1}\sum_{ij=1}^m V^{i} V^{\bar{j}}\iota_{i}\iota_{\bar{j}}\phi=\sqrt{-1}\iota_V\cdot \iota_{\ov{V}}\cdot \phi$ and the equalities
$$
[\Lambda, \iota_{V}]=[\Lambda, \iota_{\ov{V}}]=0.
$$
One can refer to (3.19) of Chapter 5 in \cite{wells} for a proof of the above identities.
Hence if $Q(\phi, V)=0$, it implies that $Q(\Lambda^{p-1}\phi, V)=\Lambda^{p-1}(Q(\phi, V))=0$. Now the result follows from Theorem 1.1 of \cite{N-jams} applying to $\Lambda^{p-1}\phi$, which is a positive $(1,1)$-form.
\end{proof}

Note that in the statement of the theorem, $Q(\phi, V)=0$ means it equals to the zero as a $(p-1, p-1)$-form.

\begin{corollary}\label{main-con1}
Let $(M, g)$ and  $\phi$ be as above. Let $\psi=\Lambda \phi$ and assume that $\phi$ is $d$-closed. Then
\begin{equation}\label{con1-lyh}
\frac{1}{t}\frac{\partial}{\partial t}\left(t \psi\right)+\frac{1}{2}\left(\bar{\partial}\bar{\partial}^* +\partial \partial^*\right)\psi \ge -\min_{V}\left(\partial_V \psi+\bar{\partial}_{\ov{V}}\psi +\phi(V, \ov{V})\right)\ge 0.
\end{equation}
In particular, for any $\phi\ge 0$, if $M$ is compact,
\begin{equation}\label{mono}
\frac{d}{dt}\int_M t \psi\wedge \omega^{m-p+1} \ge 0.
\end{equation}
\end{corollary}
\begin{proof} By adding $\epsilon \omega^p$ and then letting $\epsilon\to 0$, we can assume that $\phi$ is strictly positive. Then for any $I_{p-1}=(i_1, \cdots, i_{p-1})$,  the Hermitian bilinear form ( in $V$),
$(\psi_t+\frac{\psi}{t})_{I_{p-1}, \ov{I}_{p-1}}+(\partial_V \psi+\bar{\partial}_{\ov{V}}\psi)_{I_{p-1}, \ov{I}_{p-1}}
+\phi(V, \ov{V})_{I_{p-1}, \ov{I}_{p-1}}
$ has a minimum. It is then easy to see that for the minimizing vector $V$,
$$
(\partial_V \psi+\bar{\partial}_{\ov{V}}\psi)_{I_{p-1}, \ov{I}_{p-1}}
+\phi(V, \ov{V})_{I_{p-1}, \ov{I}_{p-1}}=-\phi(V, \ov{V})_{I_{p-1}, \ov{I}_{p-1}}.
$$
The first result then follows. For the second one, just notice that
$$
\int_M (\bar{\partial}\bar{\partial}^*+\partial \partial^*)\psi\wedge \omega^{m-p+1}=0.
$$
\end{proof}

\begin{remark} Clearly, if one can perform the integration by parts and control the boundary terms, the monotonicity (\ref{mono}) still holds on noncompact case.
\end{remark}

One can define a formal dual operator of $Q(\phi, V)$ as
\begin{equation}\label{lyh-dual}
Q^*(\psi, V^*)\doteqdot\sqrt{-1}\partial \bar{\partial} \psi +\sqrt{-1}V^*\wedge \bar{\partial}\psi-\sqrt{-1}\ov{V}^* \wedge \partial \psi +\sqrt{-1}V^*\wedge \ov{V}^* \wedge \psi+\frac{\omega\wedge \psi}{t}
\end{equation}
which maps a $(m-p, m-p)$-form $\psi$ to a $(m-p+1, m-p+1)$-form. Here $V^*$ is a $(1,0)$ type co-vector.  The following duality can be checked by  direct calculations, making use of the following well known identities on $(p,q)$-forms (cf. Proposition 2.4, (1.13) and (3.14) of Chapter 5 of \cite{wells} repectively):
\begin{eqnarray*}
\bar{\partial}^* =-* \cdot \partial \cdot *\, ,  \quad &\quad& \quad \partial^*=-* \cdot \bar{\partial} \cdot *\, ;\\
\Lambda&=& (-)^{p+q} * \cdot  (\omega \wedge)\cdot *\, ;  \\
\iota_{V} =*\cdot (\ov{V}^* \wedge)\cdot *\, ,\quad &\quad&\quad \iota_{\ov{V}} =*\cdot (V^* \wedge)\cdot *\, .
\end{eqnarray*}
Here $*$ is the Hodge-star operator.
\begin{proposition} Let $\phi$,  $V$ be as the above discussion. Let $V^*$ be the dual of $V$. Let $*$ be the Hodge star operator.  Then
\begin{equation}\label{dual1}
Q(\phi, V)=* \cdot Q^* (* \cdot \phi, V^*).
\end{equation}
\end{proposition}

By this duality, one can identify the result for the $(m, m)$-forms with that of \cite{CN}. In the rest of this section we derive some preliminary results  useful for the proof of Theorem \ref{main-lyh}.
The following lemma follows from (\ref{eq:43}), (\ref{eq:44}) and the fact that $[\Delta_{\bar{\partial}}, \partial^*]=[\Delta_{\bar{\partial}}, \bar{\partial}^*]=0$.

\begin{lemma} \label{help41} Let $\phi$ be a $(p, p)$-form satisfying (\ref{eq:11}). Then $(\frac{\partial}{\partial t}+\Delta_{\bar{\partial}})\bar{\partial}^* \phi =(\frac{\partial}{\partial t}+\Delta_{\bar{\partial}})\partial^* \phi=0$. Hence by (\ref{eq:43}), (\ref{eq:44})
\begin{eqnarray}
 &\,&\heat (\bar{\partial}^*\phi)_{i i_1\cdots i_{p-1}, \bar{j}_1\cdots \bar{j}_{p-1}}= \sum_{\mu = 1}^{p-1}\sum_{\nu =1 }^{p-1}R^{\,
k\bar{l}}_{i_\mu\,\,\bar{j_\nu}}(\bar{\partial}^*\phi)_{ii_1\cdots(k)_\mu\cdots
 i_{p-1},\bar{j_1}\cdots(\bar{l})_\nu\cdots\bar{j}_{p-1}} \label{eq:lem41}\\
\quad  \quad &
+& \sum_{\nu =1 }^{p-1}R^{\,
k\bar{l}}_{i\,\,\bar{j_\nu}}(\bar{\partial}^*\phi)_{ki_1\cdots
 i_{p-1},\bar{j_1}\cdots(\bar{l})_\nu\cdots\bar{j}_{p-1}}-\frac{1}{2}R_{i}^{\,\,
k}(\bar{\partial}^*\phi)_{ki_1\cdots i_{p-1},\bar{j}_1\cdots \bar{j}_{p-1}} \nonumber\\
\quad \quad &
-&\frac{1}{2}\left(\sum_{\nu =1}^{p-1}R^{\bar{l}}_{\,\,
\bar{j}_\nu}(\bar{\partial}^*\phi)_{ii_1\cdots i_{p-1},\bar{j_1}\cdots(\bar{l})\cdots\bar{j}_{p-1}}+\sum_{\mu =1 }^{p-1}R_{i_\mu}^{\,\,
k}(\bar{\partial}^*\phi)_{ii_1\cdots(k)_\mu\cdots i_{p-1},\bar{j}_1\cdots \bar{j}_{p-1}}\right),\nonumber\\
&\,&\heat (\partial^*\phi)_{ i_1\cdots i_{p-1}, \bar{j}\bar{j}_1\cdots \bar{j}_{p-1}}= \sum_{\mu = 1}^{p-1}\sum_{\nu =1 }^{p-1}R^{\,
k\bar{l}}_{i_\mu\,\,\bar{j_\nu}}(\partial^*\phi)_{i_1\cdots(k)_\mu\cdots
 i_{p-1},\bar{j}\bar{j_1}\cdots(\bar{l})_\nu\cdots\bar{j}_{p-1}} \label{eq:lem42}\\
\quad \quad &
+& \sum_{\mu =1 }^{p-1}R^{\,
k\bar{l}}_{i_\mu\,\,\bar{j}}(\partial^*\phi)_{i_1\cdots(k)_\mu\cdots
 i_{p-1}, \bar{l}\bar{j_1}\cdots\bar{j}_{p-1}}-\frac{1}{2}R^{\bar{l}}_{\,\,
\bar{j}}(\partial^*\phi)_{i_1\cdots i_{p-1},\bar{l}\bar{j}_1\cdots \bar{j}_{p-1}} \nonumber \\
\quad \quad&
-&\frac{1}{2}\left(\sum_{\nu =1}^{p-1}R^{\bar{l}}_{\,\,
\bar{j}_\nu}(\partial^*\phi)_{i_1\cdots i_{p-1},\bar{j}\bar{j_1}\cdots(\bar{l})\cdots\bar{j}_{p-1}}+\sum_{\mu =1 }^{p-1}R_{i_\mu}^{\,\,
k}(\partial^*\phi)_{i_1\cdots(k)_\mu\cdots i_{p-1},\bar{j}\bar{j}_1\cdots \bar{j}_{p-1}}\right).\nonumber
\end{eqnarray}
\end{lemma}

Similarly, one can write the following lemma.
\begin{lemma} \label{help42}Let $\phi$ be a solution to (\ref{eq:11}). Then $(\frac{\partial}{\partial t}+\Delta_{\bar{\partial}})(\partial^*
\bar{\partial}^* \phi)=(\frac{\partial}{\partial t}+\Delta_{\bar{\partial}})(\bar{\partial}^*\partial^* \phi)=0$. Hence (\ref{eq:43}), (\ref{eq:44}) imply similar equations for
$\heat (\partial^*
\bar{\partial}^* \phi)$ and $\heat (\bar{\partial}^*\partial^* \phi)$. Namely
\begin{eqnarray}
&\,&\heat  (\partial^*
\bar{\partial}^* \phi)_{i_1\cdots i_{p-1}, \bar{j}_1\cdots \bar{j}_{p-1}}=\sum_{\mu = 1}^{p-1}\sum_{\nu =1 }^{p-1}R^{\,
k\bar{l}}_{i_\mu\,\,\bar{j_\nu}}(\partial^*
\bar{\partial}^* \phi)_{i_1\cdots(k)_\mu\cdots
 i_{p-1},\bar{j_1}\cdots(\bar{l})_\nu\cdots\bar{j}_{p-1}} \label{eq:lm43}\\
 &\quad&
-\frac{1}{2}\left(\sum_{\nu =1}^{p-1}R^{\bar{l}}_{\,\,
\bar{j}_\nu}(\partial^*
\bar{\partial}^* \phi)_{i_1\cdots i_{p-1},\bar{j_1}\cdots(\bar{l})\cdots\bar{j}_{p-1}}+\sum_{\mu =1 }^{p-1}R_{i_\mu}^{\,\,
k}(\partial^*
\bar{\partial}^* \phi)_{i_1\cdots(k)_\mu\cdots i_{p-1}, \bar{j}_1\cdots\bar{j}_{p-1}}\right).\nonumber
\end{eqnarray}
Simply put $\heat (\partial^*
\bar{\partial}^* \phi)=\mathcal{KB}(\partial^*
\bar{\partial}^* \phi)$ and $\heat (\bar{\partial}^*\partial^* \phi)=\mathcal{KB}(
\bar{\partial}^* \partial^* \phi)$.
\end{lemma}

Lemma \ref{help41} implies that
\begin{eqnarray}
 &\,&\heat (\operatorname{div}''_{V}(\phi))_{I_{p-1}, \ov{J}_{p-1}}=\operatorname{div}''_{\heat V}(\phi) +\mathcal{KB}(\operatorname{div}''_{V}(\phi)) \label{eq:lem412}\\
\quad  \quad &+& \sum_{\nu =1 }^{p-1}R^{\,
k\bar{l}}_{V\,\bar{j_\nu}}(\operatorname{div}''(\phi))_{ki_1\cdots
 i_{p-1},\bar{j_1}\cdots(\bar{l})_\nu\cdots\bar{j}_{p-1}}-\frac{1}{2}\operatorname{div}''_{\Ric(V)}(\phi)
  \nonumber\\
 \quad \quad &-& g^{i\bar{j}}\left((\nabla_{i} \operatorname{div}''(\phi))_{\nabla_{\bar{j}}V}+(\nabla_{\bar{j}} \operatorname{div}''(\phi))_{\nabla_{i}V}\right); \nonumber \\
&\,&\heat (\operatorname{div}'_{\ov{V}}(\phi))_{I_{p-1}, \bar{J}_{p-1}}= \operatorname{div}'_{\heat \ov{V}}(\phi)+\mathcal{KB}(\operatorname{div}'_{\ov{V}}(\phi)) \label{eq:lem413}\\
\quad \quad &
+& \sum_{\mu =1 }^{p-1}R^{\,
k\bar{l}}_{i_\mu\,\bar{V}}(\operatorname{div}'(\phi))_{i_1\cdots(k)_\mu\cdots
 i_{p-1}, \bar{l}\bar{j_1}\cdots\bar{j}_{p-1}}-\frac{1}{2}\operatorname{div}'_{\ov{\Ric(V)}}(\phi) \nonumber\\
 \quad \quad &-& g^{i\bar{j}}\left((\nabla_{i} \operatorname{div}'(\phi))_{\nabla_{\bar{j}}\ov{V}}+(\nabla_{\bar{j}} \operatorname{div}'(\phi))_{\nabla_{i}\ov{V}}\right). \nonumber
\end{eqnarray}
since for any $r$-tensor $T$, the $r-1$-tensor $T_V(X_1, \cdots X_{r-1})\doteqdot T(V, X_1, \cdots X_{r-1})$ satisfies $\nabla_X T_V =(\nabla_X T)_V +T_{\nabla_X V}$. To compute the evolution equation of $Q$, since
$ \heat \Lambda \phi =\mathcal{KB}(\Lambda \phi)$, the only term left is the evolution equation on $\phi(V, \ov{V})$ which we also abbreviate  as $\phi_{V,\bar{V}}$. Since $\nabla_X \phi_{V,\bar{V}}= (\nabla_X \phi)_{V,\bar{V}}+\phi_{\nabla_X V, \bar{V}}+\phi_{V, \nabla_X \bar{V}}$, Lemma \ref{lemma11} implies that
\begin{eqnarray}
&\,&\heat \phi_{V,\bar{V}}=\mathcal{KB}(\phi_{V,\bar{V}})+ R^{\,
k\bar{l}}_{V\,\bar{V}} \phi_{k, \bar{l}}+\sum_{\nu =1 }^{p-1}R^{\,
k\bar{l}}_{V\,\bar{j_\nu}} \phi_{k I_{p-1}, \bar{V} \bar{j}_1\cdots (\bar{l})_\nu\cdots \bar{j}_{p-1}}\label{eq:414}\\
\quad \quad &+&\sum_{\mu =1 }^{p-1}R^{\,
k\bar{l}}_{i_\mu\,\bar{V}} \phi_{V i_1\cdots (k)_\mu\cdots i_{p-1}, \bar{l}\bar{J}_{p-1}}-\frac{1}{2}\left(\phi_{V, \ov{\Ric(V)}}+\phi_{\Ric(V), \bar{V}}\right) \nonumber\\
\quad \quad &+& \phi_{\heat V, \bar{V}}+\phi_{V, \heat \bar{V}}-g^{i\bar{j}}\left(\phi_{\nabla_i V, \nabla_{\bar{j}}\bar{V}}+\phi_{\nabla_{\bar{j}}V, \nabla_i \bar{V}}\right)\nonumber\\
\quad \quad &-& g^{i\bar{j}}\left((\nabla_i \phi)_{\nabla_{\bar{j}} V, \bar{V}}+(\nabla_{i}\phi)_{V, \nabla_{\bar{j}}\bar{V}}+(\nabla_{\bar{j}}\phi )_{\nabla_i V, \bar{V}}+(\nabla_{\bar{j}}\phi)_{ V, \nabla_i\bar{V}}\right). \nonumber
\end{eqnarray}

\section{The proof of Theorem \ref{main-lyh}.}
Now $Q$ is viewed as a $(p-1, p-1)$-form. For $p=1 $ and $p=m$,  the result has been proven earlier. Using the notations introduced in the last section the LYH quantity $Q$, $(p-1, p-1)$-form depending on vector field $V$,  can be written as
$$
Q=\frac{1}{2}\left[\operatorname{div}''  (\operatorname{div}'(\phi))+ \operatorname{div}' (\operatorname{div}''(\phi))\right]+\operatorname{div}'_{\ov{V}}(\phi)+\operatorname{div}''_{V}(\phi) +\phi_{V, \ov{V}}+\frac{\Lambda \phi}{t}.
$$
As before if we assume that at $(x_0, t_0)$,  for the first time, for  some $V$, $Q_{v_1 v_2\cdots v_{p-1}, \bar{v}_1\cdots\bar{v}_{p-1}}=0$ for some linearly independent vectors $\{v_i\}_{i=1}^{p-1}$. By a perturbation argument as in \cite{N-jams} we can assume without the loss of the generality that $\phi$ is strictly positive. As in \cite{richard-harnack}, it suffices to check that at the point $(x_0, t_0)$,
$\heat Q\ge 0$. Since the complex function (in terms of the variable $z$)
\begin{eqnarray*}
I(z)&\doteqdot &\frac{1}{2}\left[\operatorname{div}''  (\operatorname{div}'(\phi))+ \operatorname{div}' (\operatorname{div}''(\phi))\right]_{v_1(z)\cdots v_{p-1}(z), \bar{v}_1(z)\cdots \bar{v}_{p-1}(z)}\\
&\, & +\left[\operatorname{div}'_{\ov{V}(z)}(\phi)+\operatorname{div}''_{V(z)}(\phi) \right]_{v_1(z)\cdots v_{p-1}(z), \bar{v}_1(z)\cdots \bar{v}_{p-1}(z)}\\
&\,& +\phi_{V(z)v_1(z)\cdots v_{p-1}(z), \bar{V}(z)\bar{v}_1(z)\cdots \bar{v}_{p-1}(z)}+\frac{\Lambda \phi_{v_1(z)\cdots v_{p-1}(z), \bar{v}_1(z)\cdots \bar{v}_{p-1}(z)}}{t}
\end{eqnarray*}
satisfies $I(0)=0$ and $I(z)\ge 0$ for any variational vectors $v_{\mu}(z), V(z)$, holomorphic in $z$,  with $v_{\mu}(0)=v_{\mu}$ and $V(0)=V$. In particular, letting $v_{\mu}(z)=v_{\mu}$ and $V'(0)=X$  we have that
\begin{equation}\label{eq:51}
\operatorname{div}'_{\bar{X}}(\phi)+\phi_{V, \bar{X}}=0= \operatorname{div}''_{X}(\phi)+\phi_{X, \bar{V}}.
\end{equation}
Similarly  by fixing $V(z)=V$ and varying $v_{\mu}(z)$, we deduce for any $X$,
\begin{equation}\label{eq:52}
Q_{v_1\cdots (X)_\mu\cdots v_{p-1}, \bar{v}_1\cdots \bar{v}_{p-1}}=0=Q_{v_1\cdots v_{p-1}, \bar{v}_1\cdots (\bar{X})_{\nu}\cdots v_{p-1}}.
\end{equation}
As before, after a changing of variables we may assume that $\{ v_i\}_{i=1}^{p-1} =\{\frac{\partial}{\partial z^i}\}_{i=1}^{p-1}$.
Since $z=0$ is the minimizing point we have that $\Delta I(0)\ge 0$. If $v'_\mu(0)=X_\mu$ and $V'(0)=X$, where $v'(z)=\frac{\partial v}{\partial z}$, this implies that
\begin{eqnarray*}
&\, &\sum_{\mu, \nu=1}^{p-1} Q_{v_1\cdots X_\mu \cdots v_{p-1}, \bar{v}_1\cdots \bar{X}_\nu \cdots \bar{v}_{p-1}}+ \phi_{X v_1\cdots v_{p-1}, \bar{X} \bar{v}_1\cdots \bar{v}_{p-1}}\\
\quad \quad \quad &+& \sum_{\mu=1}^{p-1}\operatorname{div}'_{\bar{X}}(\phi)_{v_1\cdots X_\mu\cdots v_{p-1}, \bar{v}_1\cdots \bar{v}_{p-1}}+\phi_{V v_1\cdots X_\mu\cdots v_{p-1}, \bar{X} \bar{v}_1\cdots \bar{v}_{p-1}}\\
\quad \quad \quad &+& \sum_{\nu=1}^{p-1}\operatorname{div}''_{X}(\phi)_{v_1\cdots v_{p-1}, \bar{v}_1\cdots \bar{X}_\nu\cdots \bar{v}_{p-1}}+\phi_{X v_1\cdots v_{p-1}, \bar{V} \bar{v}_1\cdots \bar{X}_\nu \cdots \bar{v}_{p-1}}\\
\quad \quad \quad &\ge& 0.
\end{eqnarray*}
This amounts to that the block matrix
 \begin{eqnarray*}
\mathcal{M}_1=\left(\begin{array}{l} A \quad \quad\quad \quad S \quad \quad \\
\overline{S}^{tr}\quad \phi_{(\cdot )1\cdots p-1, (\bar{\cdot})\bar{1}, \cdots \ov{p-1}}
\end{array}\right)\ge 0
\end{eqnarray*}
where
 \begin{eqnarray*}
A= \left(\begin{array}{l}
Q_{(\cdot)\, 2\cdots\, p-1, \,      \bar{(\cdot)}\, \bar{2}\,\cdots\, \ov{p-1}}\quad
Q_{1\, (\cdot)\, \cdots\, p-1, \,      \bar{(\cdot)}\, \bar{2}\,\cdots\, \ov{p-1}}\quad \quad \quad
\cdots \quad \quad \quad
Q_{1\, 2\, \cdots\, (\cdot)_{p-1}, \,      \bar{(\cdot)}\, \bar{2}\,\cdots\, \ov{p-1}}\\
Q_{(\cdot)\, 2\cdots\, p-1, \,     \bar{1}\, \bar{ (\cdot)}\, \cdots\, \ov{p-1}}
\quad    Q_{1\, (\cdot)\, \cdots\, p-1, \,     \bar{1}\,  \bar{(\cdot)}\, \cdots\, \ov{p-1}}
\quad \quad \quad
\cdots
\quad \quad \quad
Q_{1\, 2\, \cdots\, (\cdot)_{p-1}, \,    \bar{1}\,   \bar{(\cdot)}\, \cdots\, \ov{p-1}}\\
 \quad \quad \quad \cdots\quad \quad \quad \quad \quad \quad \quad \quad \quad \cdots \quad \quad \quad \quad \quad \,  \quad \quad \cdots \quad \quad \quad \quad \quad \quad \cdots \quad \quad \quad\\
Q_{(\cdot)\, 2\cdots\, p-1, \,     \bar{1}\,\bar{2}\, \cdots\,  \bar{ (\cdot)}_{p-1}}\quad
Q_{1\, (\cdot)\, \cdots\, p-1, \,     \bar{1}\,  \bar{2}\, \cdots\,  \bar{ (\cdot)}_{p-1}}\quad \quad \,
\cdots \quad \quad \quad
Q_{1\, 2\, \cdots\, (\cdot)_{p-1}, \,    \bar{1}\,   \bar{2}\, \cdots\,  \bar{ (\cdot)}_{p-1}}
\end{array}\right)
\end{eqnarray*}
and $S$ satisfies that for vectors $X_1,\cdots,  X_{p-1}, X$
$$
(\ov{X}_1^{tr}, \cdots, \ov{X}_{p-1}^{tr})\cdot S\cdot X=\sum_{\nu=1}^{p-1}\operatorname{div}''_{X}(\phi)_{v_1\cdots v_{p-1}, \bar{v}_1\cdots \bar{X}_\nu\cdots \bar{v}_{p-1}}+\phi_{X v_1\cdots v_{p-1}, \bar{V} \bar{v}_1\cdots \bar{X}_\nu \cdots \bar{v}_{p-1}}.
$$

To check that $\heat Q_{1\cdots p-1, \bar{1}\cdots \ov{p-1}}\ge 0$ we may extend $V$ such that the following holds:
\begin{eqnarray*}
\nabla_i V  =\frac{1}{t}\frac{\partial}{\partial z^i}, \quad &\, & \quad
\nabla_{\bar{i}} V = 0,\\
\heat V &=& -\frac{1}{t} V.
\end{eqnarray*}
Using these set of equations, (\ref{eq:lem412}), (\ref{eq:lem413}) and (\ref{eq:414}) can be simplified to
\begin{eqnarray}
 &\,&\heat (\operatorname{div}''_{V}(\phi))_{I_{p-1}, \ov{J}_{p-1}}=-\frac{1}{t}\operatorname{div}''_{ V}(\phi) +\mathcal{KB}(\operatorname{div}''_{V}(\phi)) \label{eq:53}\\
\quad  \quad &+& \sum_{\nu =1 }^{p-1}R^{\,
k\bar{l}}_{V\,\bar{j_\nu}}(\operatorname{div}''(\phi))_{ki_1\cdots
 i_{p-1},\bar{j_1}\cdots(\bar{l})_\nu\cdots\bar{j}_{p-1}}-\frac{1}{2}\operatorname{div}''_{\Ric(V)}(\phi)
 -\frac{1}{t} \operatorname{div}'( \operatorname{div}''(\phi)); \nonumber \\
&\,&\heat (\operatorname{div}'_{\ov{V}}(\phi))_{I_{p-1}, \bar{J}_{p-1}}= -\frac{1}{t}\operatorname{div}'_{\ov{V}}(\phi)+\mathcal{KB}(\operatorname{div}'_{\ov{V}}(\phi)) \label{eq:54}\\
\quad \quad &
+& \sum_{\mu =1 }^{p-1}R^{\,
k\bar{l}}_{i_\mu\,\bar{V}}(\operatorname{div}'(\phi))_{i_1\cdots(k)_\mu\cdots
 i_{p-1}, \bar{l}\bar{j_1}\cdots\bar{j}_{p-1}}-\frac{1}{2}\operatorname{div}'_{\ov{\Ric(V)}}(\phi)
  -\frac{1}{t}\operatorname{div}''( \operatorname{div}'(\phi)); \nonumber\\
&\,&\heat \phi_{V,\bar{V}}=\mathcal{KB}(\phi_{V,\bar{V}})+ R^{\,
k\bar{l}}_{V\,\bar{V}} \phi_{k, \bar{l}}+\sum_{\nu =1 }^{p-1}R^{\,
k\bar{l}}_{V\,\bar{j_\nu}} \phi_{k I_{p-1}, \bar{V} \bar{j}_1\cdots (\bar{l})_\nu\cdots \bar{j}_{p-1}}\label{eq:55}\\
\quad \quad &+&\sum_{\mu =1 }^{p-1}R^{\,
k\bar{l}}_{i_\mu\,\bar{V}} \phi_{V i_1\cdots (k)_\mu\cdots i_{p-1}, \bar{l}\bar{J}_{p-1}}-\frac{1}{2}\left(\phi_{V, \ov{\Ric(V)}}+\phi_{\Ric(V), \bar{V}}\right) \nonumber\\
\quad \quad &-& \frac{2}{t}\phi_{ V, \bar{V}}-\frac{\Lambda \phi}{t^2}-\frac{1}{t}\operatorname{div}'_{\ov{V}}(\phi)-\frac{1}{t}\operatorname{div}''_{V}(\phi).\nonumber \nonumber
\end{eqnarray}
Adding them up with that
 $$\heat \left[\operatorname{div}''  (\operatorname{div}'(\phi))+ \operatorname{div}' (\operatorname{div}''(\phi))\right] =\mathcal{KB}(\left[\operatorname{div}''  (\operatorname{div}'(\phi))+ \operatorname{div}' (\operatorname{div}''(\phi))\right]  )$$ and $\heat \Lambda \phi=\mathcal{KB}(\Lambda \phi)$ , using (\ref{eq:51})  we have that
\begin{eqnarray*}
\heat Q_{I_{p-1}, \ov{J}_{p-1}} &=& \sum_{\nu =1 }^{p-1}R^{\,
k\bar{l}}_{V\,\bar{j_\nu}} \left(\phi_{k I_{p-1}, \bar{V} \bar{j}_1\cdots (\bar{l})_\nu\cdots \bar{j}_{p-1}}+(\operatorname{div}''(\phi))_{kI_{p-1},\bar{j_1}\cdots(\bar{l})_\nu\cdots\bar{j}_{p-1}}\right)\\
\quad \quad &+&\sum_{\mu =1 }^{p-1}R^{\,
k\bar{l}}_{i_\mu\,\bar{V}} \left(\phi_{V i_1\cdots (k)_\mu\cdots i_{p-1}, \bar{l}\bar{J}_{p-1}}+(\operatorname{div}'(\phi))_{i_1\cdots(k)_\mu\cdots
 i_{p-1}, \bar{l}\ov{J}_{p-1}}\right)\\
 \quad \quad&+&R^{\,
k\bar{l}}_{V\,\bar{V}} \phi_{kI_{p-1}, \bar{l}\ov{J}_{p-1}}+\mathcal{KB}(Q)_{I_{p-1}, \ov{J}_{p-1}}-\frac{2 Q_{I_{p-1}, \ov{J}_{p-1}}}{t}.
\end{eqnarray*}
Now the nonnegativity of $\heat Q_{1\cdots(p-1), \bar{1}\cdots (\ov{p-1})}$ at $(x_0, t_0)$  can be proved in a similar way as the argument in Section 2. First observe that the part of $\mathcal{KB}(Q)_{1\cdots(p-1), \bar{1}\cdots (\ov{p-1})}$ involving only $\Ric$ is
$$
-\frac{1}{2}\sum_{i=1}^{p-1}\left(Q_{1\cdots \Ric(i)\cdots (p-1), \bar{1}\cdots (\ov{p-1})}+Q_{1\cdots (p-1), \bar{1}\cdots\ov{\Ric(i)}\cdots (\ov{p-1})}\right)
$$
which vanishes due to (\ref{eq:52}).
Hence we only need to establish the nonnegativity of
\begin{eqnarray*}J&\doteqdot &
\sum_{\mu = 1}^{p-1}\sum_{\nu =1 }^{p-1}R^{\,
k\bar{l}}_{\mu\,\,\bar{\nu}}Q_{1\cdots(k)_\mu\cdots
 (p-1), \bar{1}\cdots(\bar{l})_\nu\cdots\ov{p-1}}\\
\quad \quad &+&  \sum_{\nu =1 }^{p-1}R^{\,
k\bar{l}}_{V\,\,\bar{\nu}} \left(\phi_{k 1 \cdots (p-1), \bar{V} \bar{1}\cdots (\bar{l})_\nu\cdots (\ov{p-1})}+(\operatorname{div}''(\phi))_{k1\cdots (p-1),\bar{1}\cdots(\bar{l})_\nu\cdots(\ov{p-1})}\right)\\
\quad \quad &+&\sum_{\mu =1 }^{p-1}R^{\,
k\bar{l}}_{\mu\,\,\bar{V}} \left(\phi_{V 1\cdots (k)_\mu\cdots (p-1), \bar{l}\cdots (\ov{p-1})}+(\operatorname{div}'(\phi))_{1\cdots(k)_\mu\cdots
 (p-1), \bar{l}\bar{1}\cdots (\ov{p-1})}\right)\\
 \quad \quad&+&R^{\,
k\bar{l}}_{V\,\bar{V}} \phi_{k1\cdots (p-1), \bar{l}\bar{1}\cdots (\ov{p-1})}.
\end{eqnarray*}
The curvature operator is in $\mathcal{C}_p$ implies that the matrix
$$
\mathcal{M}_2=\left(\begin{array}{l}
R_{1\bar{1}(\cdot)\bar{(\cdot)}}\quad\quad
R_{1\bar{2}(\cdot)\bar{(\cdot)}}\quad\quad\quad
\cdots\quad\quad
R_{1\ov{p-1}(\cdot)\bar{(\cdot)}}  \quad \quad \quad R_{1\ov{V}(\cdot)\bar{(\cdot)}}\\
R_{2\bar{1}(\cdot)\bar{(\cdot)}}\quad \quad
R_{2\bar{2}(\cdot)\bar{(\cdot)}}\quad\quad\quad
\cdots\quad\quad
R_{2\ov{p-1}(\cdot)\bar{(\cdot)}}\quad \quad \quad R_{2\ov{V}(\cdot)\bar{(\cdot)}}\\
\quad\cdots\quad\quad\quad\quad \cdots\quad\quad\quad\quad \cdots\quad\quad\quad\cdots\quad \quad \quad \quad \quad \quad \cdots \\
R_{p-1\bar{1}(\cdot)\bar{(\cdot)}}\quad
R_{p-1\bar{2}(\cdot)\bar{(\cdot)}}\quad\quad
\cdots\quad\quad
R_{p-1\ov{p-1}(\cdot)\bar{(\cdot)}}\quad \quad  R_{p-1\ov{V}(\cdot)\bar{(\cdot)}}\\
R_{V\bar{1}(\cdot)\bar{(\cdot)}}\quad\quad
R_{V\bar{2}(\cdot)\bar{(\cdot)}}\quad\quad\,\,
\cdots\quad\quad
R_{V\ov{p-1}(\cdot)\bar{(\cdot)}}\quad \quad \quad  R_{V\ov{V}(\cdot)\bar{(\cdot)}}
\end{array}\right)\ge 0.
$$
The nonnegativity of $J$ follows from $\operatorname{trace}(\mathcal{M}_1 \cdot \mathcal{M}_2)\ge 0$.

We define transformations on $T'M$, $(\operatorname{div}''(\phi))^{\bar{\nu}}$, $(\operatorname{div}'(\phi))^{\mu}$, $\phi_{\ov{V}}^{\bar{\nu}}$ and $\phi_{V}^{\mu}$ by
\begin{eqnarray*}
\langle (\operatorname{div}''(\phi))^{\bar{\nu}}(X), \ov{Y}\rangle &\doteqdot& (\operatorname{div}''(\phi))_{X 1\cdots (p-1), \bar{1}\cdots (\ov{Y})_{\nu}\cdots (\ov{p-1})}, \\
\langle (\operatorname{div}'(\phi))^\mu(X), \ov{Y}\rangle &\doteqdot& (\operatorname{div}'(\phi))_{ 1\cdots (X)_\mu\cdots (p-1), \ov{Y} \bar{1}\cdots (\ov{p-1})},\\
\langle \phi_{\ov{V}}^{\bar{\nu}}(X), \ov{Y}\rangle &\doteqdot& \phi_{X 1\cdots (p-1), \ov{V} \bar{1}\cdots (\ov{Y})_\nu \cdots (\ov{p-1})},\\
\langle \phi_{V}^{\mu}(X), \ov{Y}\rangle &\doteqdot& \phi_{ V1\cdots (X)_\mu \cdots (p-1), \ov{Y}\bar{1}\cdots  (\ov{p-1})}.
\end{eqnarray*}
Then the operator $S$ defined previously can be written as $S=\oplus_{\nu =1}^{p-1} \left[(\operatorname{div}''(\phi))^{\bar{\nu}}+\phi_{\ov{V}}^{\bar{\nu}}\right]$.
If we define $Q^{\mu\bar{\nu}}$ in a similar way as $\phi^{\mu\bar{\nu}}$ then the quantity $J$ above can be expressed as
\begin{eqnarray*}
J&=&\sum_{\mu, \nu =1}^{p-1} \operatorname{trace}\left(R^{\mu\bar{\nu}}  Q^{\mu\bar{\nu}}\right)
+\sum_{\nu=1}^{p-1}\operatorname{trace} \left( R^{V \bar{\nu}}\cdot  ((\operatorname{div}''(\phi))^{\bar{\nu}}+\phi_{\ov{V}}^{\bar{\nu}})\right)\\
&\quad&+\sum_{\nu=1}^{p-1}\ov{\operatorname{trace} \left( R^{V \bar{\nu}}\cdot  ((\operatorname{div}''(\phi))^{\bar{\nu}}+\phi_{\ov{V}}^{\bar{\nu}})\right)}+\operatorname {trace}(R^{V, \bar{V}}\cdot \phi^{p, \bar{p}}).
\end{eqnarray*}
Hence one can modify the definitions of transformations $\mathcal{K}$ and $\mathcal{J}$ on $\oplus_{\mu=1}^p T'M$   in Section 2 so that $J=\operatorname{trace}(\mathcal{K}\cdot \mathcal{J})$, $\mathcal{J}$ and $\mathcal{K}$ correspond to $\mathcal{M}_1$ and $\mathcal{M}_2$ respectively.

\begin{remark}We suspect that the theorem (and later results) still holds even under the weaker assumption $\mathcal{C}_1$, even though the techniques employed here seem not be able to prove such a claim.

\end{remark}

\section{Coupled with the K\"ahler-Ricci flow.}

Now we consider  $(M^m, g(t))$  a complete solution of the K\"ahler-Ricci flow
\begin{equation}\label{eq:KRF}
\frac{\p}{\p t}g_{i\bar{j}}=-R_{i\bar{j}}.
\end{equation}
Corollary \ref{thm:p-NBC} asserts that $\mathcal{C}_p$ is an invariant
curvature condition  under the K\"ahler-Ricci flow. Now we generalize the LYH estimate to the solution of (\ref{eq:11}). Again the result is proved  for $p=1$ and $p=m$ in \cite{NT-ajm} and \cite{Ni-JDG07} respectively.

\begin{theorem}\label{main-KRFlyh}
Let $(M, g(t))$ be a complete solution to the K\"ahler-Ricci flow
(\ref{eq:KRF}). When $M$ is noncompact we assume that the curvature of $(M, g)$ is uniformly bounded.  Assume that $\phi$ is a solution to (\ref{eq:11})
 with $\phi(x, 0)$ being a positive $(p, p)$-form. Assume further that
the curvature of $(M, g(t))$ satisfies $\mathcal{C}_p$. Then for any vector field $V$ of $(1,0)$ type $\widetilde{Q}\ge 0$
for any $t>0$, where
$$
\widetilde Q=Q+\Ric (\phi)
$$
Here $Q$ is the LYH quantity defined in Section 4 and 5, which is a $(p-1,p-1)$-form valued (Hermitian) quadratic form  of $V$, $\Ric(\phi)$ is a $(p-1, p-1)$-form defined by
$$
\Ric(\phi)_{I_{p-1}, \ov{J}_{p-1}}\doteqdot g^{i\bar{l}}g^{k\bar{j}}R_{i\bar{j}}\phi_{k I_{p-1}, \bar{l}\ov{J}_{p-1}}.
$$
\end{theorem}
Note that the difference between the above result and Theorem \ref{main-lyh} is that the Laplacian operator $\Delta_{\bar{\partial}}$ is time-dependent, namely the $g_{i\bar{j}}$ and the connection  used in the definition $\bar{\partial}^*$ are evolved by the K\"ahler-Ricci flow equation.  Moreover since $\partial^*$ and $\bar{\partial}^*$ depend on changing metrics now, the quantity $Q$ is different from the static case even though they are defined by the same expression. Amazingly, the theorem asserts that the result still holds if we  add a correction term $\Ric(\phi)$.

\begin{corollary}\label{main-con2}
Let $(M , g)$, $\phi$ be as in Theorem \ref{main-KRFlyh}. Assume that $\phi$ is $d$-closed and $M$ is compact. Let $\psi=\Lambda \phi$. Then
\begin{equation}\label{mono-krf1}
\frac{d}{d t} \left(t\int_M \psi\wedge \omega_0^{m-p+1}\right)\ge 0.
\end{equation}
Here $\omega_0$ is the K\"ahler form of the initial metric.
\end{corollary}
\begin{proof}
Note that $\frac{\partial}{\partial t} \psi +\Delta_{\bar{\partial}}\psi =\Ric(\phi)$, the operators $\partial$ and $\bar{\partial}$ can be commuted with $\frac{\partial}{\partial t}$ and $\Delta_{\bar{\partial}}$. The rest is the same as the proof of Corollary \ref{main-con1}.
\end{proof}

We first start with some lemmas which are the time dependent version of Lemma \ref{help41}, \ref{help42}.
\begin{lemma} \label{helpKRF61} Let $\phi$ be a $(p, p)$-form satisfying
(\ref{eq:11}). Then under a normal coordinate,
\begin{eqnarray}
 &\,&\heat (\operatorname{div}''(\phi))_{i I_{p-1},\ov{J}_{p-1}}
=\mathcal{KB}(\operatorname{div}''_{i}(\phi)) _{I_{p-1}, \ov{J}_{p-1}}\nonumber\\
&
+& \sum_{\nu =1 }^{p-1}R_{i\,\bar{j_\nu} l\bar{k}}
(\operatorname{div}''(\phi))_{k I_{p-1}, \bar{j_1}
\cdots(\bar{l})_\nu\cdots\bar{j}_{p-1}}
-\frac{1}{2}R_{i\bar{k}}(\operatorname{div}''(\phi))_{k I_{p-1}, \ov{J}_{p-1}}
 \label{eq:lem61}\\
&
+&R_{j\bar{k}}\nabla_k \phi_{i I_{p-1},\bar{j} \ov{J}_{p-1}}
+\nabla_i R_{j\bar{k}}\phi_{k I_{p-1},\bar{j}
\ov{J}_{p-1}}
+\sum_{\mu =1}^{p-1}
\nabla_{i_\mu}R_{l\bar{k}}\phi_{i i_1\cdots(k)_\mu\cdots i_{p-1}
\bar{l} \ov{J}_{p-1}};\nonumber\\
&\,&\heat (\operatorname{div}'(\phi))_{ I_{p-1},
\bar{j}\ov{J}_{p-1}}= \mathcal{KB}(\operatorname{div}'_{\bar{j}}(\phi)) _{I_{p-1}, \ov{J}_{p-1}}\nonumber
\\
\quad \quad &
+& \sum_{\mu =1 }^{p-1}R_{i_\mu\bar{j}\, l\bar{k} }(\operatorname{div}'(
\phi))_{i_1\cdots(k)_\mu\cdots i_{p-1}, \bar{l}\ov{J}_{p-1}}
-\frac{1}{2}R_{l
\bar{j}}(\operatorname{div}'(\phi))_{ I_{p-1}, \bar{l}\ov{J}_{p-1}}
 \label{eq:lem62}\\
\quad \quad &
+&R_{l\bar{k}}\nabla_{\bar{l}}\phi_{k I_{p-1}, \bar{j}\ov{J}_{p-1}}
+\nabla_{\bar{j}}R_{l\bar{k}}\phi_{k I_{p-1}, \bar{l}\ov{J}_{p-1}}
+\sum_{\nu =1}^{p-1}\nabla_{\bar{j}_\mu}R_{l\bar{k}}\phi_{k I_{p-1},
\bar{j}\bar{j}_1\cdots(\bar{l})_\nu\cdots\bar{j}_{p-1}}.\nonumber
\end{eqnarray}
\end{lemma}
\begin{proof}
Since $\frac{\partial}{\partial t} \Gamma^h_{jl}=-g^{h\bar{q}}\nabla_{j}R_{l\bar{q}}$,
\begin{eqnarray}
&\,&\frac{\partial}{\partial t} (\sum g^{\bar{l}k}\nabla_{k}\phi_{i i_1\cdots
i_{p-1},\bar{l}\bar{j}_1\cdots\bar{j}_{p-1}})\nonumber \\
&=& R_{l\bar{k}}\nabla_{k}\phi_{i I_{p-1},
\bar{l}\ov{J}_{p-1}}+\nabla_i R_{l\bar{k}}\phi_{k I_{p-1}, \bar{l}
\ov{J}_{p-1}}+\sum_{\mu =1}^{p-1}\nabla_{i_{\mu}}R_{l\bar{k}}\phi_{
i i_1\cdots(k)_{\mu}
\cdots i_{p-1}, \bar{l}\ov{J}_{p-1}}\label{eq:help61}\\
\quad\quad &+&\left(\operatorname{div}''\left(\frac{\partial}{\partial t} \phi\right)\right)_{i I_{p-1}, \ov{J}_{p-1}},\nonumber
\end{eqnarray}
then the first equation follows from the fact that  $\Delta_{\bar{\partial}}$ is commutative with $\operatorname{div}''$ and (\ref{eq:lem41}).
The second evolution equation follows from taking the conjugation of
the first one.
\end{proof}

\begin{lemma} \label{helpKRF62}Let $\phi$ be a solution to (\ref{eq:11}). Then under the normal coordinate
\begin{eqnarray*}
&\, &\heat  (\operatorname{div}'
(\operatorname{div}''(\phi)))_{ I_{p-1}, \ov{J}_{p-1}}
= \left(\mathcal{KB}(\operatorname{div}'
(\operatorname{div}''(\phi)))\right)_{ I_{p-1}, \ov{J}_{p-1}}
+\mathcal{E}(\phi)_{I_{p-1},\ov{J}_{p-1}}; \\
&\,&\heat  (\operatorname{div}''
(\operatorname{div}'(\phi)))_{ I_{p-1}, \ov{J}_{p-1}}= \left(\mathcal{KB}(
\operatorname{div}''(\operatorname{div}'(\phi)))\right)_{ I_{p-1}, \ov{J}_{p-1}}+\mathcal{E}(\phi)_{I_{p-1},\ov{J}_{p-1}},
\end{eqnarray*}
where
\begin{eqnarray}
&\, &\mathcal{E}(\phi)_{I_{p-1}, \ov{J}_{p-1}}\doteqdot
R_{j\bar{i}}
\left(\nabla_{i}(\operatorname{div}'(\phi))_{I_{p-1},\bar{j}\ov{J}_
{p-1}}+\nabla_{\bar{j}}
(\operatorname{div}''(\phi))_{i I_{p-1},\ov{J}_{p-1}}
\right)\nonumber\\
\quad\quad
&+&\sum_{\mu=1}^{p-1}\nabla_{j}R_{i_{\mu}\bar{k}}
(\operatorname{div}'(\phi))_{i_1\cdots(k)_{\mu}\cdots
i_{p-1}, \bar{j}\ov{J}_{p-1}}
+\sum_{\nu=1}^{p-1}\nabla_{\bar{i}}R_{l\bar{j}_{\nu}}
(\operatorname{div}''(\phi))_{i I_{p-1},
\bar{j}_1\cdots(\bar{l})_{\nu}\cdots\bar{j}_{p-1}}\nonumber\\
\quad\quad
&+&\sum_{\mu=1}^{p-1}R_{l\bar{k}}R_{j\bar{i}i_{\mu}\bar{l}}
\phi_{ii_1\cdots(k)_{\mu}\cdots i_{p-1}, \bar{j}\ov{J}_{p-1}}
+\sum_{\nu=1}^{p-1}R_{l\bar{k}}R_{j\bar{i}k\bar{j}_{\nu}}
\phi_{i I_{p-1},
\bar{j}\bar{j}_1\cdots(\bar{l})_{\nu}\cdots\bar{j}_{p-1}}\nonumber \\
\quad\quad
&+&\D R_{j\bar{i}}
\phi_{i I_{p-1}, \bar{j}\ov{J}_{p-1}}
+\nabla_{k}R_{j\bar{i}}\nabla_{\bar{k}}
\phi_{i I_{p-1}, \bar{j}\ov{J}_{p-1}}
+\nabla_{\bar{k}}R_{j\bar{i}}\nabla_{k}
\phi_{i I_{p-1}, \bar{j}\ov{J}_{p-1}}\nonumber\\
\quad\quad
&+&R_{j\bar{k}}R_{k\bar{i}}
\phi_{i I_{p-1}, \bar{j}\ov{J}_{p-1}}
-R_{k\bar{l}}R_{l\bar{k}j\bar{i}}
\phi_{i I_{p-1}, \bar{j}\ov{J}_{p-1}}.\nonumber
\end{eqnarray}
\end{lemma}
\begin{proof}
 \begin{eqnarray*}
&\, &\frac{\partial}{\partial t} (g^{i\bar{j}}\nabla_{\bar{j}}
(\operatorname{div}''(\phi))_{i I_{p-1}, \ov{J}_{p-1}})\\
\quad\quad
&=&R_{j\bar{i}}\nabla_{\bar{j}}
(\operatorname{div}''(\phi))_{i I_{p-1}, \ov{J}_{p-1}}
+\sum_{\nu=1}^{p-1}\nabla_{\bar{i}}R_{l\bar{j}_{\nu}}
(\operatorname{div}''(\phi))_{i I_{p-1},
\bar{j}_1\cdots(\bar{l})_{\nu}\cdots\bar{j}_{p-1}}\\
\quad\quad
&+&\nabla_{\bar{i}}(\frac{\p}{\p t}(
(\operatorname{div}''(\phi))_{i I_{p-1}, \ov{J}_{p-1}})
.
 \end{eqnarray*}
Now we plug in (\ref{eq:help61}). Applying the commutator formula, the 2nd-Bianchi identity   and Lemma
\ref{help42}, we get the first evolution equation. The second one
follows from the first by taking the conjugation.
\end{proof}

The next lemma is on $\Ric(\phi)$. The proof is via straight forward computation.
\begin{lemma} For $\phi$, a solution (\ref{eq:11}), under a normal coordinate,
 \begin{eqnarray}
 &\,&\heat  (R_{j\bar{i}}\phi_{i I_{p-1},
\bar{j}\ov{J}_{p-1}})
=\sum_{\mu = 1}^{p-1}\sum_{\nu =1 }^{p-1}R^{\,
k\bar{l}}_{i_\mu\,\,\bar{j_\nu}}R_{j\bar{i}}
\phi_{ii_1\cdots(k)_{\mu}\cdots i_{p-1}\bar{j}\bar{j}_1\cdots(\bar{l})_{\nu}
\cdots\bar{j}_{p-1}}\label{eq:lem65}\\
&\quad&
-\frac{1}{2}\left(\sum_{\nu =1}^{p-1}R^{\bar{l}}_{\,\,
\bar{j}_\nu}(R_{j\bar{i}}\phi_{i I_{p-1},\bar{j}\bar{j_1}\cdots(\bar{l})_{\nu}
\cdots\bar{j}_{p-1}}+\sum_{\mu =1
}^{p-1}R_{i_\mu}^{\,\,
k}R_{j\bar{i}}\phi_{ii_1\cdots(k)_\mu\cdots i_{p-1},
\bar{j}\ov{J}_{p-1}}\right)\nonumber\\
&\quad&-\left(\sum_{\mu=1}^{p-1}R_{l\bar{k}}R_{j\bar{i}i_{\mu}\bar{l}}
\phi_{ii_1\cdots(k)_{\mu}\cdots i_{p-1}, \bar{j}\ov{J}_{p-1}}
+\sum_{\nu=1}^{p-1}R_{l\bar{k}}R_{j\bar{i}k\bar{j}_{\nu}}
\phi_{i I_{p-1},
\bar{j}\bar{j}_1\cdots(\bar{l})_{\nu}\cdots\bar{j}_{p-1}}\right)\nonumber \\
&\quad&-\left(\nabla_{k}R_{j\bar{i}}\nabla_{\bar{k}}
\phi_{i I_{p-1}， \bar{j}\ov{J}_{p-1}}
+\nabla_{\bar{k}}R_{j\bar{i}}\nabla_{k}
\phi_{i I_{p-1}, \bar{j}\ov{J}_{p-1}}\right)
+2R_{k\bar{l}}R_{l\bar{k}j\bar{i}}\phi_{i I_{p-1},
\bar{j}\ov{J}_{p-1}}\nonumber.
 \end{eqnarray}

\end{lemma}

Adapting the notation from Section 4, Lemma \ref{helpKRF61} implies the following set of formulae.

\begin{eqnarray}
 &\,&\heat (\operatorname{div}''_{V}(\phi))_{I_{p-1},
\ov{J}_{p-1}}=\operatorname{div}''_{\heat V}(\phi)
+\mathcal{KB}(\operatorname{div}''_{V}(\phi)) \label{eq:68}\\
\quad  \quad &+& \sum_{\nu =1 }^{p-1}R^{\,
k\bar{l}}_{V\,\bar{j_\nu}}(\operatorname{div}''(\phi))_{ki_1\cdots
 i_{p-1},\bar{j_1}\cdots(\bar{l})_\nu\cdots\bar{j}_{p-1}}-\frac{1}{2}
\operatorname{div}''_{\Ric(V)}(\phi)
  \nonumber\\
\quad\quad&+&
R_{j\bar{k}}\nabla_{k}\phi_{Vi_1\cdots i_{p-1},\bar{j}\bar{j}_1
\cdots\bar{j}_{p-1}}+\nabla_VR_{j\bar{i}}\phi_{ii_1\cdots i_{p-1},
\bar{j}\bar{j}_1\cdots\bar{j}_{p-1}}\nonumber\\
 \quad \quad &+& \sum_{\mu=1}^{p-1}\nabla_{i_{\mu}}R_{j\bar{k}}\phi_{Vi_1\cdots
(k)_{\mu}\cdots i_{p-1},\bar{j}\bar{j}_1\cdots\bar{j}_{p-1}}-
g^{i\bar{j}}\left((\nabla_{i}
\operatorname{div}''(\phi))_{\nabla_{\bar{j}}V}+(\nabla_{\bar{j}}
\operatorname{div}''(\phi))_{\nabla_{i}V})\right); \nonumber \\
&\,&\heat (\operatorname{div}'_{\ov{V}}(\phi))_{I_{p-1}, \bar{J}_{p-1}}=
\operatorname{div}'_{\heat
\ov{V}}(\phi)+\mathcal{KB}(\operatorname{div}'_{\ov{V}}(\phi))
\label{eq:69}\\
\quad \quad &
+& \sum_{\mu =1 }^{p-1}R^{\,
k\bar{l}}_{i_\mu \,\bar{V}}(\operatorname{div}'(\phi))_{i_1\cdots(k)_\mu\cdots
 i_{p-1},
\bar{l}\bar{j_1}\cdots\bar{j}_{p-1}}-\frac{1}{2}\operatorname{div}'_{\ov{\Ric(V)
}}(\phi) \nonumber\\
\quad\quad &+&
R_{k\bar{i}}\nabla_{\bar{k}}\phi_{ii_1\cdots i_{p-1},
\ov{V}\bar{j}_1\cdots\bar{j}_{p-1}}
+\nabla_{\ov{V}}R_{j\bar{i}}\phi_{ii_1\cdots i_{p-1}, \bar{j}\bar{j}_1\cdots
\bar{j}_{p-1}}\nonumber\\
\quad\quad &+& \sum_{\nu=1}^{p-1}\nabla_{\bar{j}_{\nu}}R_{j\bar{i}}
\phi_{ii_1\cdots i_{p-1},
\ov{V}\bar{j}_1\cdots(\bar{j})_{\nu}\cdots\bar{j}_{p-1}}
-g^{i\bar{j}}\left((\nabla_{i}
\operatorname{div}'(\phi))_{\nabla_{\bar{j}}\ov{V}}+(\nabla_{\bar{j}}
\operatorname{div}'(\phi))_{\nabla_{i}\ov{V}}\right). \nonumber
\end{eqnarray}
For $\Lambda\phi$, we have the following evolution equation.
\begin{equation}\label{eq:610}
(\frac{\p}{\p t}-\D)(\Lambda\phi)_{i_1\cdots i_{p-1},
\bar{j}_1\cdots\bar{j}_{p-1}}=\mathcal{KB}(\Lambda\phi)_{i_1\cdots i_{p-1},
\bar{j}_1\cdots\bar{j}_{p-1}}+R_{j\bar{i}}\phi_{ii_1\cdots i_{p-1},
\bar{j}\bar{j}_1\cdots\bar{j}_{p-1}}.
\end{equation}

\section{ A family of LYH estimates for  the K\"ahler-Ricci flow under the condition $\mathcal{C}_p$}

Let $(M, g(t))$ be a complete solution to the K\"ahler-Ricci flow. Assume
further that the curvature operator satisfies $\mathcal{C}_p$.
Let
\begin{eqnarray*}
\mathcal{M}_{\abb}=\D R_{\abb}+R_{\abb\gbd}R_{\delta\bar{\g}}+\frac{R_{\abb}}{t}, \quad
P_{\abb\g}=\nabla_{\g}R_{\abb}, \quad P_{\abb\bar{\g}}=\nabla_{\bar{\g}}R_{\abb}.
\end{eqnarray*}
Also let $P_{\bar{\beta} \alpha \gamma}=\nabla_\gamma R_{\bar{\beta}\alpha}$, $P_{\bar{\beta}\alpha \bar{\gamma}}=\nabla_{\bar{\gamma}}R_{\bar{\beta}\alpha}$. Clearly $P_{\abb\g}=P_{\bar{\beta} \alpha \gamma}$ and $P_{\abb\bar{\g}}=P_{\bar{\beta}\alpha \bar{\gamma}}$.
The second Bianchi identity implies that
\begin{equation*}
P_{\abb\g}=P_{\g\bb\a}, \quad \ov{P_{\abb\g}}=P_{\b\ba\bar{\g}}=P_{\ba\b\bar{\g}}.
\end{equation*}

\begin{theorem}\label{LYH} Let $(M, g(t))$ be a complete solution to the K\"ahler-Ricci flow satisfying the condition $\mathcal{C}_p$ on $M \times [0, T]$. When $M$ is noncompact we assume that the curvature of $(M, g(t))$ is bounded on $M \times [0, T]$. Then
 for any $\wedge^{1, 1}$-vector $U$ which can be written as
$U=\sum_{i=1}^{p-1} X_i\wedge \bar{Y}_i+W\wedge \bar{V}$, for $(1, 0)$-type vectors $X_i, Y_i, W, V$, the Hermitian bilinear form  $\mathcal{Q}$ defined as
\begin{equation}\label{eq:71}
\mathcal{Q}(U\oplus W)\doteqdot
\mathcal{M}_{\abb}W^{\a}W^{\bb}+P_{\abb\g}\bar{U}^{\bb\g}W^{\a}
+P_{\abb\bar{\g}}U^{\a\bar{\g}}W^{\bb}
+R_{\abb\gbd}U^{\a\bb}\bar{U}^{\delbar\g}
\end{equation}
satisfies that $\mathcal{Q}\ge 0$ for any $t>0$. Moreover, if the equality ever occurs for some $t>0$, the universal cover of $(M, g(t))$ must be a gradient expanding K\"ahler-Ricci soliton.
\end{theorem}
The theorem says that for any vector $W$, viewing  $\mathcal{Q}$ as a Hermitian quadratic/bilinear form on $\wedge^{1,1}$ space, it  also satisfies $\mathcal{C}_p$, but only for the $\wedge^{1,1}$ vector $U$ with the form $U=\sum_{i=1}^p X_i\wedge \ov{Y}_i$ with $X_p=W$.
If we define $P: T'M \to \wedge^{1,1}$ by the equation $\langle P(W), \ov{U}\rangle =P_{\abb\g}\ov{U}^{\bb\g}W^{\a}$, the LYH expression can be written as, by abusing the notation with $\mathcal{Q}$ denoting also the Hermitian symmetric transformation,
\begin{equation}\label{def-krf}
\langle \mathcal{Q}(U), \ov{U}\rangle =\langle \mathcal{M}(W), \ov{W}\rangle +2 Re(\langle P(W), \ov{U}\rangle )+\langle \Rm(U), \ov{U}\rangle.
\end{equation}

\begin{remark}
 When $p=1$, The inequality (\ref{eq:71}) recovers  the LYH inequality of Cao \cite{Cao}. When $p>1$, $\mathcal{Q}$ can be written as
\begin{equation*}
\mathcal{Q}=Z_{\abb}W^{\a}W^{\bb}
+(P_{\abb\g}+R_{\a\bar{V}\g\bb})\ov{\tilde{U}}^{\bb\g}
W^{\a}+(P_{\abb\bar{\g}}+R_{V\bb\a\bar{\g}})\tilde{U}^{\a\bar{\g}}W^{\bb}
+R_{\abb\gbd}\tilde{U}^{\abb}\ov{\tilde{U}}^{\delbar\g},
\end{equation*}
with $\tilde{U}=\sum_{i=1}^{p-1}X_i\wedge \bar{Y}_i$ and
\begin{equation}\label{defZ}
Z_{\abb}\doteqdot \mathcal{M}_{\abb}+P_{\abb\g}V^{\g}+P_{\abb\bar{\g}}V^{\bar{\g}}
+R_{\abb\gbd}V^{\g}V^{\delbar}.
\end{equation}
Equivalently, if we write the above as $\langle Z(W\wedge \ov{V}), \ov{W\wedge \ov{V}}\rangle$,
$$
\mathcal{Q}=\langle Z(W\wedge \ov{V}), \ov{W\wedge \ov{V}}\rangle +2Re\left( \widetilde{P}(W\wedge \ov{V}), \ov{\tilde{U}}\rangle\right)+\langle \Rm (\tilde U), \ov{\tilde{U}}\rangle.
$$
Here $\widetilde{P}$ is defined as $ \widetilde{P}(W\wedge \ov{V})=P(W)+\Rm(W\wedge \ov{V})$.
Note that Hamilton in \cite{richard-harnack} proved that under the stronger assumption that the curvature operator $\Rm\ge0$,   $Q(U\oplus W)\ge 0$ for any $\wedge^2$-vector $U$. For $p$ sufficiently large $\mathcal{C}_p$ is equivalent to $\Rm \ge 0$ and by taking $U=U_1+W\wedge \ov{Y}_p$ with $U_1=U-W\wedge \ov{Y}_p$, one can see that the above result implies Hamilton's result.  Hence Theorem \ref{LYH} interpolates between Cao's result and Hamilton's result for the K\"ahler-Ricci flow. Please also see \cite{CC} for an interpretation via the space time consideration. In  a later section we shall prove another set of estimates which generalize Hamilton's estimate for the Ricci flow on Riemannian manifolds.
\end{remark}

One can easily get the following lemma through direct calculation, which can also be derived from Lemma 4.3, 4.4 of \cite{richard-harnack}.

\begin{lemma}\label{help71}
\begin{eqnarray}
&\, &\heat \mathcal{M}_{\a\bb}=R_{\abb\gbd}\mathcal{M}_{\delta\bar{\g}}
-\frac{1}{2}(R_{\a\bar{\eta}}\mathcal{M}_{\eta\bb}+R_{\eta\bb}\mathcal{M}_{\a\bar{\eta}})
+R_{\abb\gbd}R_{\delta\bar{\xi}}R_{\xi\bar{\g}}\label{lem71}\\
&\quad&+R_{\delta\bar{\g}}(\nabla_{\g}P_{\abb\delbar}
+\nabla_{\delbar}P_{\abb\g})+P_{\a\bar{\xi}\g}P_{\xi\bb\bar{\g}}
-P_{\a\bar{\xi}\bar{\g}}P_{\xi\bb\g}
-\frac{R_{\abb}}{t^2},
\nonumber\\
&\, &\heat P_{\abb\g}=R_{\abb\xi\bar{\eta}}P_{\eta\bar{\xi}\g}
+R_{\xi\bb\gbd}P_{\a\bar{\xi}\delta}
-R_{\a\bar{\xi}\gbd}P_{\xi\bb\delta}\label{lem72}\\
&\quad&
-\frac{1}{2}\left(R_{\a\bar{\xi}}P_{\xi\bb\g}
+R_{\xi\bb}P_{\a\bar{\xi}\g}+R_{\g\bar{\xi}}P_{\abb\xi}\right)
+\nabla_{\g}R_{\abb\xi\bar{\eta}}R_{\eta\bar{\xi}}.\nonumber
\end{eqnarray}

\end{lemma}

By taking the conjugation of (\ref{lem72}), we have
 \begin{eqnarray}\label{eq:74}
 &\, &\heat P_{\abb\bar{\g}}
=R_{\abb\xi\bar{\eta}}P_{\eta\bar{\xi}\bar{\g}}
+R_{\a\bar{\xi}\delta\bar{\g}}P_{\xi\bb\delbar}
-R_{\xi\bb\delta\bar{\g}}P_{\a\bar{\xi}\delbar}\label{lem}\\
&\quad&-\frac{1}{2}\left(R_{\a\bar{\xi}}P_{\xi\bb\bar{\g}}
+R_{\xi\bb}P_{\a\bar{\xi}\bar{\g}}+R_{\xi\bar{\g}}P_{\abb\bar{\xi}}\right)
+\nabla_{\bar{\g}}R_{\abb\xi\bar{\eta}}R_{\eta\bar{\xi}}\nonumber.
 \end{eqnarray}
The evolution equation for the curvature tensor is (see for example \cite{Bando})
\begin{eqnarray}\label{eq:75}
 &\, &\heat R_{\abb\gbd}
=R_{\abb\xi\bar{\eta}}R_{\eta\bar{\xi}\gbd}
-R_{\a\bar{\xi}\g\bar{\eta}}R_{\xi\bb\eta\delbar}
+R_{\a\delbar\xi\bar{\eta}}R_{\eta\bar{\xi}\g\bb}\label{eq:}\\
&\quad&-\frac{1}{2}\left(
R_{\a\bar{\xi}}R_{\xi\bb\gbd}+R_{\xi\bb}R_{\a\bar{\xi}\gbd}
+R_{\g\bar{\xi}}R_{\abb\xi\delbar}+R_{\xi\delbar}R_{\abb\g\bar{\xi}}
\right)\nonumber.
\end{eqnarray}

Now we begin to prove the theorem.  We
assume that the curvature of $(M, g(t))$ satisfies $\mathcal{C}_p$. One can adapt the perturbation argument as \cite{richard-harnack} if $\Rm$ does not have strictly $p$-positive bisectional curvature. Hence when manifold is compact without the loss of generality we may assume that $\Rm$ has strictly $p$-positive bisectional curvature.   Then it is clear that
when $t$ is small $\mathcal{Q}$ is positive, since
the bisectional curvature is strictly $p$-positive and $\mathcal{M}_{\abb}$ has a term
$\frac{R_{\abb}}{t}$. We claim $\mathcal{Q}\ge 0$ for all time. If it fails to hold, there
is a first time $t_0$, a point $x_0$, and
$U\in \Lambda^{1, 1}T_{x_0}M, W\in \Lambda^{1,0}$ such that
$\mathcal{Q}(U\oplus W)=0$, and for any $t\le t_0, x\in M$, $(1, 1)$-vector
$\hat{U}\in \Lambda^{1, 1}T_{x}M$ and $(1, 0)$-vector
$\hat{W}\in T_xM$, $\mathcal{Q}(\hat{U}\oplus\hat{W})\ge 0$.
 We extend $U$ and $W$ in space-time
at $(x_0, t_0)$ in the following way:
\begin{eqnarray*}
\heat U^{\gbd}
=\frac{1}{2}\left(R^{\g}_{\a}U^{\a\delbar}
+R^{\delbar}_{\bar{\a}}U^{\g\bar{\a}}\right),\quad &\, & \quad \heat W^{\a}=\frac{1}{2}R_{\beta}^{\a}W^{\beta}
+\frac{1}{t}W^{\a}
,\\
\nabla_{s}U^{\gbd}
=R_{s}^{\g}W^{\delbar}+\frac{1}{t}g^{\g}_{s}W^{\delbar},
\quad \quad \nabla_{\bar{s}}U^{\gbd}=0,\quad &\,& \quad \nabla_{\g}W^{\a}=\nabla_{\bar{\g}}W^{\a}=0
.
\end{eqnarray*}
Here $R^{\a}_\beta, R^{\bar{\gamma}}_{\bar{\delta}}$ are the associated tensors obtained by  raising the indices on the Ricci tensor. These sets of equations are the same as  those of \cite{richard-harnack} in disguise.
As in \cite{richard-harnack}, it suffices
to check that at the point $(x_0, t_0)$, $\heat \mathcal{Q}\ge 0$.
Using the above equations and equations (\ref{lem71}),
(\ref{lem72}), (\ref{eq:74}), (\ref{eq:75}), a lengthy but straight-forward computation shows that
\begin{eqnarray*}
\heat \mathcal{Q}&=&R_{\abb\gbd}M_{\delta\bar{\g}}W^{\a}W^{\bb}
+R_{\abb\xi\bar{\eta}}P_{\eta\bar{\xi}\g}\bar{U}^{\bb\g}W^{\a}
+R_{\abb\xi\bar{\eta}}P_{\eta\bar{\xi}\bar{\g}}
U^{\a\bar{\g}}W^{\bb}
\\
&\quad&+R_{\a\delbar\xi\bar{\eta}}R_{\eta\bar{\xi}\g\bb}U^{\abb}
\bar{U}^{\delbar\g}\\
&
-&\left(P_{\a\bar{\xi}\bar{\g}}P_{\xi\bb\g}W^{\a}W^{\bb}
+R_{\a\bar{\xi}\gbd}P_{\xi\bb\delta}\bar{U}^{\bb\g}W^{\a}\right.\\
&\quad& \left.
+R_{\xi\bb\delta\bar{\g}}P_{\a\bar{\xi}\delbar}U^{\a\bar{\g}}W^{\bb}
+R_{\a\bar{\xi}\g\bar{\eta}}R_{\xi\bb\eta\delbar}U^{\abb}
\bar{U}^{\delbar\g}\right)\\
&\quad& +(P_{\a\bar{\xi}\g}W^{\a}+R_{\a\bar{\xi}\gbd}U^{\a\delbar})
(P_{\xi\bb\bar{\g}}W^{\bb}+R_{\xi\bb\delta\bar{\g}}\bar{U}^{\bb\delta}).
\end{eqnarray*}
The above computation can also be derived using Lemma 4.5 of \cite{richard-harnack}.
In the following, $X_p=W, Y_p=V$. To prove
$\heat \mathcal{Q}\ge 0$ it is enough to show that the nonnegativity of
\begin{eqnarray*}
\mathcal{J}
&\doteqdot& R_{X_p\bar{X}_p\gbd}Z_{\delta\bar{\g}}
+\sum_{\nu=1}^{p-1}R_{X_p\bar{X}_{\nu}\gbd}(P_{\delta\bar{\g} Y_{\nu}}
+R_{\delta\bar{\g}Y_{\nu}\bar{Y}_p})
+\sum_{\mu=1}^{p-1}R_{X_{\mu}\bar{X}_p\gbd}(P_{\delta\bar{\g}\bar{Y}_{\mu}}
+R_{\delta\bar{\g}Y_p\bar{Y}_{\mu}})\\
\quad\quad&
+&\sum_{\mu, \nu=1}^{p-1}R_{X_{\mu}\bar{X}_{\nu}\gbd}R_{\delta\bar{\g}
Y_{\nu}\bar{Y}_{\mu}}
+|P_{X_p\bb\a}+\sum_{\mu=1}^pR_{X_{\mu}\bar{Y}_{\mu}\abb}|^2\\
\quad\quad&
-&\left(|P_{X_p\bar{\g}\delbar}|^2+\sum_{\nu=1}^p
R_{X_p\bar{\g}Y_{\nu}\delbar}P_{\g\bar{X}_{\nu}\delta}
+\sum_{\mu=1}^p R_{\g\bar{X}_p\delta\bar{Y}_{\mu}}
P_{X_{\mu}\bar{\g}\delbar}+
\sum_{\mu, \nu=1}^p R_{X_{\mu}\bar{\g}Y_{\nu}\delbar}
R_{\g\bar{Y}_{\mu}\delta\bar{X}_{\nu}}\right),
\end{eqnarray*}
where we have respectively replaced $U$ and $W$ by
$\sum_{i=1}^p X_{i}\wedge \bar{Y}_i$ and $X_p$.

Let
\begin{eqnarray*}
A_1=\left(\begin{array}{l}
R_{X_1\bar{X}_1(\cdot)\bar{(\cdot)}}\qquad
R_{X_1\bar{X}_2(\cdot)\bar{(\cdot)}}\quad\quad
\cdots\quad\quad
R_{X_1\bar{X}_{p}(\cdot)\bar{(\cdot)}}
\\
R_{X_2\bar{X}_1(\cdot)\bar{(\cdot)}}\qquad
R_{X_2\bar{X}_2(\cdot)\bar{(\cdot)}}\quad\quad
\cdots\quad\quad
R_{X_2\bar{X}_{p}(\cdot)\bar{(\cdot)}}\\
\quad\quad\cdots\quad\quad\quad\quad\cdots\qquad\qquad\quad\cdots\quad\quad
\quad\cdots\\
R_{X_p\bar{X}_1(\cdot)\bar{(\cdot)}}\qquad
R_{X_p\bar{X}_2(\cdot)\bar{(\cdot)}}\quad\quad
\cdots\quad\quad
R_{X_p\bar{X}_p(\cdot)\bar{(\cdot)}}
\end{array}\right),
\end{eqnarray*}
\begin{eqnarray*}
A_2=\left(\begin{array}{l}
R_{Y_1\bar{Y}_1(\cdot)\bar{(\cdot)}}\quad\quad
R_{Y_1\bar{Y}_2(\cdot)\bar{(\cdot)}}\quad\quad
\cdots\quad
R_{Y_1\bar{Y}_{p-1}(\cdot)\bar{(\cdot)}}\quad
E^{1\bar{p}}_{(\cdot)\bar{(\cdot)}}\\
R_{Y_2\bar{Y}_1(\cdot)\bar{(\cdot)}}\quad\quad
R_{Y_2\bar{Y}_2(\cdot)\bar{(\cdot)}}\quad\quad
\cdots\quad
R_{Y_2\bar{Y}_{p-1}(\cdot)\bar{(\cdot)}}\quad
E^{2\bar{p}}_{(\cdot)\bar{(\cdot)}}\\
\quad\quad\cdots\qquad\qquad\cdots\qquad\qquad\cdots\quad\quad
\quad\cdots\quad\quad\quad\cdots\\
R_{Y_{p-1}\bar{Y}_1(\cdot)\bar{(\cdot)}}\quad
R_{Y_{p-1}\bar{Y}_2(\cdot)\bar{(\cdot)}}\ \quad\cdots\
R_{Y_{p-1}\bar{Y}_{p-1}(\cdot)\bar{(\cdot)}}\
E^{p-1\bar{p}}_{(\cdot)\bar{(\cdot)}}\\
\ \ov{E^{1\bar{p}}}^{tr}_{(\cdot)\bar{(\cdot)}}\qquad\quad
\ov{E^{2\bar{p}}}^{tr}_{(\cdot)\bar{(\cdot)}}\quad\quad\cdots\quad\quad
\ov{E^{(p-1)\bar{p}}}^{tr}_{(\cdot)\bar{(\cdot)}}\quad
Z_{(\cdot)\bar{(\cdot)}}
\end{array}\right),
\end{eqnarray*}
and
\begin{eqnarray*}
A_3=\left(\begin{array}{l}
R_{X_1(\cdot)Y_1(\cdot)}\quad
R_{X_1(\cdot)Y_2(\cdot)}\quad\quad\cdots\quad
R_{X_1(\cdot)Y_{p-1}(\cdot)}\quad
R_{X_1(\cdot)Y_p(\cdot)}+P_{X_1(\cdot)(\cdot)}\\
R_{X_2(\cdot)Y_1(\cdot)}\quad
R_{X_2(\cdot)Y_2(\cdot)}\quad\quad\cdots\quad
R_{X_2(\cdot)Y_{p-1}(\cdot)}\quad
R_{X_2(\cdot)Y_p(\cdot)}+P_{X_2(\cdot)(\cdot)}\\
\quad\quad\cdots\quad\quad\qquad\cdots\qquad\quad\ \cdots\quad
\quad\quad\cdots\qquad\qquad\qquad\cdots\\
R_{X_p(\cdot)Y_1(\cdot)}\quad
R_{X_p(\cdot)Y_2(\cdot)}\quad\quad\cdots\quad
R_{X_p(\cdot)Y_{p-1}(\cdot)}\quad
R_{X_p(\cdot)Y_p(\cdot)}+P_{X_p(\cdot)(\cdot)}
\end{array}\right),
\end{eqnarray*}
where the tensor $Z_{\gbd}$ is defined as in (\ref{defZ}),
\begin{equation*}
E^{\mu\bar{p}}_{\ \gbd}=R_{Y_\mu\bar{Y}_p\gbd}+P_{\gbd Y_{\mu}}
, \quad \quad \ov{E^{\mu\bar{p}}}^{tr}_{\ \gbd}
=R_{Y_p\bar{Y}_{\mu}\gbd}+P_{\gbd\bar{Y}_{\mu}},
1\le \mu\le p-1.
\end{equation*}
Note $A_1\ge 0$ since $\Rm \in \mathcal{C}_p$, $A_2\ge 0$ since $\mathcal{Q}\ge 0$.
Now $\mathcal{J}$ can be written as
\begin{equation*}
\mathcal{J}=\operatorname{trace}(A_1\cdot A_2)+|P_{X_p\bb\a}+\sum_{\mu=1}^p
R_{X_{\mu}\bar{Y}_{\mu}\abb}|^2-\operatorname{trace}(A_3\cdot \bar{A}_3).
\end{equation*}

Since $\mathcal{Q}(U\oplus W)$ achieves the minimum at $(U, W)$ at time
$(x_0, t_0)$, then the second variation
\begin{equation*}
\frac{\p^2}{\p s^2}|_{s=0}\mathcal{Q}(U(s)\oplus W(s))\ge 0,
\end{equation*}
where $W(s)=W+sW_p, \
U(s)=\sum_{\mu=1}^{p}(X_{\mu}+sW_{\mu})\wedge
\ov{(Y_{\mu}+sV_{\mu})}$ for any $(1, 0)$-type vectors
$W_{\mu}, V_{\mu}\in T^{1, 0}_{x_0}M$.

Through calculation, $\frac{\p^2}{\p s^2}|_{s=0}\mathcal{Q}(U(s)\oplus W(s))\ge 0$ implies that
\begin{eqnarray}
&\, &\sum_{\mu, \nu=1}^p
R_{Y_{\mu}\bar{Y}_{\nu}\abb}W_{\nu}^{\a}W_{\mu}^{\bb}
+\sum_{\mu=1}^p (P_{\abb Y_{\mu}}W_p^{\a}W_{\mu}^{\bb}
+P_{\abb\bar{Y}_{\mu}}W_{\mu}^{\a}W_p^{\bb})+M_{\abb}W_p^{\a}W_{p}^{\bb}
\label{eq:77}\\
\quad
&+&\sum_{\mu=1}^p (P_{X_p\abb}V_{\mu}^{\a}W_{\mu}^{\bb}
+P_{\a\bar{X}_p\bb}W_{\mu}^{\a}V_{\mu}^{\bb})
+\sum_{\mu, \nu=1}^p(R_{Y_{\mu}\bar{X}_{\mu}\abb}W_{\nu}^{\a}V_{\nu}^{\bb}
+R_{X_{\mu}\bar{Y}_{\mu}\abb}V_{\nu}^{\a}W_{\nu}^{\bb})\nonumber\\
\quad
&+&\sum_{\mu, \nu=1}^p (
R_{X_{\mu}\ba Y_{\nu}\bb}W_{\nu}^{\ba}V_{\mu}^{\bb}
+R_{\a\bar{X}_{\mu}\b\bar{Y}_{\nu}}W_{\nu}^{\a}V_{\mu}^{\b})
+\sum_{\mu=1}^p(P_{X_{\mu}\ba\bb}W_{p}^{\ba}V_{\mu}^{\bb}+
P_{\a\bar{X}_{\mu}\b}W_p^{\a}V_{\mu}^{\b})\nonumber\\
\quad
&+&\sum_{\mu, \nu=1}^p
R_{X_{\mu}\bar{X}_{\nu}\abb}V_{\nu}^{\a}V_{\mu}^{\bb}\ge 0.\nonumber
\end{eqnarray}
By letting $\mathcal{X}=\left(\begin{array}{l} W_1\\ \vdots\\
 W_p\end{array}\right),\,  \mathcal{Y}=\left(\begin{array}{l} V_1\\ \vdots\\
 V_p\end{array}\right)$,  one can deduce from (\ref{eq:77}) that
\begin{equation*}
\ov{\mathcal{X}}^{tr}A_2\mathcal{X}+\ov{\mathcal{Y}}^{tr}A_1\mathcal{Y}+2Re(\mathcal{Y}^{tr}\bar{A}_3\mathcal{X}
+\ov{\mathcal{Y}}^{tr}A_4\mathcal{X})\ge
0,
\end{equation*}
where
\begin{eqnarray*}
A_4=\left(\begin{array}{l}
G\quad
0\quad\quad
\cdots\quad\quad
0
\\
 0\quad
G\quad\quad
\cdots\quad\quad
0\\
0\quad
0\quad\quad\cdots\quad\quad
0\\
0\quad
0\quad\quad
\cdots\quad\quad
G
\end{array}\right),
\end{eqnarray*}
where $G_{\abb}=
P_{\a\bar{X}_p\bb}
+\sum_{\mu=1}^p
R_{Y_{\mu}\bar{X}_{\mu}\abb}$.

If we regard $T^{1, 0}_{x_0}M$ as $\mathbb{C}^m$, then $\mathcal{X}, \mathcal{Y} \in
\mathbb{C}^{pm}$. By Lemma \ref{lmlyh} below, which is due to Mok according to \cite{Cao} (see also Lemma 2.86 of \cite{Chowetc}), we have
\begin{equation}\label{eq:79}
\operatorname{trace}(A_2\cdot A_1)\ge \operatorname{trace}(A_3\cdot \bar{A}_3).
\end{equation}
The inequality (\ref{eq:79}) implies that $\mathcal{J}\ge 0.$ We then complete the proof of
Theorem \ref{LYH} for the case that $M$ is compact. The case that $M$ is noncompact will be treated in Section 10.

\begin{lemma}\label{lmlyh}
 Let $S(\mathcal{X}, \mathcal{Y})$ be a Hermitian symmetric quadratic form defined by
\begin{equation*}
S(\mathcal{X}, \mathcal{Y})=A_{i\bar{j}}\mathcal{X}^i\ov{\mathcal{X}^j}+2Re(B_{ij}\mathcal{X}^i\mathcal{Y}^j
+D_{i\bar{j}}\mathcal{X}^i\ov{\mathcal{Y}^j})
+C_{i\bar{j}}\mathcal{Y}^i\ov{\mathcal{Y}^j}.
\end{equation*}
If $S$ is semi-positive definite, then
\begin{equation*}
 \sum_{i, j=1}^NA_{i\bar{j}}C_{j\bar{i}}\ge \max\{
\sum_{i, j=1}^N|B_{ij}|^2, \sum_{i, j=1}^N|D_{i\bar{j}}|^2\}.
\end{equation*}
\end{lemma}

If one prefers notations without indices the first three terms of (\ref{eq:77}) can be written as
$\langle \mathcal{Q}(\sum_{\mu=1}^p Y_\mu\wedge \ov{W}_\mu), \ov{\sum_{\nu=1}^p  Y_\nu\wedge \ov{W}_\nu}\rangle.
$ The last term can be written as $\langle \Rm( \sum_{\mu=1}^p X_\mu\wedge \ov{V}_\mu),  \ov{\sum_{\nu=1}^p  X_\nu\wedge \ov{V}_\nu}\rangle.$

\section{The proof of Theorem \ref{main-KRFlyh}}
Before we start, we remark that for the cases $p=1$ and $p=m$ the result has been previously proved in \cite{N-jams} and \cite{Ni-JDG07}. As before we deal with the compact case first. By an perturbation argument we also consider that $\phi$ is strictly positive. Then $\widetilde{Q}>0$ for small $t$. Assume that at some point $(x_0, t_0)$, $\widetilde{Q}=0$ for the first time for some linearly independent vectors $v_1, v_2, \cdots,v_{p-1}$. As in Section 5, let $v_i(z)$ and $V(z)$  be variational vectors such that they depend on $z$ holomorphically.
Now consider the following function in $z$,
\begin{eqnarray*}
\mathcal{I}(z)&\doteqdot &\frac{1}{2}\left[\operatorname{div}''
(\operatorname{div}'(\phi))+ \operatorname{div}'
(\operatorname{div}''(\phi))\right]_{v_1(z)\cdots v_{p-1}(z), \bar{v}_1(z)\cdots
\bar{v}_{p-1}(z)}\\
&\, &
+R_{j\bar{i}}\phi_{iv_1(z)\cdots v_{p-1}(z),
\bar{j}\bar{v}_1(z)\cdots\bar{v}_{p-1}(z)}\\
&\, &
+\left[\operatorname{div}'_{\ov{V}(z)}(\phi)+\operatorname{div}''_{V(z)}(\phi)
\right]_{v_1(z)\cdots v_{p-1}(z), \bar{v}_1(z)\cdots \bar{v}_{p-1}(z)}\\
&\,& +\phi_{V(z)v_1(z)\cdots v_{p-1}(z), \bar{V}(z)\bar{v}_1(z)\cdots
\bar{v}_{p-1}(z)}+\frac{\Lambda \phi_{v_1(z)\cdots v_{p-1}(z),
\bar{v}_1(z)\cdots \bar{v}_{p-1}(z)}}{t}
\end{eqnarray*}
which satisfies $\mathcal{I}(0)=0$ and $\mathcal{I}(z)\ge 0$ for any variational
vectors $v_{\mu}(z), V(z)$
with $v_{\mu}(0)=v_{\mu}$ and $V(0)=V$.
As before, without the loss of generality  we may assume that
$\{ v_1, \cdots, v_{p-1}\}=\{\frac{\partial}{\partial z^1},\cdots,
\frac{\partial}{\partial z^{p-1}}\}$. By the first and the second variation consideration as in Section 5,
we have that
\begin{equation}\label{eq:81}
\operatorname{div}'_{\bar{X}}(\phi)+\phi_{V, \bar{X}}=0=
\operatorname{div}''_{X}(\phi)+\phi_{X, \bar{V}},
\end{equation}
\begin{equation}\label{eq:82}
\widetilde{Q}_{v_1\cdots (X)_\mu\cdots v_{p-1}, \bar{v}_1\cdots
\bar{v}_{p-1}}=0=\widetilde{Q}_{v_1\cdots v_{p-1}, \bar{v}_1\cdots (\bar{X})_{\nu}\cdots
v_{p-1}},
\end{equation}
and for any $(1, 0)$-type vectors $X, X_i$,
\begin{eqnarray*}
&\, &\sum_{\mu, \nu=1}^{p-1}\widetilde{Q}_{v_1\cdots X_\mu \cdots v_{p-1}, \bar{v}_1\cdots \bar{X}_\nu \cdots
\bar{v}_{p-1}}+ \phi_{X v_1\cdots v_{p-1}, \bar{X} \bar{v}_1\cdots
\bar{v}_{p-1}}\\
\quad \quad \quad &+& \sum_{\mu=1}^{p-1}\operatorname{div}'_{\bar{X}}(\phi)_{v_1\cdots X_\mu\cdots
v_{p-1}, \bar{v}_1\cdots \bar{v}_{p-1}}+\phi_{V v_1\cdots X_\mu\cdots v_{p-1},
\bar{X} \bar{v}_1\cdots \bar{v}_{p-1}}\\
\quad \quad \quad &+& \sum_{\nu=1}^{p-1}\operatorname{div}''_{X}(\phi)_{v_1\cdots v_{p-1},
\bar{v}_1\cdots \bar{X}_\nu\cdots \bar{v}_{p-1}}+\phi_{X v_1\cdots v_{p-1},
\bar{V} \bar{v}_1\cdots \bar{X}_\nu \cdots \bar{v}_{p-1}}\\
\quad \quad \quad &\ge& 0.
\end{eqnarray*}
This amounts to that the block matrix (as defined in Section 5) $\mathcal{M}_1\ge 0$. However here $Q$ is replaced by $\widetilde{Q}$.

To check that $\heat \widetilde{Q}_{1\cdots p-1, \bar{1}\cdots \ov{p-1}}\ge 0$ we may
extend $V$ such that the following holds:
\begin{eqnarray*}
&\quad& (\frac{\p}{\p t}-\D)V^{i}=-\frac{1}{t}V^{i}
\\
&\quad& \nabla_{i}V=R^{j}_i\frac{\p}{\p z^j}+\frac{1}{t}\frac{\p}{\p z^i}, \quad \
\nabla_{\bar{i}}V^{j}=0.
\end{eqnarray*}
Using these set of equations, (\ref{eq:68}), (\ref{eq:69}) and
(\ref{eq:414}) can be simplified to
\begin{eqnarray}
 &\,&\heat (\operatorname{div}''_{V}(\phi))_{I_{p-1},
\ov{J}_{p-1}}=-\frac{1}{t}\operatorname{div}''_{ V}(\phi)
+\mathcal{KB}(\operatorname{div}''_{V}(\phi)) \label{eq:83}\\
\quad  \quad &+& \sum_{\nu =1 }^{p-1}R^{\,
k\bar{l}}_{V,\,\bar{j_\nu}}(\operatorname{div}''(\phi))_{ki_1\cdots
 i_{p-1},\bar{j_1}\cdots(\bar{l})_\nu\cdots\bar{j}_{p-1}}-\frac{1}{2}
\operatorname{div}''_{\Ric(V)}(\phi)
 -\frac{1}{t} \operatorname{div}'( \operatorname{div}''(\phi)) \nonumber \\
\quad\quad &+&R_{j\bar{k}}\nabla_k\phi_{V I_{p-1},
\bar{j}\ov{J}_{p-1}}+\nabla_VR_{j\bar{i}}\phi_{i I_{p-1},
\bar{j}\ov{J}_{p-1}}+\sum_{\mu =1}^{p-1}\nabla_{i_\mu}R_{j\bar{k}}
\phi_{V i_1\cdots(k)_\mu\cdots i_{p-1}, \bar{j}\ov{J}_{p-1}}\nonumber\\
\quad\quad&-&
R_{l\bar{k}}\nabla_{\bar{l}}(\operatorname{div}''(\phi))_{k I_{p-1},
\ov{J}_{p-1}};\nonumber\\
&\,&\heat (\operatorname{div}'_{\ov{V}}(\phi))_{I_{p-1}, \bar{J}_{p-1}}=
-\frac{1}{t}\operatorname{div}'_{\ov{V}}(\phi)+\mathcal{KB}(\operatorname{div}'_
{\ov{V}}(\phi)) \label{eq:84}\\
\quad \quad &
+& \sum_{\mu =1 }^{p-1}R^{\,
k\bar{l}}_{i_\mu,\,\bar{V}}(\operatorname{div}'(\phi))_{i_1\cdots(k)_\mu\cdots
 i_{p-1},
\bar{l}\bar{j_1}\cdots\bar{j}_{p-1}}-\frac{1}{2}\operatorname{div}'_{\ov{\Ric(V)
}}(\phi)
  -\frac{1}{t}\operatorname{div}''( \operatorname{div}'(\phi)) \nonumber\\
\quad\quad&+&
R_{k\bar{i}}\nabla_{\bar{k}}\phi_{i I_{p-1}, \ov{V}\ov{J}_{p-1}}
+\nabla_{\ov{V}}R_{j\bar{i}}\phi_{i I_{p-1}, \bar{j}\ov{J}_{p-1}}
+\sum_{\nu =1}^{p-1}\nabla_{\bar{j}_\nu}R_{l\bar{i}}
\phi_{i I_{p-1},
\ov{V}\bar{j}_1\cdots(\bar{l})_\nu\cdots\bar{j}_{p-1}}\nonumber\\
\quad\quad&-&
R_{l\bar{k}}\nabla_k(\operatorname{div}'(\phi))_{I_{p-1},
\bar{l}\ov{J}_{p-1}};\nonumber\\
&\,&\heat \phi_{V,\bar{V}}=\mathcal{KB}(\phi_{V,\bar{V}})+ R^{\,
k\bar{l}}_{V\,\bar{V}} \phi_{k, \bar{l}}+\sum_{\nu =1 }^{p-1}R^{\,
k\bar{l}}_{V\,\bar{j_\nu}} \phi_{k I_{p-1}, \bar{V} \bar{j}_1\cdots
(\bar{l})_\nu\cdots \bar{j}_{p-1}}\label{eq:85}\\
\quad \quad &+&\sum_{\mu =1 }^{p-1}R^{\,
k\bar{l}}_{i_\mu\,\bar{V}} \phi_{V i_1\cdots (k)_\mu\cdots i_{p-1},
\bar{l}\bar{J}_{p-1}}-\frac{1}{2}\left(\phi_{V, \ov{\Ric(V)}}+\phi_{\Ric(V),
\bar{V}}\right) \nonumber\\
\quad \quad &-& \frac{2}{t}\phi_{ V, \bar{V}}-\frac{\Lambda
\phi}{t^2}-\frac{2}{t}R_{j\bar{i}}\phi_{i I_{p-1}, \bar{j}\ov{J}_{p-1}}-
\frac{1}{t}\operatorname{div}'_{\ov{V}}(\phi)
-\frac {1}{t}\operatorname{div}''_{V}(\phi)\nonumber \\
\quad\quad&-&
\left(R_{j\bar{k}}\nabla_k\phi_{V I_{p-1}, \bar{j}\ov{J}_{p-1}}
+R_{k\bar{i}}\nabla_{\bar{k}}\phi_{i I_{p-1}, \bar{V}\ov{J}_{p-1}}
\right)-R_{j\bar{k}}R_{k\bar{i}}\phi_{iI_{p-1}, \bar{j}\ov{J}_{p-1}}.\nonumber
\end{eqnarray}
Adding them up with the two evolution equations in Lemma \ref{helpKRF62} and
(\ref{eq:lem65}), (\ref{eq:610}), using (\ref{eq:81})  we have that
\begin{eqnarray*}
&\,&\heat \widetilde{Q}_{I_{p-1}, \ov{J}_{p-1}}
=\mathcal{KB}(\widetilde{Q})_{I_{p-1}, \ov{J}_{p-1}}
+\sum_{i\, j=1}^m Z_{j\bar{i}}\phi_{iI_{p-1}, \bar{j}\ov{J}_{p-1}}
\\
&\quad&+
\sum_{\mu=1}^{p-1}(R_{i_\mu\bar{V}l\bar{k}}
+P_{l\bar{k}i_\mu})(\operatorname{div}'(\phi)_{i_1\cdots(k)_{\mu}\cdots i_{p-1},
\bar{l}\ov{J}_{p-1}}+\phi_{Vi_1\cdots(k)_{\mu}\cdots i_{p-1}, \bar{l}\ov{J}_{p-1}})\\
&\quad&+\sum_{\nu=1}^{p-1}(
R_{V\bar{j}_{\nu}l\bar{k}}+P_{l\bar{k}\bar{j}_\nu})
(\operatorname{div}''(\phi)_{kI_{p-1}, \bar{j}_1\cdots(\bar{l})_{\nu}\cdots\bar{j}_{p-1}}
+\phi_{kI_{p-1}, \bar{V}\bar{j}_1\cdots(\bar{l})_{\nu}\cdots\bar{j}_{p-1}})\\
&\quad&-\frac{2}{t}\widetilde{Q}_{I_{p-1}\ov{J}_{p-1}}.
\end{eqnarray*}
Now the nonnegativity of $\heat \widetilde{Q}_{1\cdots(p-1), \bar{1}\cdots (\ov{p-1})}$ at
$(x_0, t_0)$  can be proved in a similar way as the argument in Section 2. First
observe that the part of $\mathcal{KB}(\widetilde{Q})_{1\cdots(p-1), \bar{1}\cdots
(\ov{p-1})}$ involving only $\Ric$ is
$$
-\frac{1}{2}\sum_{i=1}^{p-1}\left(\widetilde{Q}_{1\cdots \Ric(i)\cdots (p-1), \bar{1}\cdots
(\ov{p-1})}+\widetilde{Q}_{1\cdots (p-1), \bar{1}\cdots\ov{\Ric(i)}\cdots (\ov{p-1})}\right)
$$
which vanished due to (\ref{eq:82}).
Hence we only need to establish the nonnegativity of
\begin{eqnarray*}\tilde{J}&\doteqdot &
\sum_{\mu = 1}^{p-1}\sum_{\nu =1
}^{p-1}R_{\mu\bar{\nu}l\bar{k}}\widetilde{Q}_{1\cdots(k)_\mu\cdots
 (p-1), \bar{1}\cdots(\bar{l})_\nu\cdots(\ov{p-1})}\\
\quad \quad &+&
\sum_{\mu=1}^{p-1}(R_{\mu\bar{V}l\bar{k}}
+P_{l\bar{k}\mu})(\operatorname{div}'(\phi)_{1\cdots(k)_{\mu}\cdots (p-1),
\bar{l}\ov{J}_{p-1}}+\phi_{V1\cdots(k)_{\mu}\cdots (p-1), \bar{l}\ov{J}_{p-1}})\\
\quad\quad
&+&\sum_{\nu=1}^{p-1}(
R_{V\bar{\nu}l\bar{k}}+P_{l\bar{k}\bar{\nu}})
(\operatorname{div}''(\phi)_{kI_{p-1}, \bar{1}\cdots(\bar{l})_{\nu}\cdots\ov{p-1}}
+\phi_{kI_{p-1}, \bar{V}\bar{1}\cdots(\bar{l})_{\nu}\cdots\ov{p-1}})\\
\quad\quad &+& \sum_{i, j=1}^m Z_{j\bar{i}}\phi_{i1\cdots (p-1), \bar{j}\bar{1}\cdots
(\ov{p-1})}.
\end{eqnarray*}
By Theorem \ref{LYH}, the assumption that the curvature operator $\Rm$ is in $\mathcal{C}_p$ implies that
the matrix
$$
\mathcal{M}_3=\left(\begin{array}{l}
R_{1\bar{1}(\cdot)\bar{(\cdot)}}\quad\quad
R_{1\bar{2}(\cdot)\bar{(\cdot)}}\quad\quad\quad
\cdots\quad\quad
R_{1\ov{p-1}(\cdot)\bar{(\cdot)}}  \quad \quad \quad
D^1_{(\cdot)\bar{(\cdot)}}\\
R_{2\bar{1}(\cdot)\bar{(\cdot)}}\quad \quad
R_{2\bar{2}(\cdot)\bar{(\cdot)}}\quad\quad\quad
\cdots\quad\quad
R_{2\ov{p-1}(\cdot)\bar{(\cdot)}}\quad \quad \quad
D^2_{(\cdot)\bar{(\cdot)}}\\
\quad\cdots\quad\quad\quad\quad \cdots\quad\quad\quad\quad
\cdots\quad\quad\quad\cdots\quad \quad \quad \quad \quad \quad \cdots \\
R_{p-1\bar{1}(\cdot)\bar{(\cdot)}}\quad
R_{p-1\bar{2}(\cdot)\bar{(\cdot)}}\quad\quad
\cdots\quad\quad
R_{p-1\ov{p-1}(\cdot)\bar{(\cdot)}}\quad \quad
D^{p-1}_{(\cdot)\bar{(\cdot)}}\\
\ov{D^1}^{tr}_{(\cdot)\bar{(\cdot)}}
\quad\qquad
\ov{D^2}^{tr}_{(\cdot)\bar{(\cdot)}} \quad\quad\,\,
\cdots\quad\qquad
\ov{D^{p-1}}^{tr}_{(\cdot)\bar{(\cdot)}}
\quad \quad \quad Z_{(\cdot)\bar{(\cdot)}}
\end{array}\right)\ge 0,
$$
where $D^i_{\mu\bar{\nu}}=R_{i\bar{V}\mu\bar{\nu}}
+P_{\mu\bar{\nu}i}$.
The nonnegativity of $\tilde{J}$ follows from
$\operatorname{trace}(\mathcal{M}_1 \cdot
\mathcal{M}_3)\ge 0$. Here $\mathcal{M}_1$ is the block matrix in Section 5. This proves Theorem \ref{main-KRFlyh} for the case that $M$ is compact. We postpone the proof of the noncompact case to a later section.

\section{LYH type estimates for the Ricci Flow under the condition $\widetilde{\mathcal{C}}_p$}

In this section we prove another set of LYH type estimates for the Ricci flow. Let $(M, g(t))$ be a complete solution to
\begin{equation}\label{rfeq}
\frac{\partial}{\partial t} g_{ij}=-2R_{ij}.
\end{equation}
Recall that Hamilton proved that if $\Rm\ge 0$ and bounded then the quadratic form
$$
\widetilde{\mathcal{Q}}(W\oplus U)\doteqdot \langle \mathcal{M}(W), W\rangle +2\langle P(W), U\rangle +\langle \Rm(U), U\rangle\ge 0
$$
where the $\mathcal{M}$ and $P$ are  defined in a normal frame by
\begin{eqnarray*}
\mathcal{M}_{ij}&\doteqdot&\Delta R_{ij}-\frac{1}{2}\nabla_i\nabla_j R +2R_{ikjl}R_{kl} -R_{ik}R_{jk}+\frac{1}{2t}R_{ij},\\
P_{ijk}&\doteqdot& \nabla_i R_{jk}-\nabla_j R_{ik}
\end{eqnarray*}
with $\langle P(W), U\rangle =P_{ijk} W^k U^{ij}$.
One can view $\HQ$ as the restriction of a Hermitian quadratic form
$$
\langle \widetilde{\mathcal{Q}}(W\oplus U), \ov{W\oplus U}\rangle \doteqdot \langle \mathcal{M}(W), \ov{W}\rangle +2Re\langle P(W), \ov{U}\rangle +\langle \Rm(U), \ov{U}\rangle
$$
which is defined on $\wedge^2(\C^n)\oplus \C^n$. We also denote by
$$
\langle \mathcal{Z}(W\wedge Z), \ov{W\wedge Z}\rangle\doteqdot \HQ (W\oplus (W\wedge Z)).
$$
Fixing a $Z$, $\mathcal{Z}$ can be viewed as a Hermitian bilinear form of $W$, which  we  denote by $\mathcal{Z}_Z$, or still by $\mathcal{Z}$ when the meaning is clear.
In terms of local frame, it can be written as
$$
\left(\mathcal{Z}_Z\right)_{cd}=\mathcal{M}_{cd}+P_{dac}\ov{Z}^a +P_{cad}Z^a +R_{Zc \ov{Z} d}.
$$

\begin{theorem}\label{rf-lyh} Assume that $(M, g(t))$ on $M \times [0, T]$ satisfies $\widetilde{\mathcal{C}}_p$. When $M$ is noncompact we also assume that the curvature of $(M, g(t))$ is uniformly bounded on $M \times [0, T]$. Then for any $t>0$,
$\HQ \ge 0$ for any $(x, t)\in M\times [0,  T]$, $W\in T_x M \otimes \C$ and $U\in \wedge^2(T_xM\otimes \C) $ such that $U=\sum_{\mu=1}^p W_\mu \wedge Z_\mu$ with  $W_p=W$. Furthermore, the equality holds for some $t>0$ implies that the universal cover of $(M, g(t))$ is a gradient expanding Ricci soliton.
\end{theorem}

\begin{remark}
In \cite{brendle}, a slightly weaker result was proved for the $p=1$ case. As before, for $p$ large enough the condition  $\widetilde{\mathcal{C}}_p$ is equivalent to that $\Rm\ge 0$ and  the above result is equivalent to Hamilton's theorem. Hence our result gives a family of estimates interpolating between those of \cite{brendle} and \cite{richard-harnack}.
\end{remark}

In Theorem 4.1 of \cite{richard-harnack},  the following result was proved by brutal force computations.
\begin{lemma}\label{ham-com} At $(x_0, t_0)$, if $W$ and $U$ are extended by the equations:
\begin{eqnarray*}
\heat W&=&\frac{W}{t}+\Ric(W), \quad \quad \nabla W=0;\\
\heat U^{ab}&=&R^a_c U^{cb}+R^{b}_cU^{ac},\\
\nabla_a U^{bc}&=&\frac{1}{2}(R_{a}^{b} W^c -R_{a}^{c}W^b)+\frac{1}{4t}(g_{a}^bW^c-g_{a}^cW^b),
\end{eqnarray*}
then under an orthonormal frame
\begin{eqnarray}\label{hamilton-help1}
\heat \HQ &=& 2R_{acbd}\mathcal{M}_{cd}W^a \ov{W}^b -2P_{acd}P_{bdc}W^a\ov{W}^b
\\&\quad&+ 8Re\left( R_{adce} P_{dbe} W^c \ov{U}^{ab}\right)+4R_{aecf}R_{bedf}U^{ab}\ov{U}^{cd} \nonumber\\
&\quad& +|P(W)+\Rm(U)|^2.\nonumber
\end{eqnarray}
Here $R^a_b$ denotes the $\Ric$ transformation in terms of the local frame.
\end{lemma}

Using the notation of \cite{H86}, the term $4R_{aecf}R_{bedf}U^{ab}\ov{U}^{cd}$ can be expressed as $8\langle \Rm^{\#} (U), \ov{U}\rangle$.

Assume that $\HQ \ge 0$ for $M\times[0, t_0]$ and at $(x_0, t_0)$ it vanished for $W\oplus U$, with $U=\sum_{\mu=1}^p W_\mu \wedge Z_{\mu}$, and $W_p=W$. Now let $W_\mu(z)$ and $Z_\mu(z)$ be a variation of $W_\mu$ and $Z_\mu$ with $W_\mu(z)=W_\mu+z X_\mu$ and $Z_\mu(z)=Z_\mu+zY_\mu$. Let $\widetilde{\mathcal{I}}(z)\doteqdot \HQ (W(z)\oplus U(z))$ with $U(z)=\sum_{\mu=1}^p W_\mu(z)\wedge Z_\mu(z)$. Using $\Delta \widetilde{\mathcal{I}}(0)\ge 0$, we deduce the following estimate:
\begin{eqnarray}
&\, &\sum_{\mu, \nu=1}^p
R_{X_\mu Z_\mu \ov{X}_\nu \ov{Z}_\nu}
+2Re \left(\langle P(X_p), \ov{\sum_{\mu=1}^p X_\mu \wedge Z_\mu}\rangle\right) +\langle \mathcal{M}(X_p), \ov{X}_{p}\rangle
\label{eq:91}\\
\quad
&+&
2Re \left(\sum_{\mu, \nu=1}^p R_{X_\mu Z_\mu \ov{W}_\nu \ov{Y}_\nu}\right)+
2Re \left(\langle P(X_p), \ov{\sum_{\mu=1}^p W_\mu \wedge Y_\mu}\rangle \right)\nonumber\\
\quad
&+&\sum_{\mu, \nu=1}^p
R_{W_{\mu}Y_\mu \ov{W}_\nu \ov{Y}_\nu}\ge 0.\nonumber
\end{eqnarray}

To prove Theorem \ref{rf-lyh}  for the compact case, it suffices to show that the right hand side of (\ref{hamilton-help1}) is nonnegative for a null vector $W\oplus U$ with $U=\sum_{\mu=1}^p W_\mu \wedge Z_\mu$ and $W_p=W$. Denote the first four terms in the right hand side of (\ref{hamilton-help1}) by $\HJ$.  Expand it and let  $\hat{P}_{dc}(Z_p)=P_{dac}Z^a_p$. We then obtain that
\begin{eqnarray*}
\HJ &=& 2 R_{W_p c \ov{W}_p d} \mathcal{Z}_{cd}
-2R_{W_p c\ov{W}_pd}\left(\hat{P}_{dc}(\ov{Z}_p)+\hat{P}_{cd}(Z_p)\right)-2R_{W_p c\ov{W}_p d}
R_{Z_pc\ov{Z}_p d}\\
&\quad& + 2\sum_{\mu, \nu=1}^p\left(R_{W_\mu e \ov{W}_\nu f}R_{Z_\mu e \ov{Z}_\nu f}-R_{W_\mu e \ov{Z}_\nu f}R_{Z_\mu e\ov{W}_\nu f}\right)\\
&\quad& +4Re \left(\sum_{\mu=1}^{p-1}R_{\ov{W}_\mu d W_p e}\hat{P}_{de}(\ov{Z}_\mu)\right)+4Re\left(R_{\ov{W}_p d W_pe}\hat{P}_{de}(\ov{Z}_p)\right)\\
&\quad& -4Re\left( \sum_{\mu=1}^{p-1}R_{\ov{Z}_\mu d W_p e}\hat{P}_{de}(\ov{W}_\mu)\right)-4Re\left(R_{\ov{Z}_p d W_pe}\hat{P}_{de}(\ov{W}_p)\right)\\
&\quad& -2\hat{P}_{de}(W_p)\hat{P}_{ed}(\ov{W}_p).
\end{eqnarray*}
After some cancelations (the 2nd term and the 7th term on the right hand side above cancel each other) the nonnegativity of $\HJ=2\left(\operatorname{trace}(B_1 B_2)-\operatorname{trace}(B_3\cdot \ov{B}_3)\right)$
where
\begin{eqnarray*}
B_1=\left(\begin{array}{l}
R_{W_1(\cdot)\ov{W}_1(\cdot)}\qquad
R_{W_1(\cdot)\ov{W}_2(\cdot)}\quad\quad
\cdots\quad\quad
R_{W_1(\cdot)\ov{W}_{p}(\cdot)}
\\
R_{W_2(\cdot)\ov{W}_1(\cdot)}\qquad
R_{W_2(\cdot)\ov{W}_2(\cdot)}\quad\quad
\cdots\quad\quad
R_{W_2(\cdot)\ov{W}_{p}(\cdot)}\\
\quad\quad\cdots\quad\quad\quad\quad\cdots\qquad\qquad\quad\cdots\quad\quad
\quad\cdots\\
R_{W_p(\cdot)\ov{W}_1(\cdot)}\qquad
R_{W_p(\cdot)\ov{W}_2(\cdot)}\quad\quad
\cdots\quad\quad
R_{W_p(\cdot)\ov{W}_p(\cdot)}
\end{array}\right),
\end{eqnarray*}
\begin{eqnarray*}
B_2=\left(\begin{array}{l}
R_{\ov{Z}_1 (\cdot) Z_1(\cdot)}\quad\quad\quad
R_{\ov{Z}_1(\cdot) Z_2(\cdot)}\quad\quad\quad
\cdots\quad
R_{\ov{Z}_1(\cdot)Z_{p-1}(\cdot)}\quad \quad
 F^{1p}\\
R_{\ov{Z}_2 (\cdot) Z_1(\cdot)}\quad\quad\quad
R_{\ov{Z}_2(\cdot) Z_2(\cdot)}\quad\quad\quad
\cdots\quad
R_{\ov{Z}_2(\cdot)Z_{p-1}(\cdot)}\quad\quad
F^{2p}\\
\quad\quad\cdots\qquad\qquad\quad \cdots\qquad\qquad\quad \cdots\quad\quad\quad
\quad\cdots\quad\quad\quad\cdots\\
R_{\ov{Z}_{p-1} (\cdot) Z_1(\cdot)}\quad\quad
R_{\ov{Z}_{p-1}(\cdot) Z_2(\cdot)}\quad\quad
\cdots\quad
R_{\ov{Z}_{p-1}(\cdot)Z_{p-1}(\cdot)}\quad F^{p-1 p}
\\
\quad \quad \quad  \ov{F^{1p}}^{tr}\qquad\quad\quad
\ov{F^{2p}}^{tr}\quad\quad\quad \cdots\quad\quad\quad
\ov{F^{(p-1)p}}^{tr}\quad \quad
\mathcal{Z}_{(\cdot)(\cdot)}
\end{array}\right),
\end{eqnarray*}
and
\begin{eqnarray*}
B_3=\left(\begin{array}{l}
R_{W_1(\cdot)\ov{Z}_1(\cdot)}\quad
R_{W_1(\cdot)\ov{Z}_2(\cdot)}\quad\quad\cdots\quad
R_{W_1(\cdot)\ov{Z}_{p-1}(\cdot)}\quad
R_{W_1(\cdot)\ov{Z}_p(\cdot)}+\hat{P}(W_1)\\
R_{W_2(\cdot)\ov{Z}_1(\cdot)}\quad
R_{W_2(\cdot)\ov{Z}_2(\cdot)}\quad\quad\cdots\quad
R_{W_2(\cdot)\ov{Z}_{p-1}(\cdot)}\quad
R_{W_2(\cdot)\ov{Z}_p(\cdot)}+\hat{P}(W_2)\\
\quad\quad\cdots\quad\quad\qquad\cdots\qquad\quad\ \cdots\quad
\quad\quad\cdots\qquad\qquad\qquad\cdots\\
R_{W_p(\cdot)\ov{Z}_1(\cdot)}\quad
R_{W_p(\cdot)\ov{Z}_2(\cdot)}\quad\quad\cdots\quad
R_{W_p(\cdot)\ov{Z}_{p-1}(\cdot)}\quad
R_{W_p(\cdot)\ov{Z}_p(\cdot)}+\hat{P}(W_p)
\end{array}\right),
\end{eqnarray*}
with $F^{\mu p}=R_{\ov{Z}_{\mu}(\cdot)Z_{p}(\cdot)}+\hat{P}_{(\cdot) (\cdot)}(\ov{Z}_\mu)$.

On the other hand using the above notation (\ref{eq:91}) can be re-written as
\begin{eqnarray*}
&\,&\mathcal{Z}_{X_p \ov{X}_p}+2\sum_{\mu=1}^{p-1}Re \left( R_{\ov{Z}_\mu \ov{X}_\mu Z_p X_p}\right)+2\sum_{\mu=1}^{p-1}Re \left( \hat{P}_{\ov{X}_\mu X_p}(\ov{Z}_\mu)\right)+\sum_{\mu, \nu =1}^{p-1}R_{\ov{Z}_\mu \ov{X}_\mu Z_\nu X_\nu}\\
&\, & -2\sum_{\mu=1}^{p}Re\left( \hat{P}_{\ov{Y}_\mu X_p}(\ov{W}_\mu)+\sum_{\nu=1}^p R_{\ov{W}_\mu \ov{Y}_\mu Z_{\nu} X_\nu}\right)+\sum_{\mu, \nu=1}^p R_{W_\mu Y_\mu \ov{W}_\nu \ov{Y}_\nu}\\
&\,& \ge 0
\end{eqnarray*}
which  amounts to $S(\mathcal{X}, \mathcal{Y})$
being nonnegative, where
$$
S(\mathcal{X}, \mathcal{Y})=(B_2)_{ij}\mathcal{X}^i\ov{\mathcal{X}^j}-2 Re \left( (B_3)_{ij}\mathcal{Y}^i \ov{\mathcal{X}^j}\right)+(B_1)_{ij}\mathcal{Y}^i\ov{\mathcal{Y}^j}
$$
with $\mathcal{X}=\left(\begin{array}{l} X_1\\ \vdots\\
 X_p\end{array}\right),\,   \mathcal{Y}=\left(\begin{array}{l} Y_1\\ \vdots\\
 Y_p\end{array}\right)$. Hence by Lemma \ref{lmlyh} we can conclude that $\HJ\ge 0$ and complete the proof of Theorem \ref{rf-lyh} for the compact case. The case that $M$ is noncompact will be proved in the next section together with Theorem \ref{LYH}.

\section{Complete noncompact manifolds with bounded curvature}

In this section we first show that under the condition that the curvature tensor of $(M, g(t))$, a solution to the Ricci flow or K\"ahler-Ricci flow, is  uniformly bounded on $M \times [0, T]$, the maximum principle can still apply and conclude the invariance of the cone $\mathcal{C}_p$, $\widetilde{\mathcal{C}}_p$ from Section 3. Moreover, the LYH type estimates for the K\"ahler-Ricci flow and the Ricci flow in Section 7 and 9 remain valid.

First we show the invariance of the $\mathcal{C}_p$ and $\widetilde{\mathcal{C}}_p$. In fact the following maximum principle holds on noncompact manifolds.  Consider $V$, a vector bundle over $M$, with a fixed metric $\widetilde{h}$, a time-dependent metric connection
$D^{(t)}$. On $M$ there are time-dependent metrics $g(t)$ and $\nabla^{(t)}$, the Levi-Civita connection of $g(t)$. When the meaning is clear we often omit the sup-script $^{(t)}$.
The main concern of this subsection is the diffusion-reaction equation:
\begin{eqnarray}\label{pde-ode}
\left\{\begin{array}{ll}
\quad &\frac{\partial  }{\partial t} f(x, t)-\Delta f(x, t) =\Phi(f)(x, t),\\
\quad& f(x,0 )=f_0(x).\end{array}
\right.
\end{eqnarray}
Here $\Delta =g^{ij}(x, t)D_iD_j$. We know that after applying the Ulenbeck's trick \cite{H86} the study of the curvature operator under the Ricci flow equation is a subcase of  this general formulation. One can modify the proof of Theorem 1.1 in \cite{B-W} to obtain the following result.

\begin{theorem} Assume that $M$ is a complete noncompact manifold and  $\Phi$ is locally Lipschitz. Let $(M, g(t))$ be a solution to Ricci flow such that $|\Rm|(x, t)\le A$ for some $A>0$ for any $(x, t)\in M\times[0, T]$. Let $C(t)\subset V$,   $t\in [0, T]$,  be a family of closed full dimensional cones, depending continuously on $t$. Suppose that each of the cones $C(t)$ is invariant under parallel transport, fiberwise convex and that the family $\{C(t)\}$ is preserved by the ODE $\frac{d}{dt} f(t)=\Phi(f)$. Moreover assume that there exists a smooth section $I$ which is invariant under the parallel transport and $I\in C(t)$ for all $t$. If $f(x, t)$ satisfies (\ref{pde-ode}) with $f(x, 0)\in C(0)$, $|f|(x, t)\le B$ on $M\times[0, T]$ for some $B>0$, then  $f(x, t)\in C(t)$ for $(x, t)\in M \times [0, T]$ .
\end{theorem}

\begin{proof} The key is Lemma \ref{lemma101} below, which ensures the existence of a smooth function $\varphi$ such that $\varphi(x, t)\to +\infty$ uniformly on $[0, \eta]$ for some $\eta>0$ and $\heat \varphi \ge C\varphi.$ Clearly once we can prove the result for $[0, \eta]$ we can iterate the procedure and get the result on $[0, T]$.

For any $\epsilon>0$, we can fix a compact region $K$ such that $\tilde{f}(x, t)\doteqdot f(x, t)+\epsilon\varphi \I\in C(t)$ for all $(x, t)$ with $x\in M\setminus K$. In fact one can choose $K=\overline{B^0(p, R_0)}$,  a closed ball of a certain radius $R_0$ with respect to the initial metric. Now for every $t$, $\rho(x, t)=\operatorname{dist}^2(\tilde f(x, t), C_x(t))$ with $C_x(t)=C(t)\cap V_x$ achieves a maximum somewhere. The argument of \cite{B-W} can be applied and we only need to restrict ourselves over $K\times [0, \eta]$. In particular we let $\rho(t)=\rho(x_0, t)=\max \rho(\cdot, t)$.  Since $\Phi$ is locally Lipschitz it is easy to infer that there exists $A'$ such that $|\tilde{f}|+|\Phi(\tilde{f})|\le A'$ for some constant $A'$, on $K\times[0, T]$. Since $\varphi>0$, we can choose $C$ large enough such that $\epsilon C \varphi \I +\Phi(f)-\Phi(\tilde{f}) \in C(t)$ for all $(x, t)\in K\times [0, \eta]$. Now the rest of the argument in \cite{B-W} can be evoked to conclude that $D_{-} \rho(t)\le L\rho(t)$ with $L$ depending on the local Lipschitz constant of $\Phi$. Here $D_{-}$ is the lower Dini's derivative from the left. Precisely we have
\begin{eqnarray*}
D_{-}\rho(t)&\le & \langle \frac{\partial }{\partial t}\tilde{f}, \tilde{f}-v_\infty\rangle|_{(x_0, t)}-2\langle \Phi(v_\infty), \tilde{f}(x_0, t)-v_\infty\rangle\\
&=& 2\langle (\Delta \tilde{f})(x_0, t), \tilde{f}(x_0, t)-v_\infty\rangle\\
  &\quad& +2\langle \Phi(f)+\epsilon \heat \varphi \I|_{(x_0, t)} -\Phi(v_\infty), \tilde{f}(x_0, t)-v_\infty\rangle.
\end{eqnarray*}
Here $v_\infty$ is a vector in $V_{x_0}$ such that $\operatorname{dist}(\tilde{f}(x_0, t), v_\infty)=\operatorname{dist}(\tilde{f}(x_0, t), C_{x_0}(t)).$ By Lemma 1.2 of \cite{B-W}
 $$
 \langle (\Delta \tilde{f})(x_0, t), \tilde{f}(x_0, t)-v_\infty\rangle\le 0.
 $$
 For sufficient large $C$,
$\epsilon C \varphi \I +\Phi(f)-\Phi(\tilde{f}) \in C(t)$, which implies that
$\langle \epsilon C \varphi \I +\Phi(f)-\Phi(\tilde{f}),  \tilde{f}(x_0, t)-v_\infty\rangle\le 0$. Hence by the convexity of $C(t)$,
 $$
 \langle \Phi(f)+\epsilon \heat \varphi \I|_{(x_0, t)} +\Phi(\tilde{f}), \tilde{f}(x_0, t)-v_\infty\rangle\le 0.
 $$
Combining the above we conclude that
$$
D_{-}\rho(t)\le 2\langle \Phi(\tilde{f}(x_0, t))-\Phi(v_\infty), \tilde{f}(x_0, t)-v_\infty\rangle \le L\rho(t).
$$
The rest of the proof follows from \cite{B-W} verbatim.
\end{proof}

\begin{lemma}\label{lemma101} Assume that $M$ is a complete noncompact manifold. Let $(M, g(t))$ be a solution to Ricci flow such that $|\Rm|(x, t)\le A$ for some $A>0$ for any $(x, t)\in M\times[0, T]$.  Then there exist $C_1>0$ and a positive function $\varphi(x, t)$ such that for any given $C>0$, there exists $\eta>0$ such that on $M\times[0, \eta]$
\begin{eqnarray*}
\exp(C_1^{-1}(r_0(x)+1))&\le& \varphi(x, t)\le \exp(C_1(r_0(x)+1)+1), \\
\heat \varphi &\ge & C\varphi.
\end{eqnarray*}
Here $r_0(x)$ is the distance to a fixed point with respect to the initial metric.
\end{lemma}
\begin{proof} First by Lemma 5.1 of \cite{richard-harnack}, there exist $f(x)$ such that
\begin{eqnarray*}
C_1^{-1} (1+r_0(x)) &\le& f(x)\le C_1(1+r_0(x)),\\
|\nabla f|^2+|\nabla^2 f|&\le& C_1.
\end{eqnarray*}
Now we let $\varphi=\exp(\alpha t+f(x))$.  The claimed result follows easily.
\end{proof}

\begin{corollary} Let $(M, g(t))$ be a solution to Ricci flow (or K\"ahler-Ricci flow) such that $|\Rm|(x, t)\le A$ for some $A>0$ for any $(x, t)\in M\times[0, T]$. Then $\widetilde{\mathcal{C}}_p$ is invariant under the Ricci flow. (Respectively, $\mathcal{C}_p$ is invariant under the K\"ahler-Ricci flow.)
\end{corollary}

Concerning the LYH type estimates for the Ricci flow and K\"ahler-Ricci flow, we can evoke the perturbation argument of Hamilton \cite{richard-harnack}. Note that by passing to $[\epsilon, T-\epsilon]$, the curvature
bound, due to Shi's derivative estimates,  implies that all the derivatives of the curvature are uniformly bounded. Now consider the perturbed quantity
$$
\HQ'(W\oplus U)=\langle \mathcal{M}(W), \ov{W}\rangle+\frac{\varphi}{t} \langle W, \ov{W}\rangle +2Re \left( P(W), \ov{U}\rangle \right)+\langle \Rm(U),\ov{U}\rangle +\psi |U|^2
$$
where $\varphi$ and $\psi$ are the functions from Lemma 5.2 of \cite{richard-harnack}. Following the argument of Section 5 in \cite{richard-harnack} verbatim we can show the following result.

\begin{corollary} Assume that $(M, g(t))$ a solution to the Ricci flow on $M \times [0, T]$ such that $|\Rm|(x, t)\le A$. Assume that  $ \Rm(g(x, 0))\in \widetilde{\mathcal{C}}_p$. Then for any $t>0$,
$\HQ \ge 0$ for any $(x, t)\in M\times [0,  T]$, $W\in T_x M \otimes \C$ and $U\in \wedge^2(T_xM\otimes \C) $ such that $U=\sum_{\mu=1}^p W_\mu \wedge Z_\mu$ with  $W_p=W$. (Respectively, if $(M, g(t))$ is a solution to the K\"ahler-Ricci flow with $\Rm(g(x, 0))\in \mathcal{C}_p$, then $\mathcal{Q}\ge 0$).
\end{corollary}

\section{Complete noncompact manifolds without curvature bound}
We first discuss the existence of the Cauchy problem for  (\ref{eq:11}). First we observe that the maximum principle of Section 2 holds for the Dirichlet boundary problem by a perturbation argument adding $\epsilon \omega^p$. Precisely we have the following proposition.

\begin{proposition}
Let $(M, g)$ be a K\"ahler manifold whose curvature operator $\Rm \in \mathcal{C}_p$. Let $\Omega$ be a bounded domain in $M$. Assume that $\phi(x, t)$ satisfies that
\begin{eqnarray}\label{boundary1}
\left\{\begin{array}{ll}
\quad &\frac{\partial  }{\partial t} \phi(x, t)+\Delta_{\bar{\partial}} \phi(x, t) =0,\\
\quad &\phi(x, t)|_{\partial \Omega}\ge 0, \\
\quad& \phi(x,0 )=\phi_0(x)\ge 0.\end{array}
\right.
\end{eqnarray}
Then $\phi(x, t)\ge 0$ for $t>0$.
\end{proposition}

Note that the boundary condition $\phi(x, t)|_{\partial \Omega}\ge 0$ means that at any $x\in \partial \Omega$ and for any $v_1, \cdots, v_p$, $\phi(v_1, \cdots v_p; \bar{v}_1, \cdots, \bar{v}_p)\ge 0$ at $x$. Namely $\phi(x, t)|_{\partial \Omega}$ does not mean the restriction of the $(p, p)$-form to the boundary.

This will help us to obtain needed estimate to obtain a global solution in the case that $M$ is a noncompact complete manifold via some a priori estimates. A basic assumption is needed to ensure even the short time  existence of the Cauchy problem on  open manifolds. Here we assume that there exists a positive constant $a$ such that
\begin{equation}\label{ass-1}
\mathcal{B}\doteqdot\int_M |\phi_0(y)|\exp(-ar^2(y))\, d\mu(y)< \infty,
\end{equation}
 where $r(x)$ is the distance function to some fixed point $o\in M$. The pointwise norm $|\cdot|$ for $\phi$ is defined as $$|\phi|^2=\frac{1}{(p!)^2}\sum \phi_{I_p, \ov{J}_p}\ov{\phi_{K_p, \ov{L}_p}} g^{i_1 \bar{k}_1}\cdots g^{i_p \bar{k}_p} g^{l_1 \bar{j}_1}\cdots g^{l_p\bar{j}_p}.$$

By basic linear algebra, for example, Lemma 2.4 of \cite{Siu-74}, it is easy to see that for
positive $(p, p)$-forms, there exists $C_{p, m}$ such that
\begin{equation}\label{basic-1}
|\phi|(x)\le C_{p, m} |\Lambda \phi|(x).
\end{equation}
Now the existence of the solution to the Cauchy problem can be proved for any continuous  positive $(p, p)$-form $\phi_0(x)$ satisfying (\ref{ass-1}).

\begin{proposition}\label{short-ext}
Let $(M, g)$ be a K\"ahler manifold whose curvature operator $\Rm \in \mathcal{C}_p$. Assume that $\phi_0(x)$ satisfies (\ref{ass-1}), then there exists $T_0$ such that the Cauchy problem
\begin{eqnarray}\label{cauchy1}
\left\{\begin{array}{ll}
\quad &\frac{\partial  }{\partial t} \phi(x, t)+\Delta_{\bar{\partial}} \phi(x, t) =0,\\
\quad& \phi(x,0 )=\phi_0(x)\ge 0\end{array}
\right.
\end{eqnarray}
has a solution $\phi(x, t)$ on $M\times [0, T_0]$. Moreover, $\phi(x, t)\ge 0$ on $M \times [0, T_0]$ and satisfies the estimate
\begin{equation}\label{con-sol}
|\phi|(x, t)\le \mathcal{B}\cdot \frac{C(m, p)}{V_x(\sqrt{t})}\exp\left(2a\, r^2(x)\right).
\end{equation}
Here $V_x(r)$ is the volume of ball $B(x, r)$.
\end{proposition}
\begin{proof} Let $\Omega_\mu$ be a sequence of exhaustion bounded domains. By the standard theory on the linear parabolic system \cite{Friedman, Lady1}, there exist solutions $\phi_\mu(x, t)$ on $\Omega_\mu \times [0, \infty)$ such that $\phi_\mu(x, t)=0$ on $\partial \Omega \times [0, \infty)$. Note that in terms of the language of \cite{Morrey-har}, $\phi_\mu(x, t)=0$ on $ \partial \Omega$ means that both the tangential part $\mathsf{t} \phi_\mu$ and the normal part $\mathsf{n}\phi_\mu$ vanish on $\partial \Omega$. Hence this is different from the more traditional {\it relative} or {\it absolute} boundary value problem for differential forms which requires $\mathsf{t} \phi=\mathsf{t}\delta \phi=0$ and $\mathsf{n} \phi=\mathsf{n}d \phi=0$ respectively. Nevertheless it is a boundary condition (which was studied in \cite{Morrey-har}) such that together with $\Delta_{\bar{\partial}}\phi=0$ it is hypo-elliptic and the Schauder estimate of \cite{Simon} applies.  To get a global solution we shall prove that there exist uniform (in terms of $\mu$) estimates so that we can extract a convergent sub-sequence. Note that $\Lambda^p \phi_\mu$ is a solution to the heat equation and $|\phi_\mu|\le C_{m, p}\Lambda^p \phi_\mu$. Let
$$
u(x, t)=\int_M H(x, y, t) |\phi_0|(y)\, d\mu(y)
$$
where $H(x, y, t)$ is the positive heat kernel of $M$. By the fundamental heat kernel estimate of Li-Yau \cite{LY}, it is easy to see that,  under the assumption (\ref{ass-1}), there exists $T_0$ such that $u(x, t)$ is finite on $K\times[0, T_0]$ for any compact subset $K$. It is easy to see that $|\phi_\mu|(x, t)\le C_{m, p}\Lambda^p \phi_\mu\le  C'_{p, m} u(x, t)$ by (\ref{basic-1}) and the maximum principle for the scalar heat equation. Now the interior Schauder estimates \cite{Simon} (see also \cite{Morrey}, Theorem 5.5.3 for the corresponding estimates in the elliptic cases) imply that for any $0< \alpha< 1$, $K$, a compact subset of $M$,
$$
\|\phi_\mu\|_{2, \alpha, \frac{\alpha}{2}, K\times [0, T_0]} \le C(K, p, m, \|\phi_\mu\|_{\infty,K\times [0, T_0]} ).
$$
Here $\|\cdot\|_{2, \alpha, \frac{\alpha}{2}}$ is the $C^{2, \alpha}$-H\"odler norm on the parabolic region.
Since $\|\phi_\mu\|_{\infty,K\times [0, T_0]}$ is estimated by $u(x, t)$ uniformly, we have established the uniform estimates so that, after passing to a subsequence, $\{\phi_\mu(x, t)\}$ converges to a solution $\phi(x, t)$ on $M\times [0, T_0]$.

It is obvious from the construction that $\phi(x, t)\ge 0$. To prove the estimate (\ref{con-sol}), appealing  Li-Yau's upper estimate
$$
H(x, y, t)\le \frac{C(n)}{V_x(\sqrt{t})}\exp\left(-\frac{r^2(x, y)}{5t}\right)
$$
we can derive that for $0\le t\le T_0\le \frac{1}{10a}$,
\begin{eqnarray*}
u(x, t)&\le& \frac{C(m)}{V_x(\sqrt{t})}\int_M \exp\left(-\frac{r^2(x, y)}{5t}+ar^2(y)\right) |\phi_0|(y)\exp\left(-ar^2(y)\right)\, d\mu(y)\\
&\le & \mathcal{B}\cdot \frac{C(m)}{V_x(\sqrt{t})}\exp\left(2a\, r^2(x)\right).
\end{eqnarray*}
In the second inequality above we used the estimate that for $0\le t\le T_0\le \frac{1}{10a}$
$$
-\frac{r^2(x, y)}{5t}+ar^2(o, y) \le -\frac{r^2(x, y)}{5t}+2a r^2(o, x)+2a r^2(x, y)\le 2a r^2(x).
$$
\end{proof}

It is clear from the proof that if $\phi_0(x)$ satisfies stronger assumption that
\begin{equation}\label{ass-2}
\int_M |\phi_0|(y)\exp( -a r^{2-\delta}(y))\, d\mu(y) < \infty
\end{equation}
for some positive constants $\delta$ and  $a$ then the Cauchy problem has a global solution on $M \times [0, \infty)$.

To deform a general $(p, p)$-form, we need the following generalization on a well-known lemma of Bishop-Goldberg concerning $(1,1)$-forms on manifolds with $\mathcal{C}_1$. This also holds the key to extending Proposition \ref{max-pp} to the noncompact manifolds.

\begin{lemma}\label{bg} Assume that $(M, g)$ satisfies $\mathcal{C}_2$. Then for any $(p, q)$-form $\phi$,
\begin{equation}\label{eq:117}
\langle \mathcal{KB}(\phi), \ov{\phi}\rangle \le 0.
\end{equation}
\end{lemma}
\begin{proof} We shall check for the $(p, p)$-forms since the argument is the same for $(p, q)$-forms. For $\phi=\frac{1}{(p!)^2}\sum \phi_{I_p, \ov{J}_p} \left(\sqrt{-1}dz^{i_1}\wedge dz^{\bar{j_1}}\right)\wedge \cdot\cdot\cdot\wedge \left(\sqrt{-1}dz^{i_p}\wedge dz^{\bar{j_p}}\right)$, where the summation is for $1\le i_1,\cdots, i_p, j_1, \cdots, j_p\le m$.
Under a normal coordinate,
$$
\langle \phi, \ov{\psi}\rangle =\frac{1}{(p!)^2} \sum \phi_{I_p, \ov{J}_p}\ov{\psi_{I_p, \ov{J}_p}}.
$$
Recall that also under the normal coordinate, $\langle \Rm( dz^k\wedge dz^{\bar{l}}), \ov{dz^s\wedge dz^{\bar{t}}}\rangle =R_{l\bar{k}s\bar{t}}$. It is easy to check that $\mathcal{C}_2$ implies that
$\langle \Rm(Z^*_1\wedge \ov{W}^*_1+Z^*_2\wedge \ov{W}^*_2), \ov{Z^*_1\wedge \ov{W}^*_1+Z^*_2\wedge \ov{W}^*_2}\rangle \ge 0$, for any $(1,0)$-forms $Z^*_1, Z^*_2, W^*_1, W^*_2$. We shall prove the claim by computing the expression under a normal coordinate.
For any fixed $I_p, J_p$ and $\mu, \nu$ with $1\le \mu, \nu \le p$ we can define
$$
\ov{\eta}_\mu \doteqdot\sum_{k_\mu=1}^m \phi_{i_1\cdots (k_\mu)_\mu\cdots i_p, \ov{J}_p} dz^{\bar{k}_\mu}, \quad \xi_\nu\doteqdot \sum_{l_\nu=1}^m \phi_{I_p, \bar{j}_1\cdots \ov{(l_\nu)_\nu}\cdots \bar{j}_p}dz^{l_\nu}.
$$
Now $\Rm$ being in $\mathcal{C}_2$ implies that
\begin{equation}\label{eq:118}
\langle \Rm( dz^{i_\mu}\wedge \ov{\eta}_\mu -\xi_\nu \wedge dz^{\bar{j}_\nu}), \ov{dz^{i_\mu}\wedge \ov{\eta}_\mu -\xi_\nu \wedge dz^{\bar{j}_\nu}} \rangle \ge 0.
\end{equation}
Now using that
\begin{eqnarray*}\Rm(dz^{i_\mu}\wedge \ov{\eta}_\mu -\xi_\nu \wedge dz^{\bar{j}_\nu})&=&\sum_{k_\mu, st}\phi_{i_1\cdots (k_\mu)_\mu\cdots i_p, \ov{J}_p} R_{k_\mu \bar{i}_\mu s\bar{t}}dz^s \wedge dz^{\bar{t}}\\
&\quad& -\sum_{l_\nu, s, t} \phi_{I_p, \bar{j}_1\cdots \ov{(l_\nu)_\nu}\cdots \bar{j}_p} R_{j_\nu \bar{l}_\nu s\bar{t}}dz^s\wedge dz^{\bar{t}}
\end{eqnarray*}
we can expand the left hand side of (\ref{eq:118}) and obtain that
\begin{eqnarray}
0&\le& \sum_{k_\mu, k'_\mu}\phi_{i_1\cdots (k_\mu)_\mu\cdots i_p, \ov{J}_p}R_{k_\mu \bar{i}_\mu i_\mu \bar{k'}_\mu}\ov{\phi_{i_1\cdots (k'_\mu)_\mu\cdots i_p, \ov{J}_p}}\nonumber \\
&-& \sum_{l_\nu, k'_\mu}\phi_{I_p, \bar{j}_1\cdots \ov{(l_\nu)_\nu}\cdots \bar{j}_p}R_{j_\nu \bar{l}_\nu i_\mu \bar{k'}_\mu}  \ov{\phi_{i_1\cdots (k'_\mu)_\mu\cdots i_p, \ov{J}_p}}\\
&-& \sum_{ k_\mu, l'_\nu}\phi_{i_1\cdots (k_\mu)_\mu\cdots i_p, \ov{J}_p}R_{k_\mu \bar{i}_\mu l'_\nu \bar{j}_\nu} \ov{ \phi_{I_p, \bar{j}_1\cdots \ov{(l'_\nu)_\nu}\cdots \bar{j}_p}}\nonumber\\
&+&\sum_{l_\nu, l'_\nu} \phi_{I_p, \bar{j}_1\cdots \ov{(l_\nu)_\nu}\cdots \bar{j}_p} R_{j_\nu \bar{l}_\nu l'_\nu \bar{j}_\nu}\ov{\phi_{I_p, \bar{j}_1\cdots \ov{(l'_\nu)_\nu}\cdots \bar{j}_p}}. \nonumber
\end{eqnarray}
Now summing for all $1\le i_1, \cdots, i_p, j_1, \cdots, j_p\le m$ and $1\le \mu, \nu\le p$, a tedious, but straight forward checking shows that the total sum of the right hand
 side above is
$-2\langle \mathcal{KB}(\phi), \ov{\phi}\rangle$.
\end{proof}

An immediate consequence of the above lemma is that {\it any harmonic $(p, q)$-form on a compact K\"ahler
manifold with $\mathcal{C}_2$  must be parallel}. This fact was known for harmonic $(p, 0)$-forms under the weaker assumption that $\Ric\ge 0$ and for harmonic $(1,1)$-forms under the nonnegativity of bisectional curvature.  In fact, using the full power of the uniformization result of Mori-Siu-Yau-Mok, the result holds even under $\mathcal{C}_1$. Hence it does not give any new information for compact K\"ahler manifolds.

Another consequence of Lemma \ref{bg} is the following result, which generalizes Lemma 2.1 of \cite{NT-jdg}, by  virtually the same argument.

\begin{corollary}\label{sub-com} Let $M^m$ be a complete K\"ahler manifold
with $\mathcal{C}_2$. Let $\phi(x, t)$ be
a $(p,p)$-form  satisfying (\ref{eq:11})  on $M\times
[0, T]$. Then $|\phi|(x,t)$ is a sub-solution of the heat equation.
\end{corollary}

This together with the proof to Proposition \ref{short-ext} gives the following improvement on the existence of the Cauchy problem for initial $(p, p$-forms not necessarily positive.

\begin{proposition}\label{short2}
Let $(M, g)$ be a K\"ahler manifold whose curvature operator $\Rm \in \mathcal{C}_2$. Assume that $\phi_0(x)$ satisfies (\ref{ass-1}), then there exists $T_0$ such that the Cauchy problem (\ref{cauchy1})
has a solution $\phi(x, t)$ on $M\times [0, T_0]$. Moreover, (\ref{con-sol}) holds.
\end{proposition}
\begin{proof} Observe that $\phi_\mu(x, t)$ in the proof of Proposition \ref{short-ext} satisfies that
$|\phi_\mu|(x, t)$ is a sub-solution to the heat equation, hence $|\phi_\mu|(x, t)\le u(x, t)$. The rest proof of Proposition \ref{short-ext} applies here.
\end{proof}

A more important application of the lemma is the following extension of  Proposition \ref{max-pp}.
This also extends Theorem 2.1 of \cite{NT-jdg}.
\begin{theorem}\label{max-pp-noncom}
Let $(M, g)$ be a complete noncompact K\"ahler manifold with $\mathcal{C}_p$. Let $\phi(x, t)$ be a $(p,p)$-form  satisfying (\ref{eq:11})  on $M\times
[0, T]$. Assume  that $\phi(x, 0)\ge0$ and satisfies (\ref{ass-1}). Assume further that for some $a>0$,
\begin{equation}\label{ass-max}
\liminf_{r\to \infty} \int_0^T \int_{B_o(r)}|\phi|^2(x, t) \exp(-a r^2(x))\, d\mu(x) dt < \infty.
\end{equation}
Then $\phi(x, t)\ge 0$. Moreover (\ref{con-sol}) holds.
\end{theorem}

Before we prove the theorem, we should remark that even though Proposition \ref{short-ext} provides a solution to the Cauchy problem which is a positive $(p, p)$-form, it is also useful to be able to assert that certain solutions, which are not constructed by Proposition \ref{short-ext}, preserve the positivity. For example, if $\phi=\sqrt{-1}\partial \bar{\partial} \eta$ and $\eta$ satisfies (\ref{eq:11}). It is easy to see that $\phi$ satisfies (\ref{eq:11}) since $\Delta_{\bar{\partial}}$ is commutable with $\partial$ and $\bar{\partial}$. If we know that $\sqrt{-1}\partial \bar{\partial} \eta\ge 0$ at $t=0$, it is desirable to know when we have $\phi(x, t)\ge 0$.

\begin{proof} We employ the localization technique of \cite{NT-jdg}. Let $\sigma_R$ be a cut-off function between $0$ and $1$ being $1$ in the annulus $A(\frac{R}{4}, 4R)=B(o, 4R)\setminus B(o, \frac{R}{4})$ and supported in the annulus $A(\frac{R}{8}, 8R)$. Let $$u_R(x, t)=\int_M H(x, y, t) |\phi|(y, 0) \sigma_R(y)\, d\mu(y)\quad \quad  u(x, t)=\int_M H(x, y, t)|\phi|(y, 0)d\mu(y).$$ Clearly $u_R(x, t)\le u(x, t)$. However the following result is proved in \cite{NT-jdg}, Lemma 2.2.

\begin{lemma}[Ni-Tam] Assume that $\phi(x, 0)$ satisfies (\ref{ass-1}).
 Then there exists $T_0>0$ depending only on
 $a$ such that for $R\ge \max\{ \sqrt{T_0}, 1\},$
  the following are true.
\begin{enumerate}
\item There exists a function $\tau=\tau(r)$ with
$\lim_{r\to \infty}\tau(r)=0$ such that for all $(x,t)\in
A_o(\frac R2, 2R)\times[0,T_0]$,
$$
u(x,t)\le u_R(x,t)+\tau(R).
$$
\item For any $r>0$,
$$
\lim_{R\to\infty}\sup_{B_o(r)\times [0,T_0]}u_R=0.
$$
\end{enumerate}
\end{lemma}

 Lemma \ref{bg} above implies that $(u_R(x, t)+\tau(R))\omega^p$ can be used as a barrier on $ \partial B_o(R)\times [0, T_0]$ since by  Corollary \ref{sub-com} and the maximum principle on complete noncompact Riemannian manifolds, that is  where the assumption (\ref{ass-max}) is needed, $|\phi|(x, t)\le u(x, t)\le u_R(x, t)+\tau(R)$ on $ \partial B_o(R)\times [0, T_0]$. In fact $$\left((u_R(x, t)+\tau(R))\omega^p+\phi\right)(v_1, \cdots, v_p; \bar{v}_1, \cdots, \bar{v}_p)\ge u_R(x, t)+\tau(R)-|\phi|(x, t)\ge 0$$ for any $(v_1, \cdots, v_p)$ which can be extended into a unitary basis of $T'_xM$.   Now apply the argument of Proposition \ref{max-pp} on $B_0(R)\times [0, T]$ we can conclude that
$$(u_R(x, t)+\tau(R))\omega^p+\phi\ge 0
$$
on $B_o(R)\times [0, T_0]$
as a $(p, p)$-form since $\phi(x, 0)\ge 0$. Now the result follows by letting $R\to \infty$ and the facts that $\lim_{R\to\infty}\sup_{B_o(r)\times [0,T_0]}u_R=0$ proved in the lemma, and $\tau(R)\to 0$. Since $|\phi|(x, t)\le u(x, t)$, the estimate (\ref{con-sol}) follows as before.
\end{proof}

We devote the rest to the proof of Theorem \ref{main-lyh} and Theorem \ref{main-KRFlyh} for the case that $M$ is noncompact complete. Since one can pick a  small $\delta>0$ and shift the time $t\to t-\delta$ and multiply the expression $Q$ by $t-\delta$, we may assume without the loss of the generality that $\phi, \partial^* \phi, \bar{\partial}^* \phi, \bar{\partial}^* \partial^* \phi$ are all smooth up to $t=0$.

 First we observe that if $\phi$ is a positive $(p, p)$-form, then $\Lambda^p\phi(x, t)$ is a nonnegative solution to the heat equation,  hence satisfies (\ref{ass-max}) by the estimate of Li and Yau.
 Precisely,
 $$
 \Lambda^p \phi(x, t)\le \Lambda^p \phi (o, 1) \cdot \frac{1}{t^{m}}\cdot \exp\left(\frac{r^2(x)}{4(1-t)}\right).
 $$
 In particular, for $\frac{\delta}{2}<T<1-\delta$, one can find $a>0$ such that
 $$
\int_M \left(\Lambda^p \phi\right)^2(x,\frac{\delta}{2})  \exp(-ar^2(x))\, d\mu(x)+\int_{\frac{\delta}{2}}^T \int_M (\Lambda^p \phi )^2 \exp(-ar^2(x))\, d\mu(x)\, dt <\infty.
 $$
Since $|\phi|\le C_{m, p} \Lambda^p \phi$ we can conclude that $|\phi|$ satisfies the above estimate. Applying Lemma \ref{bg} to $(p-1, p)$ and $(p, p-1)$ forms implies the following estimates: There exists $c_{m, p}>0$ such that
 \begin{eqnarray}
  \left(\Delta -\frac{\partial}{\partial t}\right)|\phi|^2 &\ge& c_{m, p}\left(|\partial^* \phi|^2+|\bar{\partial}^* \phi|^2\right),\label{11-help1}\\
 \left(\Delta -\frac{\partial}{\partial t}\right) \left(|\partial^* \phi|^2+|\bar{\partial}^* \phi|^2\right) &\ge& c_{m, p} |\bar{\partial}^* \partial^* \phi|^2, \label{11-help2}\\
  \left(\Delta -\frac{\partial}{\partial t}\right) |\bar{\partial}^* \partial^* \phi|^2 &\ge 0. \label{11-help3}
 \end{eqnarray}
By the same argument of the proof to Lemma 1.4 in \cite{N-jams} we can conclude that there exists $a'>0$ such that
\begin{equation}
\int_{\frac{\delta}{2}}^T \int_M \left(|\phi|^2 +|\partial^* \phi|^2+|\bar{\partial}^* \phi|^2+|\bar{\partial}^* \partial^* \phi|^2\right)\exp(-a'r^2(x))< \infty.
\end{equation}
Note that by the mean value theorem for the subsolution to the heat equation \cite{LT}, one can obtain pointwise estimates for $|\phi|^2 +|\partial^* \phi|^2+|\bar{\partial}^* \phi|^2+|\bar{\partial}^* \partial^* \phi|^2$  at $t=\delta$.
Now  with the help of the argument for the compact case, the same proof as Theorem \ref{max-pp-noncom} via the barrier argument, applying to $Q$ which is viewed a $(p-1, p-1)$-form valued Hermitian symmetric tensor,   proves Theorem \ref{main-lyh} on complete noncompact manifolds. The corresponding result with the K\"ahler-Ricci flow, namely Theorem \ref{main-KRFlyh},  is very similar. Hence we keep it brief. Due to the bound on the curvature there exists a positive constant $\alpha_{m, A}$, depending on the upper bound $A$ of the curvature  tensor so that
\begin{eqnarray}&\,&
  \left(\Delta -\frac{\partial}{\partial t}\right)\left(e^{-\alpha_{m, A} t}\cdot |\phi|^2 \right)\ge c_{m, p}e^{-\alpha_{m, A} t}\cdot \left(|\partial^* \phi|^2+|\bar{\partial}^* \phi|^2\right),\label{11-help1-1}\\
 &\, &\left(\Delta -\frac{\partial}{\partial t}\right)\left(e^{-\alpha_{m, A} t}\cdot \left(|\partial^* \phi|^2+|\bar{\partial}^* \phi|^2\right)\right) \ge c_{m, p} e^{-\alpha_{m, A} t}\cdot |\bar{\partial}^* \partial^* \phi|^2, \label{11-help2-2}\\
 &\, & \left(\Delta -\frac{\partial}{\partial t}\right)\left(e^{-\alpha_{m, A} t}\cdot |\bar{\partial}^* \partial^* \phi|^2 \right) \ge 0. \label{11-help3-3}
 \end{eqnarray}
There are modified point-wise estimates for the positive solutions coupled with the Ricci flow to replace the Li-Yau's estimate. See for example Theorem 2.7 of \cite{Ni-JDG07} (a result of Guenther \cite{Gu}). There is a corresponding  mean value theorem for the nonnegative sub-solutions to the heat equation. See for example Theorem 1.1 of \cite{N-MRL05}. Putting them together the similar argument as the above applies to Theorem \ref{main-KRFlyh}.

\begin{remark} The argument here in fact proves Theorem 1.1 of \cite{N-jams} without the assumption  (1.5) there since that assumption  is a consequence of the semi-positivity of $h$ and Li-Yau's estimate for positive solutions to the heat equation.
\end{remark}

\section{Appendix.}

Here we include  Wilking's proof to Theorem \ref{ww} (also included in \cite{CT}) following the notations of Section 3, which was explained to the first author by Wilking  in May of  2008.

Recall that $\operatorname{ad}_v:  \mathfrak{so}(n, \C)\to \mathfrak{so}(n, \C)$ mapping $w$ to $[v, w]$. The operator $\operatorname{ad}_{(\cdot)}$ can be viewed as a map from $\mathfrak{so}(n, \C)$ to the space of endmorphisms of $\mathfrak{so}(n, \C)$. It is the derivative of $\operatorname{Ad}$, the adjoint action of $\mathsf{SO}(n, \C)$ on $\mathfrak{so}(n, \C)$, which maps  $\mathsf{SO}(n, \C)$ to automorphisms of $\mathfrak{so}(n, \C)$.  This is a basic fact in  Lie group theory.  Another basic fact from Lie group theory  asserts that
$e^{\operatorname{ad}_v}=\operatorname{Ad}(\operatorname{Exp}(v))$.

For the proof of Wilking's theorem we need to recall  the following  identity for $\Rm^\#$.
\begin{equation}\label{A1}
\langle \Rm^{\#} (v), w\rangle =\frac{1}{2}\sum_{\alpha, \beta}\langle [ \Rm(b^\alpha), \Rm(b^\beta)], v\rangle \langle [b^\alpha, b^\beta], w\rangle.
\end{equation}
Here $\{b^\alpha\}$ is a basis for $\frak{so}(n, \C)$. It suffices to show that if $\langle \Rm(v_0), \ov{v}_0\rangle =0$, $\langle \Rm^{\#}(v_0), \ov{v}_0\rangle\ge 0$. Here we identify $\wedge^2(\C^n)$ with $\mathfrak{so}(n, \C)$ and and observe that the action of  $\mathsf{SO}(n, \C)$ on $\wedge^2 (\C^n)$ is the same as the adjoint action under the identification.

Given above basic facts from Lie group theory, for any $b\in \mathfrak{so}(n, \C)$, and $z\in \C$, consider $\left(\operatorname{Ad}(\operatorname{Exp}(zb))\right)(v_0)=\operatorname{Exp}(zb) \cdot v_0 \cdot \operatorname{Exp}(-zb)$. Since $\operatorname{Exp}(zb)\in \mathsf{SO}(n, \C)$, we conclude that $\left(\operatorname{Ad}(\operatorname{Exp}(zb))\right)(v_0)\in \Sigma$. Hence $v(z)=e^{z \, \operatorname{ad}_b}\cdot v_0\in \Sigma$. Thus if we define
$$
I(z):=\langle \Rm(v(z)), \overline{v(z)}\rangle
$$
it is clear that $I(z)\ge 0$ and $I(0)=0$, which implies that  $\frac{\partial^2}{\partial z\partial \overline{z}} I(z)|_{z=0} \ge 0$. Hence  for any $b\in \mathfrak{so}(n, \C)$,
$\langle \Rm ( \operatorname{ad}_b(v_0)), \overline{\operatorname{ad}_b(v_0)}\rangle \ge 0 $, which can be equivalently written as
\begin{equation}\label{A2}
\langle \Rm \cdot \operatorname{ad}_{v_0}(b), \operatorname{ad}_{\overline{v}_0} (\overline{b})\rangle \ge 0.
\end{equation}
This is equivalent to $-\operatorname{ad}_{\overline{v}_0}\cdot \Rm \cdot \operatorname{ad}_{v_0}\ge 0$, as a Hermitian symmetric tensor.

By equation (\ref{A1})  $ \langle \Rm^{\#}(v_0), \ov{v}_0\rangle \ge 0$  is the same as
$\frac{1}{2}\operatorname{trace}(-\operatorname{ad}_{\overline{v}_0} \cdot \Rm \cdot \operatorname{ad}_{v_0} \cdot \Rm )\ge 0$. This last fact is implied by (\ref{A2}) as follows. Let $\lambda_{\alpha}$ be the eigenvalues of $-\operatorname{ad}_{\overline{v}_0}\cdot \Rm \cdot \operatorname{ad}_{v_0}$ with eigenvectors $b^{\alpha}$. Then for  $\lambda_\alpha>0$,
$
b^{\alpha}=\frac{1}{\lambda_\alpha} \operatorname{ad}_{\overline{v}_0}(w^{\alpha}),
$
where $w^{\alpha}=\Rm \cdot \operatorname{ad}_{v_0}(b^\alpha)$. At the mean time
\begin{eqnarray*}
\operatorname{trace}(-\operatorname{ad}_{\overline{v}_0} \cdot \Rm \cdot \operatorname{ad}_{v_0}
 \cdot \Rm )&=&\sum \langle \Rm \cdot (-\operatorname{ad}_{\overline{v}_0} \cdot \Rm \cdot \operatorname{ad}_{v_0}) (b^\alpha), \overline{b^{\alpha}}\rangle  \\
&=&\sum_{\lambda_\alpha >0}\lambda_\alpha \langle \Rm (b^{\alpha} ), \overline{b^{\alpha}}\rangle\cr
&=&\sum_{\lambda_\alpha >0}\frac{1}{\lambda_\alpha} \langle \Rm \cdot \operatorname{ad}_{\overline{v}_0}(w^{\alpha}), \operatorname{ad}_{v_0} (\overline{w^\alpha})\rangle\\
&=&\sum_{\lambda_\alpha >0}\frac{1}{\lambda_\alpha} \langle \Rm\cdot  \operatorname{ad}_{v_0} (\overline{w^\alpha}), \operatorname{ad}_{\overline{v}_0}(w^{\alpha}) \rangle.
\end{eqnarray*}
The last expression  is nonnegative by (\ref{A2}).

\section*{Acknowledgments.} {  } We  thank Brett Kotschwar for bringing our attention to \cite{CT}. We also thank Nolan Wallach for helpful discussions, Shing-Tung Yau for his interests. The first author's research is partially supported by a NSF grant DMS-0805834.

\bibliographystyle{amsalpha}

{\sc  Addresses:}

{\sc  Lei Ni},
 Department of Mathematics, University of California at San Diego, La Jolla, CA 92093, USA


email: lni@math.ucsd.edu

{\sc Yanyan Niu}, Department of Mathematics  and  Institute of Mathematics and
Interdisciplinary Science, Capital Normal University, Beijing, China

email: yyniukxe@gmail.com

\end{document}